\documentclass[a4paper, 10pt]{article}

\usepackage[T1]{fontenc}				% codifica dei font
\usepackage[utf8]{inputenc}			% codifica di input
\usepackage[english]{babel}

\usepackage{amsmath}
\usepackage{amssymb}
\usepackage{asymptote}
\usepackage{amsthm}
	\theoremstyle{plain}
		\newtheorem{thm}{Theorem}[section]	% numbered within each section
		\newtheorem{lem}[thm]{Lemma}		% numbered along with Theorem
		\newtheorem{prop}[thm]{Proposition}
		\newtheorem*{teo}{Theorem}

	\theoremstyle{definition}
		\newtheorem{defn}[thm]{Definition}	% numbered along with Theorem
		\newtheorem{ex}[thm]{Example}		% numbered along with Theorem

	\theoremstyle{remark}
		\newtheorem{rmk}[thm]{Remark}		% numbered along with Theorem
\numberwithin{equation}{section}	% equation numbering

\usepackage{mathtools}		% per i simboli := e =:
\usepackage{mathrsfs}		% per il math script
\usepackage{eucal}			% per il mathcal + figo

\usepackage{braket}			% per il comando \Set: definizione automatica di insiemi
\usepackage[all,pdf]{xy}		% per i diagrammi commutativi

\usepackage{mparhack}		% fix bug for \marginpar
\usepackage{fixltx2e}		% improve and fix bugs

\usepackage{a4wide}		% aumenta i margini di pagina

\usepackage{booktabs}		% per fare bene le tabelle
\usepackage{multirow}		% per celle su pi righe nelle tabelle
\usepackage{caption}		% per le didascalie di tabelle e figure
\usepackage{rotating}
\usepackage{subfig}

\usepackage{varioref}		% per indicazione di pagina nei riferimenti incrociati
\usepackage{footmisc}

\usepackage{graphicx}		% per le figure
\usepackage{enumerate}		% per impostare pi facilmente gli elenchi numerati
\usepackage[colorlinks=true, linkcolor=black, citecolor=black, urlcolor=blue]{hyperref}		% per attivare i collegamenti ipertestuali

\newcommand{\C}{\mathbb{C}}		% the complex plane

\renewcommand{\H}{\mathscr{H}}
\renewcommand{\L}{\mathscr{L}}	% Lagrangiana
\newcommand{\N}{\mathbb{N}}		% the natural numbers
\renewcommand{\P}{\mathscr{P}}
\newcommand{\R}{\mathbb{R}}		% the real line
\newcommand{\U}{\mathbb{U}}		% gruppo unitario
\newcommand{\Z}{\mathbb{Z}}		% the integer numbers

\newcommand{\GL}{\mathrm{GL}}
\newcommand{\SO}{\mathrm{SO}}
\newcommand{\Sp}{\mathrm{Sp}}

\newcommand{\Mat}{\mathrm{Mat}}

\newcommand{\Gr}{\mathrm{Gr}}

\let\d\relax
\newcommand{\d}{\mathrm{d}}				% differenziale (esterno)
\newcommand{\eps}{\varepsilon}
\newcommand{\norm}[1]{\left\lVert #1 \right\rVert}			% norma
\newcommand{\abs}[1]{\left\lvert #1 \right\rvert}				% valore assoluto
\newcommand{\trasp}[1]{{#1}^\mathsf{T}}					% trasposto
\newcommand{\traspinv}[1]{{#1}^\mathsf{-T}}				% trasposto inverso
\newcommand{\Uhom}{U_{\alpha}}
\newcommand{\Ulog}{U_{\log}}
\newcommand{\iMor}{i_{\textup{Morse}}}		% indice di Morse
\newcommand{\igeo}{i_{\textup{geo}}}		% indice geometrico
\newcommand{\igeomega}{i_{\textup{geo,$\omega$}}}		% omega-indice geometrico
\newcommand{\Lag}{\mathrm{Lag}}
\newcommand{\iclm}{i_{\textup{CLM}}}		% indice di Cappell-Lee-Miller

\DeclareMathOperator{\diag}{diag}		% diagonal matrix
\DeclareMathOperator{\sgn}{sgn}		% signature
\DeclareMathOperator{\tr}{tr}			% trace
\renewcommand{\leq}{\leqslant}
\renewcommand{\geq}{\geqslant}
\renewcommand{\hat}{\widehat}
\renewcommand{\tilde}{\widetilde}
\renewcommand{\=}{\coloneqq}			% definisce :=
\newcommand{\eq}{\eqqcolon}			% definisce =:

\newcommand{\ie}{i.e.~}

\newcommand{\email}[1]{\href{mailto:#1}{\textsf{#1}}}

\title{Morse index and linear stability of the Lagrangian circular orbit in a three-body-type problem via index theory}
\author{Vivina Barutello, Riccardo D.~Jadanza, Alessandro Portaluri\thanks{The authors are partially supported by the project ERC Advanced Grant 2013 n.~339958 ``Complex Patterns for Strongly Interacting Dynamical Systems --- COMPAT''.} }
\date{\today}

\begin{document}

	\maketitle

	\begin{abstract}
		It is well known that the linear stability of the Lagrangian elliptic solutions in the classical planar three-body problem depends on a mass parameter $\beta$ and on the eccentricity $e$
		of the orbit.
		We consider only the circular case ($e = 0$) but under the action of a broader family of singular potentials: $\alpha$-homogeneous potentials, for $\alpha \in (0,2)$, and
		the logarithmic one. It turns out indeed that the Lagrangian circular orbit persists also in this more general setting.
		
		We discover a region of linear stability expressed in terms of the homogeneity parameter $\alpha$ and the mass parameter $\beta$, then we compute the Morse index of this orbit and of
		its iterates and we find that the boundary of the stability region is the envelope of a family of curves on which the Morse indices of the iterates jump.
		In order to conduct our analysis we rely on a Maslov-type index theory devised and developed by Y.~Long, X.~Hu and S.~Sun; a key role is played by an appropriate index
		theorem and by some precise computations of suitable Maslov-type indices.
	
		\bigskip
		
	\end{abstract}


	\section*{Introduction and main results}

	We consider a planar three-body-type problem governed by a \emph{singular potential function} $U : X \subset \R^6 \to \R$,
	where $X \= \Set{ q = (q_1,q_2,q_3) \in \R^6 | q_i \neq q_j\ \forall\, i \neq j }$. We shall deal with homogeneous and logarithmic potentials of the form
		\begin{equation}\label{eq:potentials}
				\Uhom(q) \= \sum_{\substack{i, j = 1\\ i<j}}^3 \frac{m_i m_j} {\abs{q_i - q_j}^{\alpha}}, \quad \alpha \in (0, 2); 
				\qquad
				\Ulog(q) \= \sum_{\substack{i, j = 1\\ i<j}}^3 m_i m_j \log\dfrac{1}{\abs{q_i - q_j}}.
		\end{equation}
	Newton's equations for this problem (which as $U=U_\alpha$ is commonly known as the \emph{generalised $3$-body problem}) are
		\begin{equation}\label{eq:Newtonintro}
			m_i \ddot q_i = \frac{\partial U}{\partial q_i}
		\end{equation}
	and we seek solutions that satisfy periodic boundary conditions. By taking into account the conservation law of the centre of mass we see that the \emph{configuration space} is
	$4$-dimensional and is given by
		\[
			\hat{X} \= \Set{ q \in \R^6 | \sum_{i=1}^3 m_i q_i = 0,\ q_i \neq q_j\ \forall\, i \neq j }.
		\]
	Let $(q, v)$ be an element of the tangent bundle $T\hat{X}$, so that $q \in \hat{X}$ and $v \in T_q \hat{X}$. The \emph{Lagrangian function} $\L \in \mathscr{C}^{\infty}(T\hat{X}, \R)$ is
	given by
		\begin{equation}\label{eq:Lagrangian}
			\L(q, v) \= \frac{1}{2} \sum_{i = 1}^3 m_i \abs{v_i}^2 + U(q),
		\end{equation}
	Let $ W^{1,2}(\R/2\pi\Z, \hat{X})$ be the Sobolev space of $L^2$-loops with weak $L^2$-derivatives and define on it
	the \emph{Lagrangian action functional} $\mathbb{A}$ as
		\begin{equation} \label{eq:action}
			\mathbb{A}(\gamma) \= \int_0^{2\pi} \L\bigl( \gamma(t), \dot{\gamma}(t) \bigr)\, dt,
		\end{equation}
	which is smooth on its domain, since it consists of collisionless loops. Its critical points in this space are the $2\pi$-periodic (classical) solutions of Equations~\eqref{eq:Newtonintro}.
	
	The first solutions of the classical ($U=U_\alpha$ with $\alpha =1$) planar three-body problem have been shown in 1772 by J.-L.~Lagrange \cite{Lagrange:3corps}: for any choice of the
	three masses there exists a family of periodic motions during which the bodies are always arranged in an equilateral triangle that rotates around its barycentre, changing its size but not its
	shape; moreover, each particle describes a Keplerian conic.
	In the special case where the trajectory of each body around the centre of mass is a circle swept with some appropriate angular frequency, Lagrange's triangular solution is an example of
	\emph{relative equilibrium}, called \emph{Lagrange circular orbit}. We observe that this kind of circular motion is maintained also in the case of the more general potentials defined
	in~\eqref{eq:potentials}.
	
	Given a periodic solution of \eqref{eq:Newtonintro}, it is natural to investigate
	its stability properties in order to understand the dynamical behaviour of the orbits nearby. Linear stability of periodic 
	orbits is a paradigm of a complex behaviour of a dynamical system: positive topological entropy, existence of 
	transversal heteroclinic connections and KAM tori, presence of horseshoes.
	Our main concern is the linear stability of these circular Lagrangian solutions. It turns out that it depends on two parameters: the \emph{mass parameter}
		\[
			\beta \= 27\,\dfrac{m_1m_2+m_2m_3+m_1 m_3}{(m_1+m_2+m_3)^2}\in (0,9]
		\]
	and the homogeneity parameter $\alpha \in [0,2)$. Note that we now include the value $\alpha = 0$ because it will be shown that this corresponds to the logarithmic case. These two
	parameters define a family of Lagrangian circular solution, which we denote by $\gamma_{\alpha, \beta}$.
	
	In order to investigate the linear stability of this family we need to reformulate the Newtonian problem \eqref{eq:Newtonintro} in Hamiltonian language.
%		\begin{equation}\label{eq:hamiltonianintro}
%			\begin{cases}
%				\dot z(t) = J \nabla \H(z)\\
%				z(0)=z(2\pi),
%			\end{cases}
%		\end{equation}
%	where $J \=\begin{pmatrix} 0& -I_6\\ I_6 & 0 \end{pmatrix}$ and $z \= \trasp{(\trasp{p}, q)} \in T^*X$, and set the \emph{Hamiltonian function}
%		\[
%			\H(p,q) \= \sum_{i=1}^3 \frac{\abs{p_i}^2}{2m_i} - U(q).
%		\]
	A $2\pi$-periodic solution of this autonomous Hamiltonian system is \emph{spectrally stable} if the spectrum of the monodromy matrix of the corresponding linearised system is contained in
	the unit circle of the complex plane; it is \emph{linearly stable} if in addition such matrix is diagonalisable.
	
	\begin{figure}
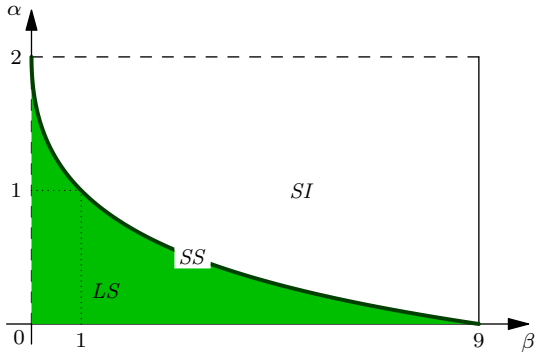

		\centering
			\begin{asy}
				import graph;
		
				size(200, 200*2/3, IgnoreAspect);
				
				real x1(real t) {return 9*(t - 2)^2/(t + 2)^2;}		// Curva di stabilità
				real y(real t) {return t;}					// Serve solo per la parametrizzazione
		
			/// 	Assi coordinati
				xaxis(xmin = -0.5, xmax = 10, arrow=EndArrow);
				yaxis(ymin = 0, ymax=2, dashed);
				
				path ss = buildcycle(graph(x1, y, 0, 2, operator ..), (0,2)--(0,0)--(9,0));
				fill(ss, heavygreen);
				
			///	Disegno curve
				yequals(1, xmin=0, xmax=1, dotted);
				xequals(1, ymin=0, ymax=1, dotted);
				yequals(2, xmin=0, xmax=9, dashed);
				xequals(9, ymin=0, ymax=2);
				draw(graph(x1, y, 0, 2, operator ..), darkgreen+linewidth(1.5));
				draw((0,-0.15)--(0,0));
				draw((0,2)--(0,2.35), arrow = EndArrow);
		
			///	Etichette
				label("$\scriptstyle 0$", (0, 0), SW);
				label("$\scriptstyle 1$", (1, 0), S);
				label("$\scriptstyle 9$", (9, 0), S);
				label("$\scriptstyle \beta$", (10, 0), S);
				label("$\scriptstyle 1$", (0, 1), W);
				label("$\scriptstyle 2$", (0, 2), W);
				label("$\scriptstyle \alpha$", (0, 2.35), W);
				
				label("$\scriptstyle \mathit{SI}$", (5,1), E);
				label("$\scriptstyle \mathit{SS}$", (3.24,0.5), UnFill);
				label("$\scriptstyle \mathit{LS}$", (1.5,0.25));
			\end{asy}
			\caption{Stability regions: spectral instability ($\mathit{SI}$), spectral stability ($\mathit{SS}$) and linear stability ($\mathit{LS}$).
				The curve drawn is the \emph{stability curve} $\beta = 9 \big( \frac{\alpha - 2}{\alpha + 2} \big)^2$, which marks the transition between stability and instability.
				For a fixed value of $\alpha$ the motion becomes linearly stable if $\beta$ is small enough, \ie if there is a dominant mass.
				For a fixed value of $\beta$, linear stability is achieved if $\alpha$ is small enough, \ie if the potential is sufficiently weak.} \label{fig:stab}
	\end{figure}
	
	When analysing $\gamma_{\alpha, \beta}$ we face a very degenerate situation because of the invariance of $n$-body-type problems under the symmetry group of
	Euclidean transformations and the presence of first integrals. It is possible, through a wise change of coordinates originally found by Meyer and Schmidt and here modified, to factorise the
	contributions of these constants of motion and split the phase space into a direct sum of invariant 4-dimensional symplectic subspaces: $T^*X = E_1 \oplus E_2 \oplus E_3$
	(see~\cite{Moeckel:notes, MR2145251, MR3227283}). It turns out that the degeneracy is confined in $E_1$ and partly in $E_2$, whereas $E_3$ contains the essence of the dynamics.	
	More precisely the subspace $E_1$ corresponds to the four integrals of motion of the center of mass, whilst $E_2$ includes the conservation of the angular momentum.
	Furthermore, the restriction of the Hamiltonian to the symplectic invariant subspace $E_2$ of the phase space agrees with the Hamiltonian of a generalised Kepler problem (\ie a Kepler
	problem with potential $U$ of the form~\eqref{eq:potentials}). We note that the eigenvalues of the monodromy matrix restricted to $E_2$ are $1, 1, e^{\pm 2\pi i\sqrt{2-\alpha}}$; hence,
	for any $\alpha \in (0,2)$, the circular solutions of the Kepler-type problem (corresponding to the line $\beta = 0$ in Figure~\ref{fig:stab}) are spectrally stable and $\alpha=2$ is the boundary
	of their stability region (which is also called in the literature \emph{elliptic region}).
	The portrait of the stability properties of $\gamma_{\alpha, \beta}$, which takes into account the essence of the dynamics, taking place on $E_3$, is depicted in Figure~\ref{fig:stab}, where one
	can neatly distinguish three regions: that of spectral instability ($\mathit{SI}$), that of linear stability ($\mathit{LS}$) and the curve of spectral stability ($\mathit{SS}$) that separates them.	
	
	A very deep and intriguing question is the relation between the linear stability of a periodic solution or of a closed geodesic and the Morse index of its iterations \cite{MR0090730}: a famous
	result by H.~Poincaré states that every closed minimising geodesic on a Riemannian surface is unstable. Motivated by this question we computed the Morse index of the Lagrangian circular
	orbit in the free loop space of $\hat{X}$. Very few results are known about this topic; a classical one is due to W.~B.~Gordon \cite{MR0502484}, who proved that the minimisers of
	the Lagrangian action functional for the Kepler problem on the subspace of $W^{1,2}(\R/2\pi\Z, \R^2 \setminus \{0\})$ of loops with winding number $\pm 1$ with respect to the 
	origin are the ellipses.
	{S.~Zhang and Q.~Zhou \cite{MR1852963} and A.~Venturelli \cite{MR1841900} proved in 2001 that the Lagrangian equilateral triangle solutions of the $3$-body problem are minimisers of the
	corresponding action functional with $\alpha = 1$.}
	However M.~Ramos and S.~Terracini showed in \cite{MR1329405} a sort of double variational characterisation of the set of all periodic solutions of the
	$\alpha$\nobreakdash-homogeneous Kepler problem; this can give a heuristic explanation of the degeneracy occurring at $\alpha=1$. In \cite{Venturelli:phd} Venturelli proved that for
	$\alpha \in (1,2)$ and winding numbers $\pm 1$ the minimisers are precisely the circular solutions, whilst for $\alpha \in (0,1)$ the minima are attained by the ejection\nobreakdash-collision
	solutions. He left, however, completely open the problem of computing the Morse index of the circular solutions in the case $\alpha \in (0,1)$.
	
	Our first main result concerns the computation of the Morse index of the circular solution of Kepler-type problems (we write $\gamma_{\alpha,0}$ for the Keplerian trajectory, in view of the
	formal correspondence with the case $\beta = 0$). As already observed, this means to compute the Morse index of the restriction of $\gamma_{\alpha,\beta}$ to the subspace $E_2$
	(see Figure~\ref{fig:MaslovE2}). Note that this quantity does not depend on $\beta$; however, we represent its values in the plane $(\beta, \alpha)$ in order to relate them more clearly with the
	restriction of the system to $E_3$: the Morse index of the original problem is indeed given by the sum of the indices of the restrictions and it is easy to visualise this with the superposition of
	the graphs.
	
	\begin{teo}
		The Morse index of the circular solution ($\gamma_{\alpha, 0}$) of the generalised Kepler problem is
			\[
				\iMor(\gamma_{\alpha, 0}) = \begin{cases}
							0 & \text{if } \alpha \in [1, 2) \\
							2 & \text{if } \alpha \in [0, 1).
						\end{cases}
			\]
	\end{teo}
	
	\begin{figure}[tb]
	\centering
	\subfloat[][Values of $\iMor(\gamma_{\alpha, 0})$. On the line $\alpha = 1$ it is equal to $0$.\label{fig:MaslovE2}]{\includegraphics[width=0.45\textwidth]{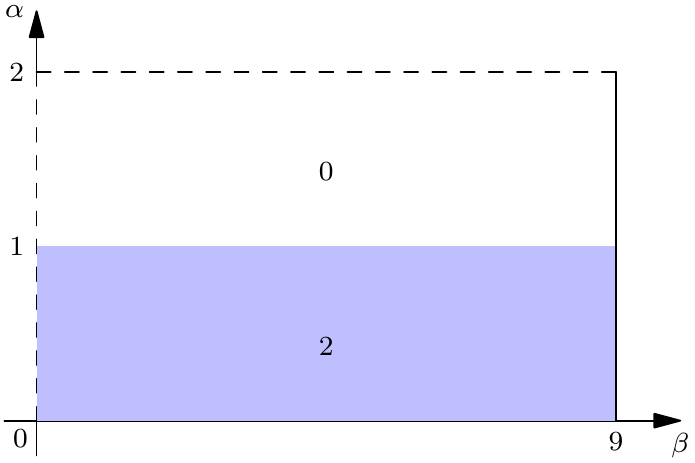}} \qquad
	\subfloat[][Values of $\iMor(\gamma_{\alpha, \beta})$. For a fixed $\alpha$, the Morse index is a monotone decreasing function of $\beta$ that attains the minimum value in
					correspondence of equal masses. The dotted curve is the stability curve.\label{fig:iMor}]{\includegraphics[width=0.45\textwidth]{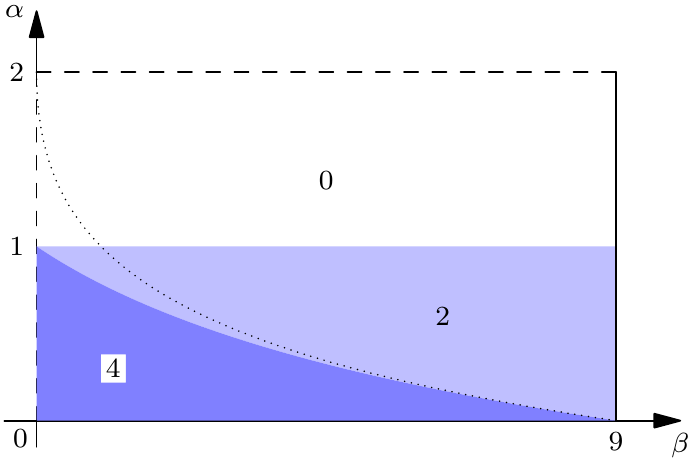}}
	\caption{Values of the Morse index of the generalised Kepler problem (a) and of the Lagrangian circular solution (b).} \label{fig:iMorse}
	\end{figure}
	
	We can then go further by computing the Morse index of any $k$-th iteration $\gamma^k_{\alpha,0}$ of $\gamma_{\alpha,0}$ for $k \in \N$, $k>1$; this is made possible by the
	$\omega$-index theory and the Bott-Long iteration formula. What we obtain is that $\iMor(\gamma^k_{\alpha, 0})$ is a piecewise constant and non-increasing function of $\alpha$ for every
	fixed $k > 1$. In particular, for any fixed $k$, there exists an interval $\big( 2 - \frac{1}{k^2}, 2 \big)$ on which $\iMor(\gamma^k_{\alpha, 0})=0$.
	On the other hand, for any fixed value of $\alpha$, the quantity $\iMor(\gamma^k_{\alpha, 0})$ diverges to $+\infty$ as $k \to +\infty$. Let us observe that $\alpha_k \= 2 - \frac{1}{k^2}$ tends to
	the value 2 as $k$ diverges: this means that the jumps of the Morse index tend to the boundary of the stability region for the Kepler-type problem. See Figure~\ref{fig:iterates} for some
	examples.
	
	\begin{figure}[p]
	\centering
	\subfloat[][$k=1$.]{\includegraphics[width=0.45\textwidth]{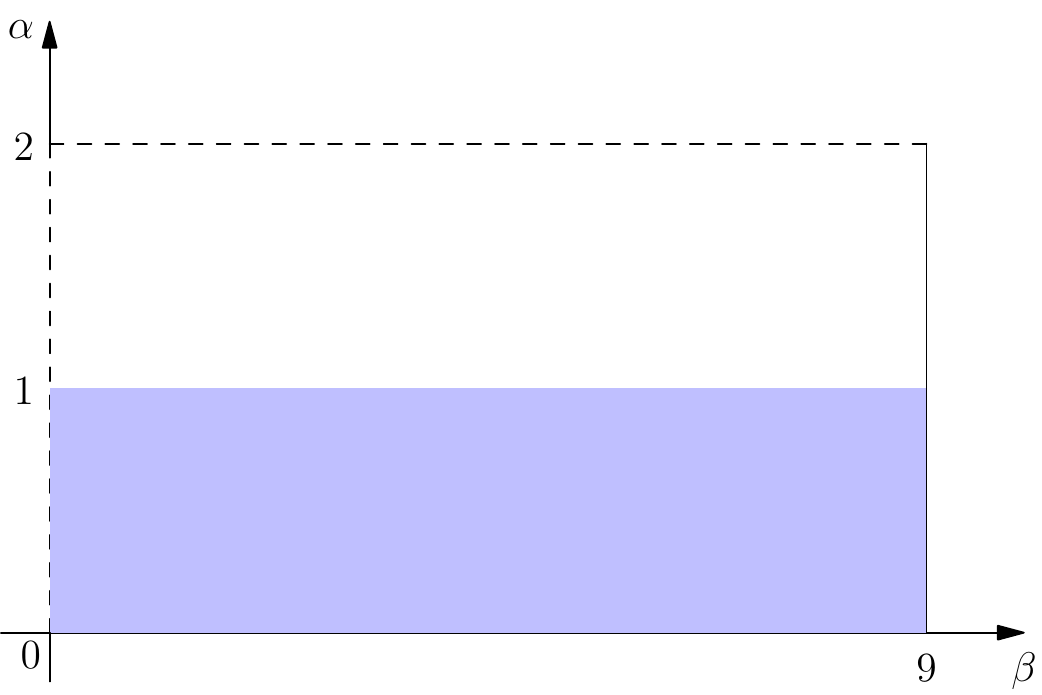}} \quad
	\subfloat[][$k=2$.]{\includegraphics[width=0.45\textwidth]{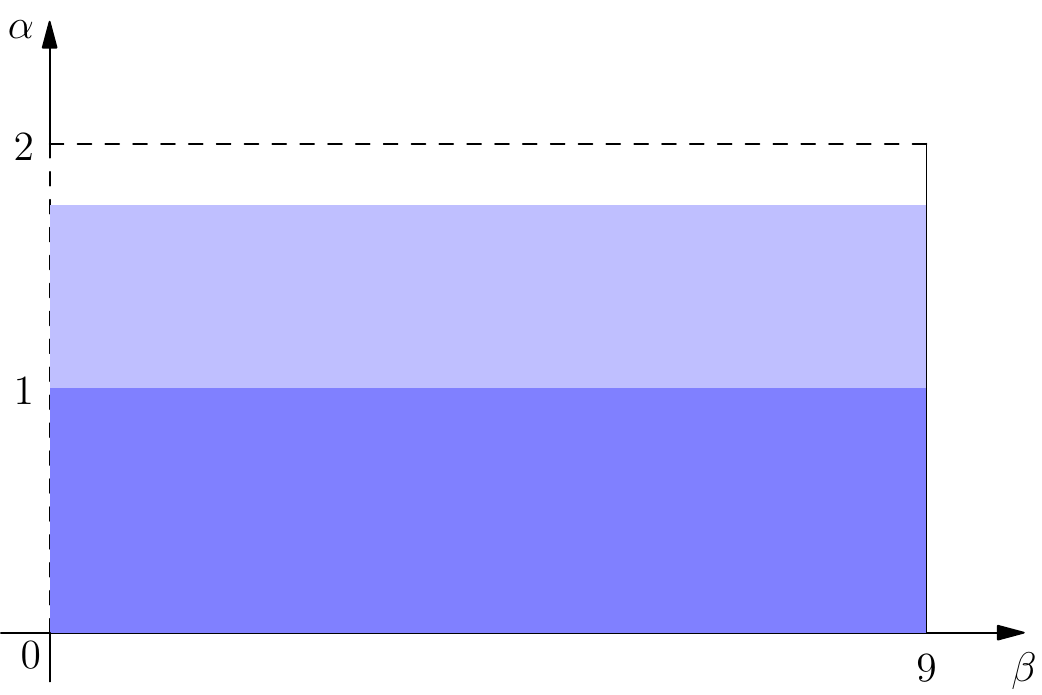}} \\
	\subfloat[][$k=3$.]{\includegraphics[width=0.45\textwidth]{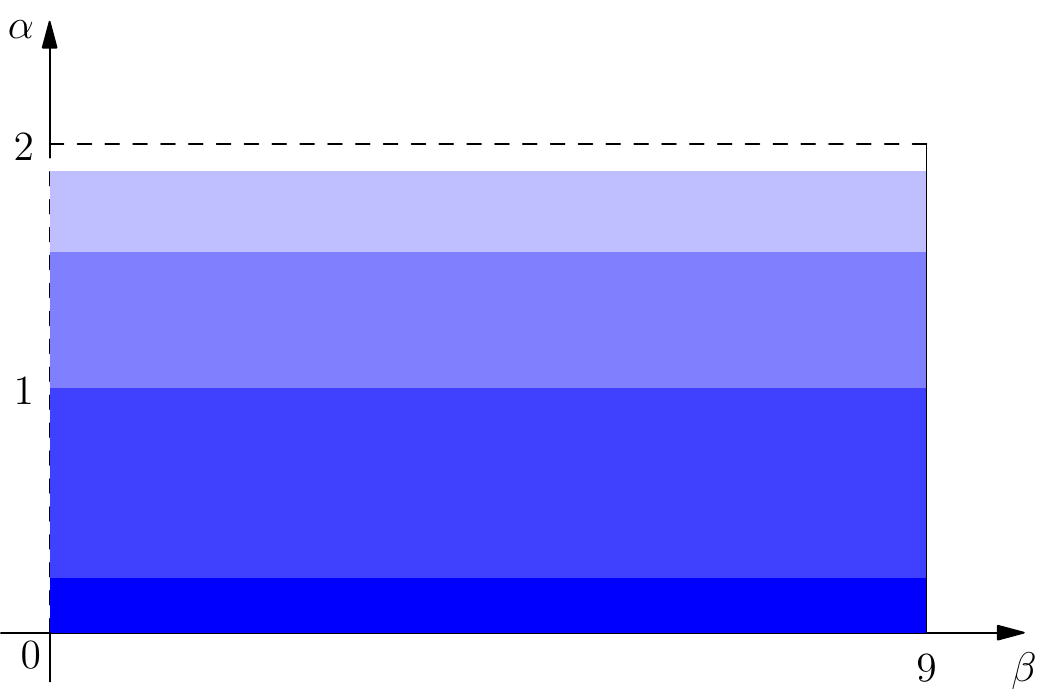}} \quad
	\subfloat[][$k=4$.]{\includegraphics[width=0.45\textwidth]{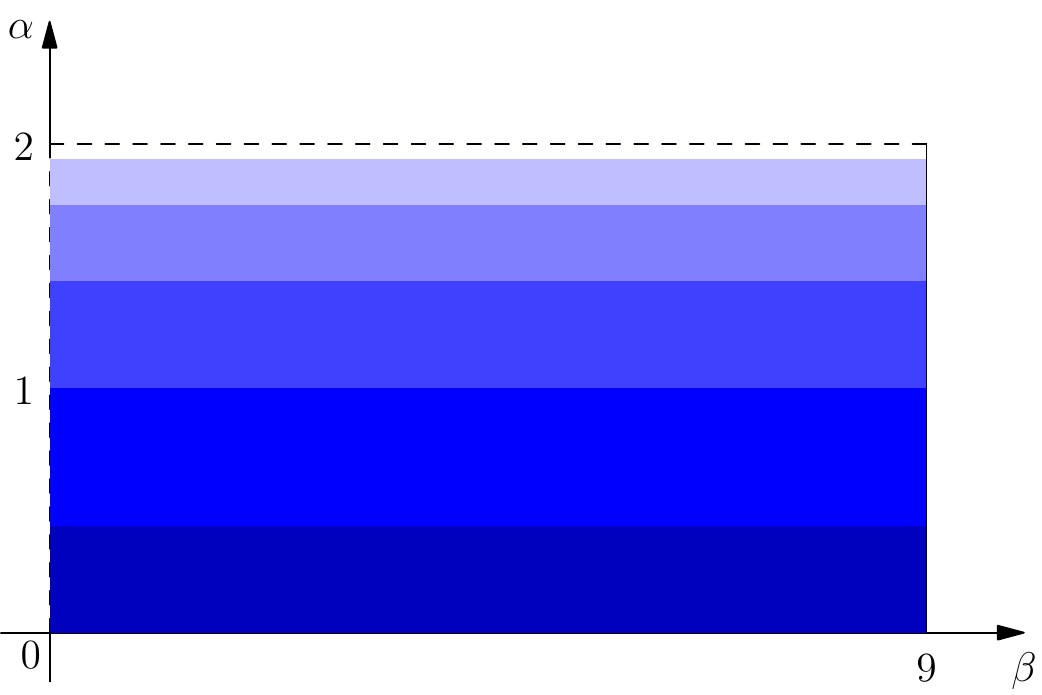}} \\
	\subfloat[][$k=6$.]{\includegraphics[width=0.45\textwidth]{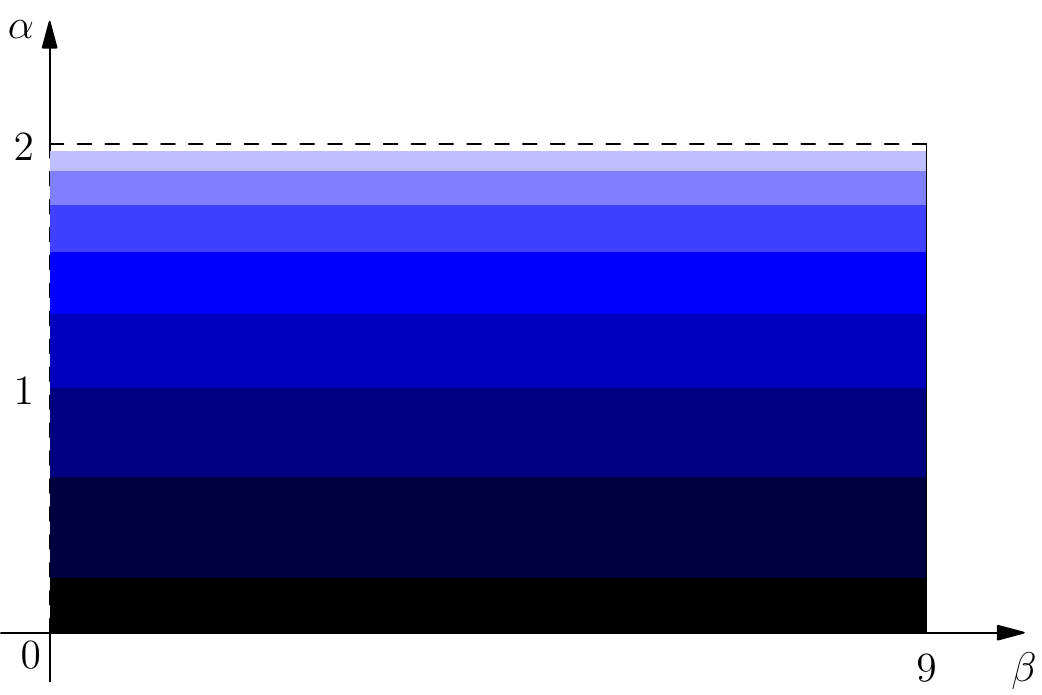}} \quad
	\subfloat[][$k=8$.]{\includegraphics[width=0.45\textwidth]{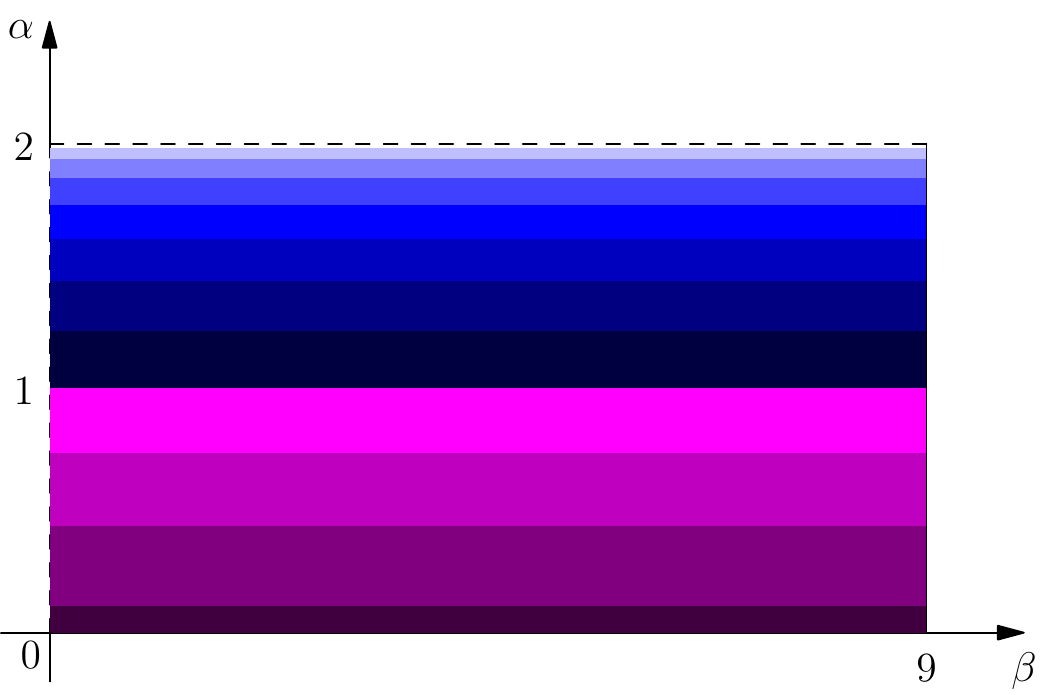}}
	\caption{Values of the Morse index of the $k$-th iteration of the Kepler circular orbit $\gamma^k_{\alpha,0}$ for some values of $k$. The white upper band in each subfigure represents
			the value $0$; going downwards and passing through the lower boundary of each band increases the Morse index by $2$. As $k$ increases there is an accumulation of bands
			at the value $\alpha = 2$.} \label{fig:iterates}
	\end{figure}
	
	\begin{figure}[p]
	\centering
	\subfloat[][$k=1$.]{\includegraphics[width=0.45\textwidth]{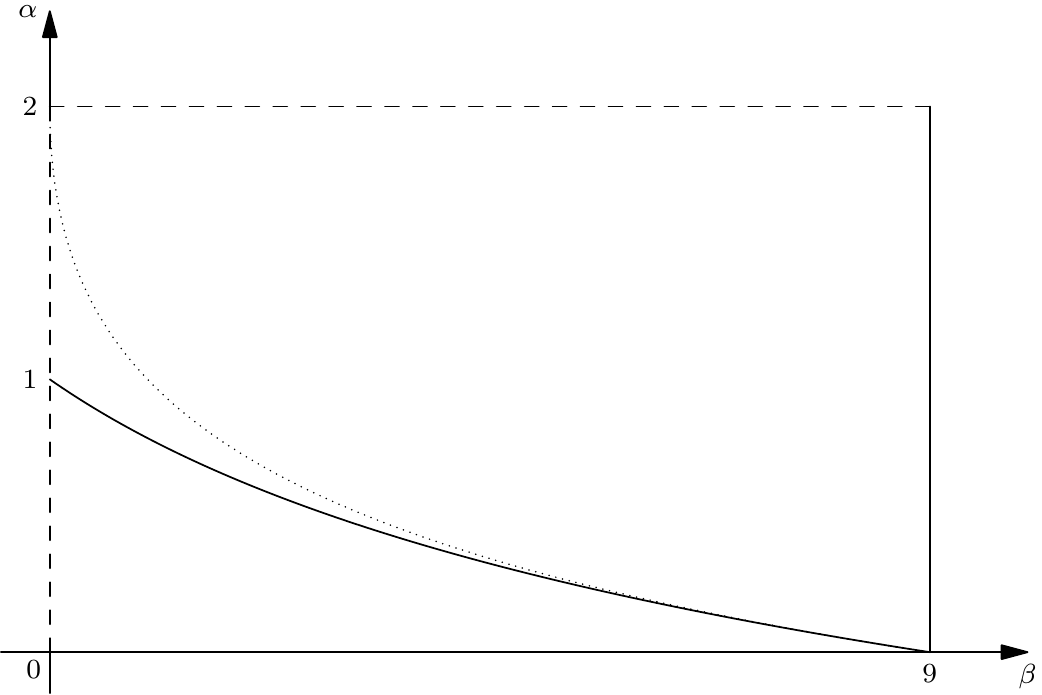}} \quad
	\subfloat[][$k=2$.]{\includegraphics[width=0.45\textwidth]{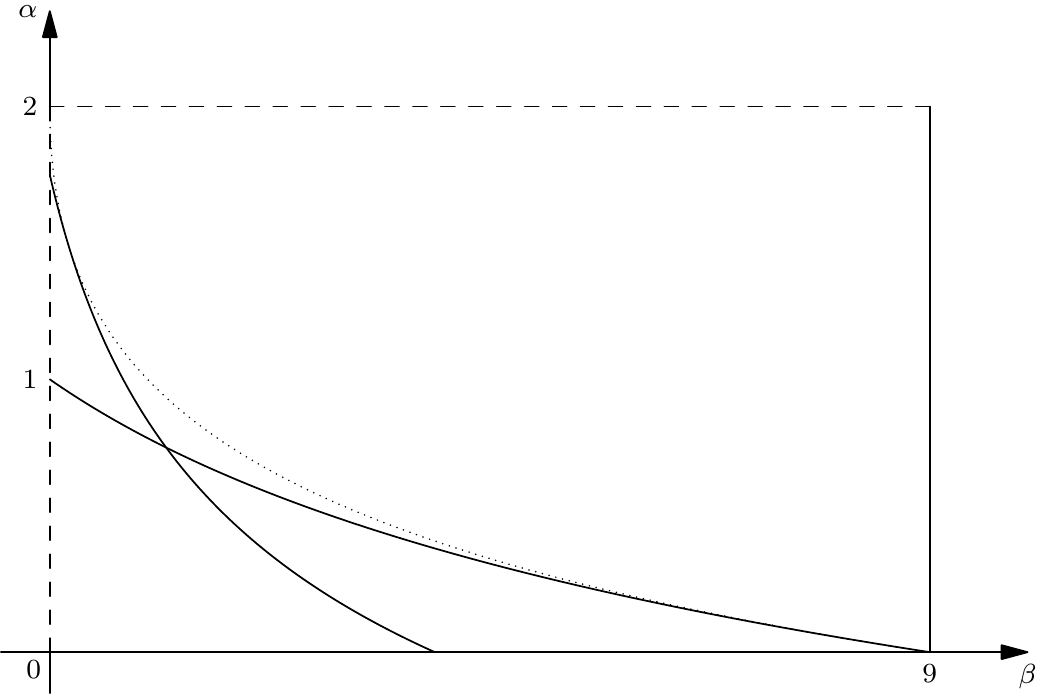}} \\
	\subfloat[][$k=3$.]{\includegraphics[width=0.45\textwidth]{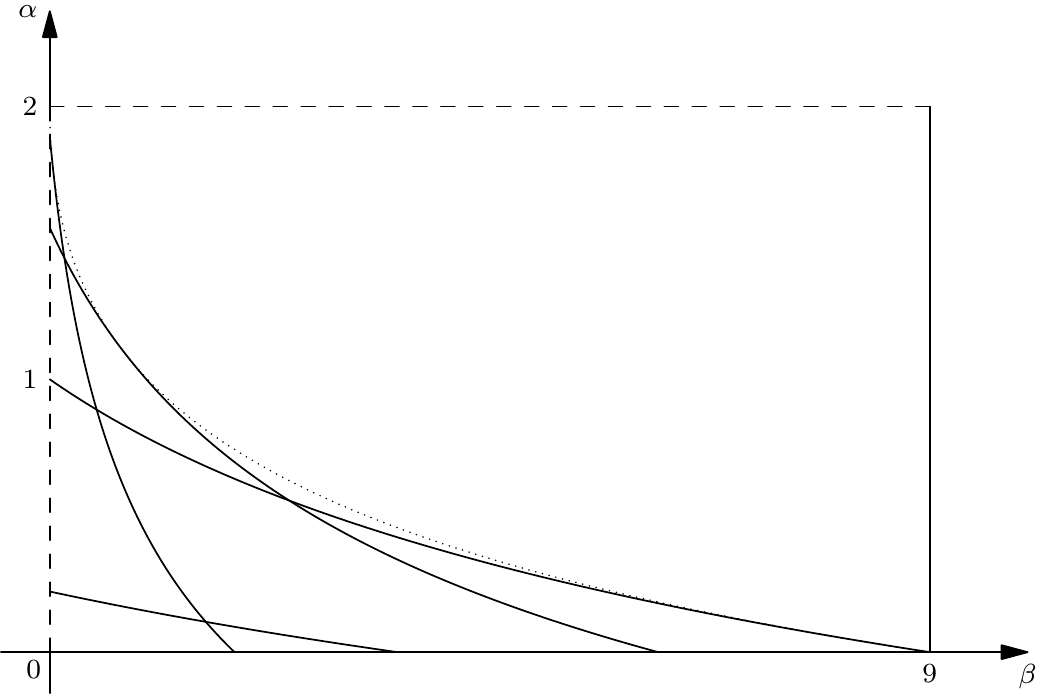}} \quad
	\subfloat[][$k=4$.]{\includegraphics[width=0.45\textwidth]{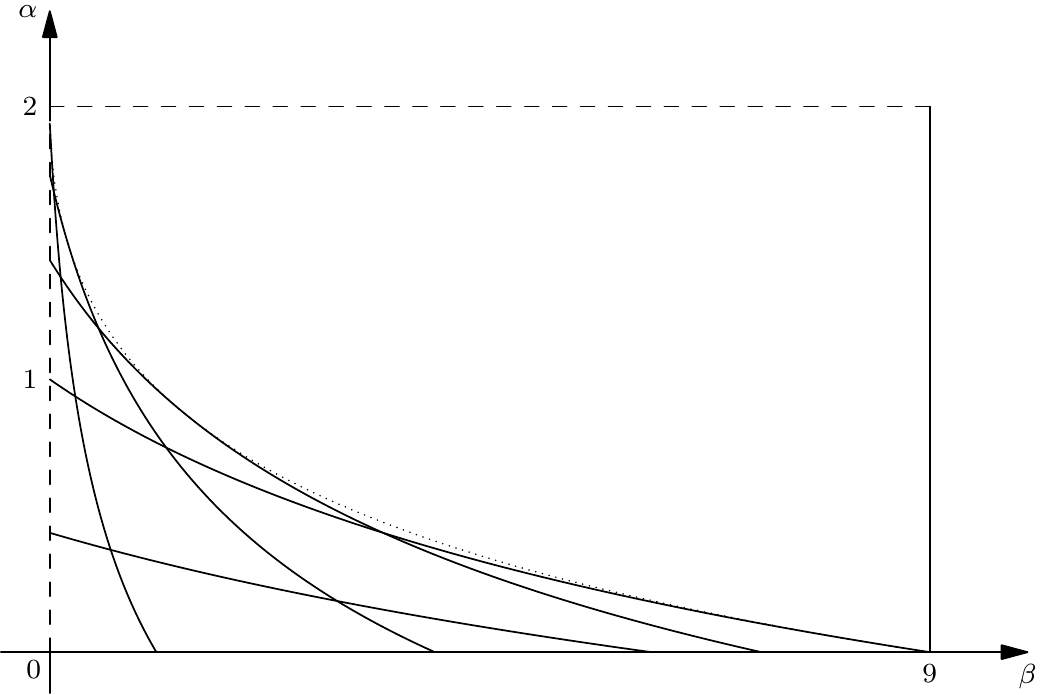}} \\
	\subfloat[][$k=6$.]{\includegraphics[width=0.45\textwidth]{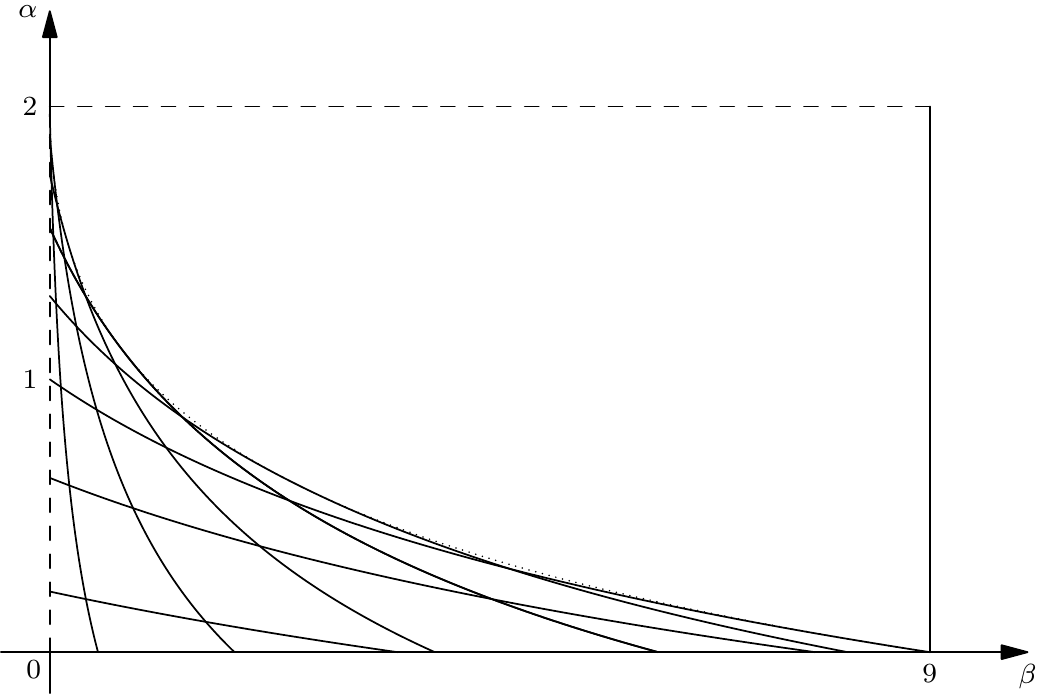}} \quad
	\subfloat[][$k=8$.]{\includegraphics[width=0.45\textwidth]{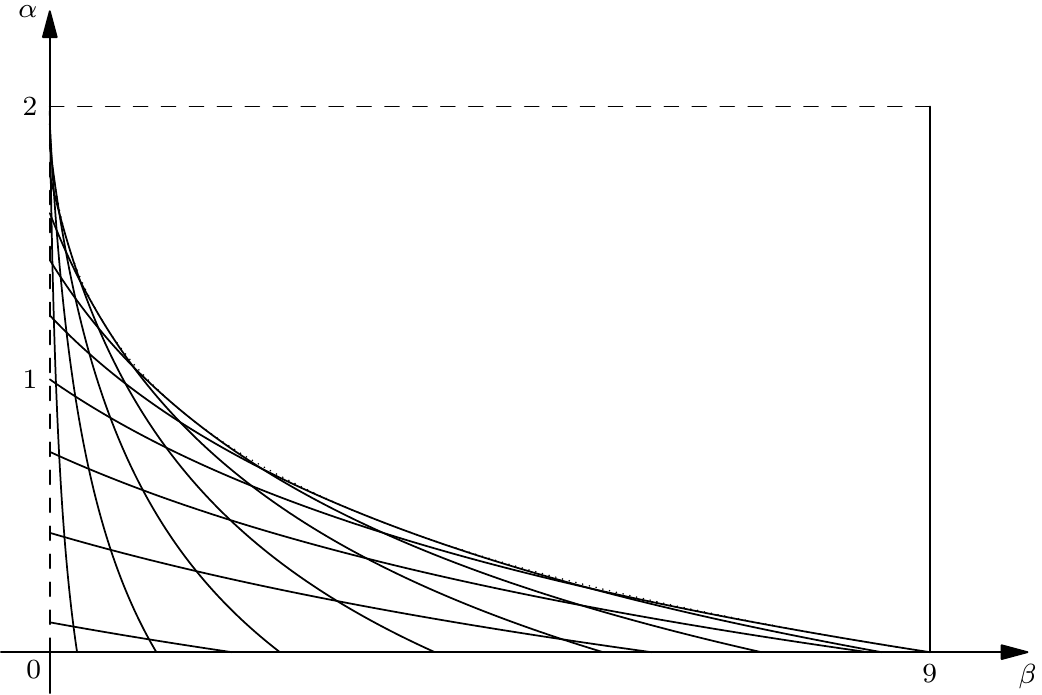}}
	\caption{Curves $\{f_{k,l}\}$ on which the Maslov index of the $k$-th iteration of the problem restricted to $E_3$ jumps (the values of $k$ taken into account are shown below each subfigure).
			The dotted line is the stability curve, which is approximated more and more accurately as $k$ increases.} \label{fig:iterates3}
	\end{figure}	
	
	As for the Morse index of the family of circular Lagrangian solutions of the planar $3$-body-type problem, an interesting result is due to Venturelli
	\cite[Theorem~3.1.7, page~25]{Venturelli:phd}, who proved that for $\alpha = 1$ the minimisers of the Lagrangian action functional among the loops under a homological constraint are circular
	orbits. Moreover, he showed that for equal masses ($\beta = 9$) and $\alpha\in (1,2)$ the periodic solution $\gamma_{\alpha,9}$ is a strict local minimiser, whereas for $\alpha \in [0,1)$ it is a
	saddle. 
	The problem of determining the Morse index of the circular Lagrangian orbit for different masses and for any parameter $\alpha \in (0,2)$ has been left unsolved until now.
	
	 \begin{teo}
		The Morse index of the  Lagrangian circular solution $\gamma_{\alpha, \beta}$ is given by		
		\[
			\iMor(\gamma_{\alpha, \beta}) = 
					\begin{cases}
						0 & \text{if $\alpha \in [1, 2)$} \\[10pt]
						2 & \textup{if $\beta \geq \dfrac{36(1 - \alpha)}{(\alpha + 2)^2}$ and $\alpha \in [0,1)$} \\[10pt]
						4 & \textup{if $0 < \beta < \dfrac{36(1 - \alpha)}{(\alpha + 2)^2}$}.
					\end{cases}
		\]
		The result is depicted in Figure~\ref{fig:iMor}.
	\end{teo}
	
	In the particular case of $\alpha=1$ we recover the results proved by X.~Hu and S.~Sun in \cite[Formulas~(55)--(56)]{MR2563212}: they compute the Morse index of the Lagrangian elliptic
	orbits of the classical three-body problem taking as parameters the eccentricity of the orbit and $\beta$.
	
	As for the generalised Kepler problem, we are able to determine, via $\omega$-index theory, any $\iMor(\gamma^k_{\alpha, \beta})$ for all $k \geq 2$. It is worth 
	noting that $\alpha=2$ is the limit of some values $\alpha_k$ that are the points where the Morse index of the $k$-th 
	iteration of the circular Keplerian solution jumps. Moreover this limit value (which coincides with the lower bound of the 
	strong force condition) is the boundary of the spectral stability region. The very same behaviour appears also in the restriction to the symplectic invariant subspace $E_3$, although the 
	curves of the $(\beta,\alpha)$-plane over which the Morse index of all the $k$-iterations jumps are no longer straight lines.
	As we show at the end of Subsection~\ref{subsec:omegaE3}, the boundary of the stability region is the enveloping curve of a two-parameter family of curves representing the jumps in the
	Morse index of the iterations of the solution.
	It seems then quite plausible to conjecture that the points at which a transition of stability occurs could be locally approximated, in a suitable sense, by curves along which there is a change in
	the Morse index of all the iterates.
	
	Let us now compare our result with some other important contributions on the subject. Being every relative equilibrium a zero-average loop solution, our theorem and
	\cite[Theorem~3.1.7, page~25]{Venturelli:phd} seem to be in striking contrast with the main theorem by A.~Chenciner and N.~Desolneux in \cite{MR1642007}, where they proved that
	$\gamma_{\alpha, \beta}$, for $\alpha \in (0, +\infty)$, are global minima of the action functional defined on the space of $W^{1,2}$-loops with zero average (and fixed centre of mass).
	However, although for $\alpha \in [0,1)$ we show that the Morse index is strictly positive, there is no contradiction because we do not restrict ourselves to the zero-average $W^{1,2}$-loop
	space. One might observe that the domain of the functional analysed by Chenciner and Desolneux includes collisions and ours does not, but this is not at all influential on the question: even
	taking into account those singularities the Morse index would not be affected, being it a local function and being relative equilibria always collisionless by definition.
	The main result in \cite{MR1642007} has been recently generalised in \cite{MR2097664}, where V.~Barutello and S.~Terracini proved that for every $\alpha \in (0, +\infty)$ 
	the absolute minimum among simple choreographies is attained on a relative equilibrium motion associated with the 
	regular $n$-gon. We observe that in imposing the choreographic symmetry constraint the authors require as well that the masses 
	be equal, so that the symmetry may act transitively on the bodies' labels. This corresponds in our setting to fixing $\beta = 9$.
	Their result \cite[Theorem~1]{MR2097664} entails that the circular Lagrange solution is an absolute minimum of the action functional on the $W^{1,2}$-choreographies.
	However, by our theorem we have that for $\alpha \in [0,1)$ the Morse index is $2$. We observe as above that this is not in contrast with our result
	since we are computing the Morse index in a strictly larger space.
	
	The main tool we used to demonstrate these results is an index theory, namely a Morse index theorem that relates the Morse index of a critical point of the Lagrangian action functional and the
	Maslov index of the fundamental solution associated with the corresponding Hamiltonian system.
	The problem of computing the Morse index is then translated into the computation of the Maslov index. The key ingredient in 
	order to switch from the Morse index to the Maslov index is the use of the Morse index theorem. 
	In order to compute this symplectic invariant we avail ourselves of some canonical
	transformations that involve a symplectic change of coordinates. Such new coordinates provide two useful advantages: first, the linearised Hamiltonian
	system becomes autonomous; second, the reduced phase space is split into two symplectic $4$-dimensional subspaces $E_2$ and $E_3$ which are invariant under the phase flow. As a
	consequence, the Maslov index is obtained as the sum of the Maslov indices of the restrictions of the fundamental solution to these subspaces.
	Although some formulas for the computation of the Maslov index exist for non-degenerate situations (involving for instance the Krein signature), we point out that $E_2$ gives rise to
	a really degenerate setting. We overcome all of these problems by using different notions of Maslov index available in the literature, all of which differ by the 
	contribution at the endpoints and by their homotopy properties. In order to overcome the degeneracy on $E_2$ we used the axiomatic definition given by Cappell, Lee and Miller in their 
	well-known paper \cite{MR1263126}, while to manage the degeneracy represented by the boundary of the stability region on $E_3$ we mainly employ the Maslov index introduced by Long.
	In Section~\ref{sec:symplecticpreliminaries} we recall the puzzle of all these indices trying to point out their main properties as well as the intertwining relations between them.
	Due to the low dimension, in all of our computation a big role is played by the geometry of $\Sp(2)$. To this end and for the sake of the reader we dedicate Appendix~\ref{app:SP2} to fix our
	notation and to recall some well-known facts scattered in the literature.  
	As already observed, a key result is represented by the Morse index theorem stating the relation between the Morse index of the essentially positive Fredholm quadratic forms 
	associated with the second variation and the Maslov index of the periodic solution. Appendix~\ref{sec:Fredholmforms} is devoted to fixing and clarifying the functional-analytical setting.
		
	\tableofcontents

	\paragraph{Acknowledgements.} We would like to thank Prof.~Susanna Terracini for many helpful discussions, Prof.~Yiming Long and Prof.~Xijun Hu for numerous conversations and
	suggestions on this research project.

%	--------------------------------------------------------------------------------------------------------------------
	\section{Relative equilibria and a symplectic decomposition of the phase space} \label{sec:description}
%	--------------------------------------------------------------------------------------------------------------------
	
	Consider three bodies with positive masses $m_1$, $m_2$, $m_3$ moving in the Euclidean plane $\R^2$ and denote by $q \= \Bigl( \begin{smallmatrix} q_1\\ q_2\\ q_3 \end{smallmatrix}
	\Bigr) \in \R^6$ the column vector of all positions, where each $q_i$ is a column vector in $\R^2$.	 
	 
	We are interested in finding periodic solutions of the Newtonian system
			\begin{equation}\label{eq:Newton}
				M \ddot{q} = \nabla U(q),
			\end{equation}
	where $U : X \subset \R^6 \to \R$ is one of the two potential functions 
			\begin{subequations} \label{eq:Uhom}
				\begin{align}
					\Uhom(q) & \= \sum_{\substack{i, j = 1\\ i<j}}^3 \frac{m_i m_j}{\abs{q_i - q_j}^{\alpha}}, \qquad \alpha \in (0, 2), \\
					\Ulog(q) & \= - \sum_{\substack{i, j = 1\\ i<j}}^3 m_i m_j \log{\abs{q_i - q_j}}
				\end{align}
			\end{subequations}
	($\alpha = 1$ corresponds to the gravitational case) defined on the \emph{collision-free configuration space}
		\[
			X \= \R^6 \setminus \Set{ q \in \R^6 | q_i = q_j \text{ for some } i \neq j }.
		\]
	The symbol $\abs{\, \cdot\,}$ indicates the Euclidean norm in $\R^2$, whilst $M \in \Mat(6, \R)$ is the diagonal \emph{mass matrix} $\diag(m_1 I_2, m_2 I_2, m_3 I_2)$ and $I_k$ is the
	$k \times k$ identity matrix. 
	
	In order to rewrite the second-order system~\eqref{eq:Newton} as a first-order Hamiltonian system we define the \emph{Hamiltonian function} $\H : T^*{X} \to \R$ to be
		\begin{equation} \label{eq:hamiltonian}
			\H(p, q) \= \frac{1}{2}\langle M^{-1}\trasp{p}, \trasp{p} \rangle - U(q),
		\end{equation}
	where $p \= (p_1, p_2, p_3) \in \R^6$ is the row vector of the linear momenta conjugate to $q$.
	Hence System~\eqref{eq:Newton} becomes
		\begin{equation} \label{eq:Hamilton}
			\begin{cases}
				\trasp{\dot{p}} = -\partial_q \H = \nabla U(q) \\
				\dot{q} = \partial_p \H =M^{-1} \trasp{p}.	
			\end{cases}
		\end{equation}
	Let us remark that by summing up the equations of~\eqref{eq:Newton} we obtain that the centre of mass of the system moves uniformly along a straight line; therefore, without loss of generality,
	we can fix it at the origin and study the dynamics on the \emph{reduced (collision-free) configuration space}
		\[
 			\hat{X} \= \Set{ q \in X | \sum_{i = 1}^3 m_i q_i = 0}.
		\]
	the reduced phase space $T^*{\hat{X}}$ is therefore $8$-dimensional.

	%\red
	{
	%	----------------------------------------------
		\subsection{Relative equilibria and central configurations} \label{subsec:cc}
	%	----------------------------------------------
	%	
	Among all the non-colliding solutions of Newton's Equations~\eqref{eq:Newton}, maybe the simplest are represented by a special class of periodic solutions called \emph{relative equilibria}:
	they are special motions which are at rest in a uniformly rotating frame. In the following and throughout all this paper, the matrix
		\[
			J_{2n} \= \begin{pmatrix}
					0 & -I_n \\
					I_n & 0
				\end{pmatrix}
		\]
	will denote the complex structure in $\R^{2n}$, but it will always be written simply as $J$, its dimension being clear from the context. The symplectic form on $\R^{2n}$ is then represented
	through the scalar product $\langle J \cdot, \cdot \rangle$.
	Let $e^{\omega J t} = \begin{pmatrix} \cos\omega t & -\sin\omega t \\ \sin\omega t & \cos\omega t \end{pmatrix}$ be the matrix representing the rotation in the plane with angular velocity
	$\omega$. With the symplectic change of coordinates
		\begin{equation*}
			\begin{cases}
				\trasp{y} \= e^{\omega K t}\, \trasp{p} \\
				x \= e^{\omega K t}\, q,
			\end{cases}
		\end{equation*}
	where $K$ is the $6 \times 6$ block-diagonal matrix $\diag(J, J, J)$, we rewrite Hamilton's Equations~\eqref{eq:Hamilton} in a frame uniformly rotating about the origin with a period
	$2\pi/\omega$:
		\begin{equation}\label{eq:newHamilton}
			\begin{cases}
				\trasp{\dot{y}} = -\partial_x \hat{\H} = \omega K \trasp{y} + \nabla U(x) \\
				\dot{x} = \partial_y \hat{\H} = M^{-1} \trasp{y} + \omega Kx,
			\end{cases}
		\end{equation}
	where $\hat{\H}$ is the new Hamiltonian function given by
		\begin{equation}\label{eq:newHamiltonian}
			 \hat{\H}(y, x) \= \frac{1}{2} \langle M^{-1} \trasp{y}, \trasp{y} \rangle - U(x) - \omega \langle K \trasp{y}, x \rangle.
		\end{equation}
	From the physical point of view, the terms involving $K$ come from the Coriolis force. A \textit{relative equilibrium} $\trasp{(\bar{y}, \trasp{\bar{x}})}$ is then an equilibrium point for System~\eqref{eq:newHamilton} and
	must satisfy the conditions
		\begin{equation}\label{eq:equilibrium2}
			\begin{cases}
				M^{-1} \nabla U(\bar{x}) + \omega^2 \bar{x} = 0 \\
				\trasp{\bar{y}} = -\omega M K \bar{x}.
			\end{cases}
		\end{equation}
	Note that the first equation just involves the configuration $\bar{x}$ and it is the well known \textit{central configuration equation} 
	(for further details see \cite{MR3227283}).
	Using Euler's Theorem for homogeneous functions one can compute 
	\begin{equation} \label{eq:lambdahom}
				\omega^2 = \begin{cases}
							\lambda_\alpha \= \dfrac{\alpha \Uhom(\bar{x})}{\mathcal{I}(\bar{x})} & \text{if } U = \Uhom \\[10pt]
							\lambda_{\log} \= \displaystyle \frac{1}{\mathcal{I}(\bar{x})} \sum_{\substack{i, j = 1 \\ i < j}}^n m_i m_j & \text{if } U = \Ulog,
						\end{cases}
	\end{equation}
	where
	    	\[
	    		\mathcal{I}(\bar{x}) \= \langle M\bar{x}, \bar{x} \rangle = \sum_{i = 1}^3 m_i \abs{\bar{x}_i}^2,
	    	\]
	is the (double of) the \emph{moment of inertia} (and a norm in $\R^6$).				
	Hence if we let three bodies, distributed in a planar central configuration, rotate with an angular velocity $\omega$ equal to $\sqrt{\lambda_\alpha}$ or to $\sqrt{\lambda_{\log}}$ we get a
	relative equilibrium, which becomes an equilibrium in a uniformly rotating coordinate system.
	\begin{rmk}
		We observe that $\bar{x}$ is a central configuration if and only if it is a constrained critical point of $\Uhom$ on a level surface of $\mathcal{I}$; furthermore if $\bar{x}$ is a central configuration then $c\bar{x}$
		and $O \bar{x}$ are, for any $c \in \R \setminus \{ 0 \}$ and any $6 \times 6$ block-diagonal matrix $O$ with entries given by a $2 \times 2 $ fixed matrix in $\SO(2)$. Because of these facts, it is standard practice
		to take the quotient of the configuration space $\hat{X}$ with respect to homotheties and rotations about the origin, which gives the so-called \emph{shape sphere} $\mathbb{S}$.
		It is well known by the studies of Lagrange and Euler (\cite{MR3227283}) that on $\mathbb{S}$ (for any choice of the masses) there are exactly five central configurations: three of them are collinear (the three
		bodies lie on the same line), while in the other two the bodies are arranged at the vertices of a regular triangle.  
	\end{rmk}
	}

	%	--------------------------------------------------------------------------------
		\subsection{A symplectic decomposition of the phase space for the linearised system} 
		\label{subsec:decomp}
	%	--------------------------------------------------------------------------------
	
	Consider the Hamiltonian System~\eqref{eq:Hamilton} in $\R^{12}$
		\begin{equation} \label{hamzeta}
			\dot{\zeta}(t) = J\nabla \H \bigl( \zeta(t) \bigr),
		\end{equation}
	where $\zeta \= \trasp{(p, \trasp{q})}$ and $\H$ is the Hamiltonian of the $3$-body problem defined in \eqref{eq:hamiltonian}. We linearise it around a relative equilibrium $\bar{\zeta}$ and write
		\begin{equation} \label{eq:hamzetalinearised}
			\dot{\zeta}(t) = J D^2 \H(\bar{\zeta})\, \zeta(t).
		\end{equation}
	
	The presence of the first integrals of motion and the invariance of the problem under some isometries gives rise to three symplectic invariant subspaces of the phase space: $E_1$,
	carrying the information about the translational invariance, $E_2$, generated by the conservation of the angular momentum and by the invariance by dilations, and $E_3$, defined as the
	symplectic orthogonal complement of the first two.

	%\red
	{
	Indeed, a basis for the position and momentum of the centre of mass is given by the four vectors in $\R^{12}$
		\[
			G_1 \= \begin{pmatrix}
					Mv \\
					0
				\end{pmatrix}, \qquad
			G_2 \= \begin{pmatrix}
					KMv \\
					0
				\end{pmatrix}, \qquad
			g_1 \= \begin{pmatrix}
					0 \\
				 	v
				\end{pmatrix}, \qquad
			g_2 \= \begin{pmatrix}
					0 \\
					Kv
				\end{pmatrix}
		\]
	with $v \= \trasp{(1, 0, 1, 0, 1, 0)} \in \R^6$. If we let $E_1$ be the space spanned by these vectors, it turns out that it is invariant and also symplectic.
	Note that the symplectic complement of $E_1$ is the space where the barycentre of the system is fixed at the origin and the total linear momentum is zero.
	The scaling and rotational symmetries generate another linear symplectic invariant subspace $E_2$, a basis of which is given by the four vectors in $\R^{12}$
		\[
			Z_1 \= \begin{pmatrix}
					M\bar{q} \\
					0
				\end{pmatrix}, \qquad
			Z_2 \= \begin{pmatrix}
					KM\bar{q} \\
					0
				\end{pmatrix}, \qquad
			z_1 \= \begin{pmatrix}
					0 \\
					\bar{q}
				\end{pmatrix}, \qquad
			z_2 \= \begin{pmatrix}
					0 \\
					K\bar{q}
				\end{pmatrix}.
		\]
	The coordinates on third subspace $E_3$ will be denoted by $\trasp{(W, \trasp{w})}$; note that this also is $4$\nobreakdash-dimensional.
	}
		      	
	We now derive a useful expression of the matrix of the linearised system by adapting the proof of Meyer and Schmidt in \cite[Lemma~3.1, pages~271--273]{MR2145251} to the case of the
	$\alpha$-homogeneous potential, but restricting ourselves to the circular case, i.e.~with zero eccentricity. In order to simplify the computations we set, without loss of generality,
		\[
			m_1 + m_2 + m_3 = 1;
		\]
	furthermore we introduce the key parameter
		\[
			\beta \= \frac{27(m_1 m_2 + m_1 m_3 + m_2 m_3)}{{(m_1 + m_2 + m_3)}^2} = 27(m_1 m_2 + m_1 m_3 + m_2 m_3) \in (0, 9].
		\]
		
	\begin{prop}[$\alpha$-homogeneous case] \label{prop:Lambdahom}
		There exists a system of symplectic coordinates $\xi \= \trasp{(\bar{Z}, \bar{W}, \trasp{\bar{z}}, \trasp{\bar{w}})} \in \R^8$ and a rescaled time $\tau$ such that the linearised
		System~\eqref{eq:hamzetalinearised} restricted to $E_2 \oplus E_3 = T^*\hat{X}$ has the form
			\begin{equation} \label{eq:sysR8}
				\frac{d\xi}{d\tau} = \Lambda \xi,
			\end{equation}
		where
			\begin{equation} \label{eq:Lambda}
				\Lambda \= \begin{pmatrix}
					0 & 1 & 0 & 0 & \alpha + 1 & 0 & 0 & 0 \\
					-1 & 0 & 0 & 0 & 0 & -1 & 0 & 0 \\
					0 & 0 & 0 & 1 & 0 & 0 & \dfrac{1}{2} \Bigl( \alpha + \dfrac{\alpha + 2}{3}\sqrt{\smash[b]{9 - \beta}} \Bigr) & 0 \\
					0 & 0 & -1 & 0 & 0 & 0 & 0 & \dfrac{1}{2} \Bigl( \alpha - \dfrac{\alpha + 2}{3}\sqrt{\smash[b]{9 - \beta}} \Bigr) \\
					1 & 0 & 0 & 0 & 0 & 1 & 0 & 0 \\
					0 & 1 & 0 & 0 & -1 & 0 & 0 & 0 \\
					0 & 0 & 1 & 0 & 0 & 0 & 0 & 1 \\
					0 & 0 & 0 & 1 & 0 & 0 & -1 & 0
				\end{pmatrix}.
			\end{equation}
	\end{prop}
	
	\begin{proof}
		The Hamiltonian of the system in the fixed reference frame is
			\[
				\H(p, q) \= \frac{1}{2} \langle M^{-1}\trasp{p}, \trasp{p} \rangle - \Uhom(q),
			\]
		We make the following symplectic change of coordinates:
			\begin{equation} \label{eq:changepqtogzw}
				\trasp{p} = \traspinv{C}\begin{pmatrix} \trasp{G}\\ \trasp{Z}\\ \trasp{W} \end{pmatrix}, \qquad q = C \begin{pmatrix} g\\ z\\ w \end{pmatrix},
			\end{equation}
		where $C$ is given by (cf.~\cite[pages~268--269]{MR2145251})
			\[
				C \= \begin{pmatrix}
						1 & 0 & \frac{9(m_2 + m_3)}{2 \sqrt{\smash[b]{\beta}}} & \frac{3\sqrt{3}(m_2 - m_3)}{2 \sqrt{\smash[b]{\beta}}} & 0 &
																							- \frac{3\sqrt{3} \sqrt{m_2 m_3}}{\sqrt{\smash[b]{\beta}}\sqrt{m_1}} \\[10pt]
						0 & 1 & - \frac{3\sqrt{3}(m_2 - m_3)}{2 \sqrt{\smash[b]{\beta}}} & \frac{9(m_2 + m_3)}{2 \sqrt{\smash[b]{\beta}}} &
																						\frac{3\sqrt{3} \sqrt{m_2 m_3}}{\sqrt{\smash[b]{\beta}}\sqrt{m_1}} & 0 \\[10pt]
						1 & 0 & - \frac{9 m_1}{2 \sqrt{\smash[b]{\beta}}} & - \frac{3\sqrt{3}(m_1 + 2m_3)}{2 \sqrt{\smash[b]{\beta}}} &
											\frac{9 \sqrt{m_1 m_3}}{2 \sqrt{\smash[b]{\beta}} \sqrt{m_2}} & \frac{3\sqrt{3}\sqrt{m_1 m_3}}{2 \sqrt{\smash[b]{\beta}} \sqrt{m_2}} \\[10pt]
						0 & 1 & \frac{3\sqrt{3}(m_1 + 2m_3)}{2 \sqrt{\smash[b]{\beta}}} & - \frac{9 m_1}{2 \sqrt{\smash[b]{\beta}}} &
											- \frac{3\sqrt{3}\sqrt{m_1 m_3}}{2 \sqrt{\smash[b]{\beta}} \sqrt{m_2}} & \frac{9 \sqrt{m_1 m_3}}{2 \sqrt{\smash[b]{\beta}} \sqrt{m_2}} \\[10pt]
						1 & 0 & - \frac{9 m_1}{2 \sqrt{\smash[b]{\beta}}} & \frac{3\sqrt{3}(m_1 + 2m_2)}{2 \sqrt{\smash[b]{\beta}}} &
											- \frac{9 \sqrt{m_1 m_2}}{2 \sqrt{\smash[b]{\beta}} \sqrt{m_3}} & \frac{3\sqrt{3} \sqrt{m_1 m_2}}{2 \sqrt{\smash[b]{\beta}} \sqrt{m_3}} \\[10pt]
						0 & 1 & - \frac{3\sqrt{3}(m_1 + 2m_2)}{2 \sqrt{\smash[b]{\beta}}} & - \frac{9 m_1}{2 \sqrt{\smash[b]{\beta}}} &
												- \frac{3\sqrt{3} \sqrt{m_1 m_2}}{2 \sqrt{\smash[b]{\beta}} \sqrt{m_3}} & - \frac{9 \sqrt{m_1 m_2}}{2 \sqrt{\smash[b]{\beta}} \sqrt{m_3}}
					\end{pmatrix}.
			\]
		It is a straightforward computation to verify that $C$ is invertible and it satisfies the relations
			\[
				\trasp{C}MC = I, \qquad C^{-1}JC = J.
			\]
		After fixing the centre of mass at the origin (\ie setting $g = \trasp{G} = 0$, thus restricting the system to $E_2 \oplus E_3$), the Hamiltonian of the system becomes
			\[
				\H \bigl( Z, W, z, w \bigr) = \frac{1}{2} \bigl( Z_1^2 + Z_2^2 + W_1^2 + W_2^2 \bigr) - \Uhom(z, w).
			\]
		Consider now the rotation in the plane
			\[
				R(t) \= \begin{pmatrix}
						\cos(\lambda_\alpha t) & -\sin(\lambda_\alpha t) \\
						\sin(\lambda_\alpha t) & \cos(\lambda_\alpha t)
					\end{pmatrix},
			\]
		where $\lambda_\alpha$ is the Lagrange multiplier \eqref{eq:lambdahom} of the central configuration, corresponding to the square of the angular velocity of each body. Accordingly, we
		move to a uniformly rotating reference frame in the following way:
			\begin{equation} \label{eq:rotation}
				\begin{cases}
					\trasp{Z} = R(t) \trasp{\tilde{Z}} \\
					\trasp{W} = R(t) \trasp{\tilde{W}} \\
					z = R(t) \tilde{z} \\
					w = R(t) \tilde{w}.
				\end{cases}
			\end{equation}
		Since we are moving to a new set of canonical coordinates (see for instance \cite[Chapter~9]{Goldstein}) via the time-depending generating function
			\[
				F\bigl( Z, W, \tilde{z}, \tilde{w}, t \bigr) := -ZR(t)\tilde{z} - WR(t)\tilde{w},
			\]
		the new Hamiltonian function (still denoted by $\H$) must contain the extra term $\displaystyle \frac{d F}{d t}$:
			\[
				\begin{split}
					\H \bigl( \tilde{Z}, \tilde{W}, \tilde{z}, \tilde{w} \bigr) & = \frac{1}{2} \bigl( \tilde{Z}_1^2 + \tilde{Z}_2^2 + \tilde{W}_1^2 + \tilde{W}_2^2 \bigr) - \Uhom(\tilde{z}, \tilde{w}) \\
												& \quad\, + \lambda_\alpha \bigl( \tilde{Z}_1\tilde{z}_2 - \tilde{Z}_2\tilde{z}_1 + \tilde{W}_1\tilde{w}_2 - \tilde{W}_2\tilde{w}_1 \bigr).
				\end{split}
			\]
		Then we operate the following symplectic scaling with multiplier $\lambda_\alpha^{-\frac{\alpha}{\alpha + 2}}$:
			\begin{equation} \label{eq:scaling}
				\begin{cases}
					\tilde{Z} = \lambda_\alpha^{\frac{\alpha + 1}{\alpha + 2}} \hat{Z} \\
					\tilde{W} = \lambda_\alpha^{\frac{\alpha + 1}{\alpha + 2}} \hat{W} \\
					\tilde{z} = \lambda_\alpha^{-\frac{1}{\alpha + 2}} \hat{z} \\
					\tilde{w} = \lambda_\alpha^{-\frac{1}{\alpha + 2}} \hat{w}
				\end{cases}
			\end{equation}
		obtaining thus
			\[
				\begin{split}
					\H \bigl( \hat{Z}, \hat{W}, \hat{z}, \hat{w} \bigr) & = \frac{\lambda_\alpha}{2} \bigl( \hat{Z}_1^2 + \hat{Z}_2^2 + \hat{W}_1^2 + \hat{W}_2^2 \bigr)
																										- \Uhom(\hat{z}, \hat{w}) \\
													& \quad\, + \lambda_\alpha \bigl( \hat{Z}_1\hat{z}_2 - \hat{Z}_2\hat{z}_1 + \hat{W}_1\hat{w}_2 - \hat{W}_2\hat{w}_1 \bigr).
				\end{split}
			\]
		The next step consists in a time scaling: define $\tau \= \lambda_\alpha t$ and rewrite System~\eqref{hamzeta} as
			\[
				\frac{d \zeta \bigl( \tau(t) \bigr)}{d \tau} \frac{d\tau(t)}{dt} = J\nabla \H \bigl( \zeta(\tau(t)) \bigr),
			\]
		or equivalently as
			\begin{equation} \label{eq:hamsystau}
				\zeta'(\tau) \lambda_\alpha = J\nabla \H \bigl( \zeta(\tau) \bigr),
			\end{equation}
		where the prime $'$ denotes the derivative with respect to $\tau$. Hence a division of both sides of \eqref{eq:hamsystau} by $\lambda_\alpha$ yields the equivalent system
			\[
				\zeta'(\tau) = J\nabla \hat{\H} \bigl( \zeta(\tau) \bigr),
			\]
		where
			\[
				\begin{split}
					\hat{\H} \bigl( \hat{Z}, \hat{W}, \hat{z}, \hat{w} \bigr) & = \frac{1}{2} \bigl( \hat{Z}_1^2 + \hat{Z}_2^2 + \hat{W}_1^2 + \hat{W}_2^2 \bigr)
																										- \frac{1}{\lambda_\alpha}\Uhom(\hat{z}, \hat{w}) \\
													& \quad\, + \hat{Z}_1\hat{z}_2 - \hat{Z}_2\hat{z}_1 + \hat{W}_1\hat{w}_2 - \hat{W}_2\hat{w}_1 \\
													& = \frac{1}{\lambda_\alpha} \H \bigl( \hat{Z}, \hat{W}, \hat{z}, \hat{w} \bigr).
				\end{split}
			\]
		Finally, in order to shift the equilibrium point into the origin, we operate a translation and set
			\begin{equation} \label{eq:translation}
				\begin{cases}
					\bar{Z}_1 \= \hat{Z}_1 \\
					\bar{Z}_2 \= \hat{Z}_2 - 1 \\
					\bar{W}_1 \= \hat{W}_1 \\
					\bar{W}_2 \= \hat{W}_2 \\
					\bar{z}_1 \= \hat{z}_1 - 1 \\
					\bar{z}_2 \= \hat{z}_2 \\
					\bar{w}_1 \= \hat{w}_1 \\
					\bar{w}_2 \= \hat{w}_2
				\end{cases},
			\end{equation}
		whence
			\[
				\begin{split}
					\hat{\H} \bigl( \bar{Z}, \bar{W}, \bar{z}, \bar{w} \bigr) & = \frac{1}{2} \bigl[ \bar{Z}_1^2 + (\bar{Z}_2 + 1)^2 + \bar{W}_1^2 + \bar{W}_2^2 \bigr] -
																\frac{1}{\lambda_\alpha}\Uhom(\bar{z}, \bar{w}) \\
														& \quad\, + \bar{Z}_1\bar{z}_2 - (\bar{Z}_2 + 1)(\bar{z}_1 + 1) + \bar{W}_1\bar{w}_2 - \bar{W}_2\bar{w}_1.
				\end{split}
			\]
		The matrix of the linearised system is ($J$ times) the Hessian of this Hamiltonian, evaluated at the origin. In order to write it down we need the Hessian of the potential $\Uhom$
		expressed in the coordinates $(\bar{z}, \bar{w})$, but since the computations are quite long and tedious we shall omit them and indicate only the way in which we obtained the result.
		We have that
			\[
				\Uhom(\bar{z}, \bar{w}) = \frac{m_1 m_2}{d_{12}^\alpha} + \frac{m_1 m_3}{d_{13}^\alpha} + \frac{m_2 m_3}{d_{23}^\alpha},
			\]
		where
			\begin{align*}
				d_{12} & \= \frac{3\sqrt{3}}{\sqrt{\smash[b]{\beta}}} \left[ (\bar{z}_1 + 1)^2 + \bar{z}_2^2 + \frac{m_3(m_1^2 + m_1 m_2 + m_2^2)}{m_1 m_2}(\bar{w}_1^2 + \bar{w}_2^2)\right. \\
									 & \qquad + \left. \sqrt{\frac{3m_2 m_3}{m_1}} \bigl( \bar{z}_2 \bar{w}_1 - (\bar{z}_1 + 1)\bar{w}_2 \bigr)
									 		- (2m_1 + m_2)\sqrt{\frac{m_3}{m_1 m_2}} \bigl( (\bar{z}_1 + 1) \bar{w}_1 + \bar{z}_2 \bar{w}_2 \bigr) \right]^{\frac{1}{2}}, \\
				d_{13} & \= \frac{3\sqrt{3}}{\sqrt{\smash[b]{\beta}}} \left[ (\bar{z}_1 + 1)^2 + \bar{z}_2^2 + \frac{m_2(m_1^2 + m_1 m_3 + m_3^2)}{m_1 m_3}(\bar{w}_1^2 + \bar{w}_2^2)\right. \\
									 & \qquad + \left. \sqrt{\frac{3m_2 m_3}{m_1}} \bigl( \bar{z}_2 \bar{w}_1 - (\bar{z}_1 + 1) \bar{w}_2 \bigr)
									 		+ (2m_1 + m_3)\sqrt{\frac{m_2}{m_1 m_3}} \bigl( (\bar{z}_1 + 1) \bar{w}_1 + \bar{z}_2 \bar{w}_2 \bigr) \right]^{\frac{1}{2}}, \\
				d_{23} & \= \frac{3\sqrt{3}}{\sqrt{\smash[b]{\beta}}} \left[ (\bar{z}_1 + 1)^2 + \bar{z}_2^2 + \frac{m_1(m_2^2 + m_2 m_3 + m_3^2)}{m_2 m_3}(\bar{w}_1^2 + \bar{w}_2^2)\right. \\
									 & \qquad - \left. (m_2 + m_3)\sqrt{\frac{3m_1}{m_2 m_3}} \bigl( \bar{z}_2 \bar{w}_1 - (\bar{z}_1 + 1) \bar{w}_2 \bigr)
									 		+ (m_2 - m_3)\sqrt{\frac{m_1}{m_2 m_3}} \bigl( (\bar{z}_1 + 1) \bar{w}_1 + \bar{z}_2 \bar{w}_2 \bigr) \right]^{\frac{1}{2}}. \\
			\end{align*}
		Now, calculating the Hessian of $\frac{1}{\lambda_\alpha}\Uhom$ and evaluating it at the origin yields
			\[
				\frac{1}{\lambda_\alpha}D^2\Uhom(0, 0) = \begin{pmatrix}
								\alpha + 1 & 0 & 0 & 0 \\
								0 & -1 & 0 & 0 \\
								0 & 0 & a & b \\
								0 & 0 & b & c
							\end{pmatrix},
			\]
		with $a \= \frac{1}{4}\bigl[ 4(\alpha + 1)m_1 + (\alpha - 2)(m_2 + m_3) \bigr]$, $b \= \frac{1}{4}\bigl[ \sqrt{3}(\alpha + 2)(m_2 - m_3) \bigr]$ and
		$c \= \frac{1}{4}\bigl[ -4m_1 + (3\alpha + 2)(m_2 + m_3) \bigr]$. The matrix of the linearised system is thus
			\begin{equation} \label{eq:Lambdabrutta}
				\begin{pmatrix}
					0 & 1 & 0 & 0 & \alpha + 1 & 0 & 0 & 0 \\
					-1 & 0 & 0 & 0 & 0 & -1 & 0 & 0 \\
					0 & 0 & 0 & 1 & 0 & 0 & a & b \\
					0 & 0 & -1 & 0 & 0 & 0 & b & c \\
					1 & 0 & 0 & 0 & 0 & 1 & 0 & 0 \\
					0 & 1 & 0 & 0 & -1 & 0 & 0 & 0 \\
					0 & 0 & 1 & 0 & 0 & 0 & 0 & 1 \\
					0 & 0 & 0 & 1 & 0 & 0 & -1 & 0
				\end{pmatrix}.
			\end{equation}
		Extracting from it the submatrix representing the dynamics on $E_3$ (\ie the one acting on the $W$'s and $w$'s only):
			\[
				\begin{pmatrix}
					0 & 1 & \frac{1}{4}\bigl[ 4(\alpha + 1)m_1 + (\alpha - 2)(m_2 + m_3) \bigr] & \frac{1}{4}\bigl[ \sqrt{3}(\alpha + 2)(m_2 - m_3) \bigr] \\
					-1 & 0 & \frac{1}{4}\bigl[ \sqrt{3}(\alpha + 2)(m_2 - m_3) \bigr] & \frac{1}{4}\bigl[ -4m_1 + (3\alpha + 2)(m_2 + m_3) \bigr] \\
					1 & 0 & 0 & 1 \\
					0 & 1 & -1 & 0
				\end{pmatrix},
			\]
		we apply a rotation to both positions $w$ and momenta $W$ and obtain
			\begin{equation} \label{eq:partesuE3}
				\begin{pmatrix}
					0 & 1 & \dfrac{1}{2} \Bigl( \alpha + \dfrac{\alpha + 2}{3}\sqrt{\smash[b]{9 - \beta}} \Bigr) & 0 \\
					-1 & 0 & 0 & \dfrac{1}{2} \Bigl( \alpha - \dfrac{\alpha + 2}{3}\sqrt{\smash[b]{9 - \beta}} \Bigr) \\
					1 & 0 & 0 & 1 \\
					0 & 1 & -1 & 0
				\end{pmatrix},
			\end{equation}
		so that the final matrix depends only on $\alpha$ and $\beta$. Now substitute \eqref{eq:partesuE3} back into \eqref{eq:Lambdabrutta} to get \eqref{eq:Lambda}.
	\end{proof}
	
	In the logarithmic case there is a completely similar result.
	
	\begin{prop}[Logarithmic case]
		There exist a system $\xi \= \trasp{(\bar{Z}, \bar{W}, \trasp{\bar{z}}, \trasp{\bar{w}})} \in \R^8$ of symplectic coordinates and a rescaled variable $\tau$ such that the linearised
		System~\eqref{eq:hamzetalinearised} restricted to $E_2 \oplus E_3 = T^*\widehat{X}$ has the form
			\[
				\frac{d\xi}{d \tau} = \Lambda \xi,
			\]
		where
			\begin{equation} \label{eq:Lambdalog}
				\Lambda \= \begin{pmatrix}
					0 & 1 & 0 & 0 & 1 & 0 & 0 & 0 \\
					-1 & 0 & 0 & 0 & 0 & -1 & 0 & 0 \\
					0 & 0 & 0 & 1 & 0 & 0 & \frac{1}{3} \sqrt{\smash[b]{9 - \beta}} & 0 \\
					0 & 0 & -1 & 0 & 0 & 0 & 0 & - \frac{1}{3} \sqrt{\smash[b]{9 - \beta}} \\
					1 & 0 & 0 & 0 & 0 & 1 & 0 & 0 \\
					0 & 1 & 0 & 0 & -1 & 0 & 0 & 0 \\
					0 & 0 & 1 & 0 & 0 & 0 & 0 & 1 \\
					0 & 0 & 0 & 1 & 0 & 0 & -1 & 0
				\end{pmatrix}.
			\end{equation}
	\end{prop}
	
	\begin{proof}
		We proceed exactly as in Proposition~\ref{prop:Lambdahom} with some slight modifications. After the symplectic change of coordinates \eqref{eq:changepqtogzw}, we have of course to
		replace $\lambda_{\alpha}$ with $\lambda_{\log}$.
	%\red
	{Hence we apply the rotation in the plane
			\[
				R(t) \= \begin{pmatrix}
						\cos{\lambda_{\log} t} & -\sin{\lambda_{\log} t} \\
						\sin{\lambda_{\log} t} & \cos{\lambda_{\log} t}
					\end{pmatrix},
			\]
		in the same way as in \eqref{eq:rotation}, getting
			\[
				\begin{split}
					\H \bigl( \tilde{Z}, \tilde{W}, \tilde{z}, \tilde{w} \bigr) & = \frac{1}{2} \bigl( \tilde{Z}_1^2 + \tilde{Z}_2^2 + \tilde{W}_1^2 + \tilde{W}_2^2 \bigr) - \Ulog(\tilde{z}, \tilde{w}) \\
												& \quad\, + \lambda_{\log} \bigl( \tilde{Z}_1\tilde{z}_2 - \tilde{Z}_2\tilde{z}_1 + \tilde{W}_1\tilde{w}_2 - \tilde{W}_2\tilde{w}_1 \bigr).
				\end{split}
			\]}
		Transformation~\eqref{eq:scaling} is now the following:
			\[
				\begin{cases}
					\tilde{Z} = \lambda_{\log}^{1/2} \hat{Z} \\
					\tilde{W} = \lambda_{\log}^{1/2} \hat{W} \\
					\tilde{z} = \lambda_{\log}^{-1/2} \hat{z} \\
					\tilde{w} = \lambda_{\log}^{-1/2} \hat{w}
				\end{cases}
			\]
		and gives
			\[
				\begin{split}
					\H \bigl( \hat{Z}, \hat{W}, \hat{z}, \hat{w} \bigr) & = \frac{\lambda_{\log}}{2} \bigl( \hat{Z}_1^2 + \hat{Z}_2^2 + \hat{W}_1^2 + \hat{W}_2^2 \bigr)
																							- \Ulog(\lambda_{\log}^{-1/2} \hat{z}, \lambda_{\log}^{-1/2} \hat{w}) \\
													& \quad\, + \lambda_{\log} \bigl( \hat{Z}_1\hat{z}_2 - \hat{Z}_2\hat{z}_1 + \hat{W}_1\hat{w}_2 - \hat{W}_2\hat{w}_1 \bigr).
				\end{split}
			\]
		Then we rescale time by setting $\tau \= \lambda_{\log} t$ and obtain
			\[
				\begin{split}
					\hat{\H} \bigl( \hat{Z}, \hat{W}, \hat{z}, \hat{w} \bigr) & = \frac{1}{2} \bigl( \hat{Z}_1^2 + \hat{Z}_2^2 + \hat{W}_1^2 + \hat{W}_2^2 \bigr)
																		- \frac{1}{\lambda_{\log}} \Ulog(\lambda_{\log}^{-1/2} \hat{z}, \lambda_{\log}^{-1/2} \hat{w}) \\
													& \quad\, + \hat{Z}_1\hat{z}_2 - \hat{Z}_2\hat{z}_1 + \hat{W}_1\hat{w}_2 - \hat{W}_2\hat{w}_1 \\
													& = \frac{1}{\lambda_{\log}} \H \bigl( \hat{Z}, \hat{W}, \hat{z}, \hat{w} \bigr).
				\end{split}
			\]
		Translation~\eqref{eq:translation} sets the equilibrium point at the origin and we have%
			\footnote{Here and in the following, with a slight abuse of notation, we denote by $\lambda_{\log}^{-1/2}\bar{z}$ the vector
					$\trasp{\bigl( \lambda_{\log}^{-1/2}(\bar{z}_1 + 1), \lambda_{\log}^{-1/2} \bar{z}_2 \bigr)}$.}
			\[
				\begin{split}
					\hat{\H} \bigl( \bar{Z}, \bar{W}, \bar{z}, \bar{w} \bigr) & = \frac{1}{2} \bigl[ \bar{Z}_1^2 + (\bar{Z}_2 + 1)^2 + \bar{W}_1^2 + \bar{W}_2^2 \bigr] -
																\frac{1}{\lambda_{\log}} \Ulog(\lambda_{\log}^{-1/2} \bar{z}, \lambda_{\log}^{-1/2} \bar{w}) \\
														& \quad\, + \bar{Z}_1\bar{z}_2 - (\bar{Z}_2 + 1)(\bar{z}_1 + 1) + \bar{W}_1\bar{w}_2 - \bar{W}_2\bar{w}_1.
				\end{split}
			\]
		The Hessian of $\frac{1}{\lambda_{\log}}\Ulog$ evaluated at the origin is
			\[
				\frac{1}{\lambda_{\log}}D^2\Ulog(0, 0) = \begin{pmatrix}
								1 & 0 & 0 & 0 \\
								0 & -1 & 0 & 0 \\
								0 & 0 & \frac{1}{2} (2 m_1 - m_2 - m_3) & \frac{\sqrt{3}}{2} (m_2 - m_3) \\
								0 & 0 & \frac{\sqrt{3}}{2} (m_2 - m_3) & -\frac{1}{2} (2 m_1 - m_2 - m_3)
							\end{pmatrix};
			\]
		an orthogonal transformation applied on the subspace $E_3$ to both positions $\bar{w}$ and momenta $\bar{W}$ diagonalises the lower right corner of
		$\lambda_{\log}^{-1} D^2\Ulog(0,0)$, making it dependent only on $\beta$:
			\[
				\frac{1}{\lambda_{\log}}D^2\Ulog(0,0) = \begin{pmatrix}
								1 & 0 & 0 & 0 \\
								0 & -1 & 0 & 0 \\
								0 & 0 & \dfrac{1}{3} \sqrt{\smash[b]{9 - \beta}} & 0 \\
								0 & 0 & 0 & -\dfrac{1}{3} \sqrt{\smash[b]{9 - \beta}}
							\end{pmatrix}.
			\]
		By computing the Hessian of $\H$ and multiplying on the left by $J$, we find the matrix $\Lambda$ of the statement.
	\end{proof}
	
	\begin{rmk} \label{rmk:alpha=0}
		We observe that \eqref{eq:Lambdalog} can be obtained from \eqref{eq:Lambda} simply by setting $\alpha = 0$. Therefore in the analysis that will follow we shall consider the logarithmic
		case as a subcase of the $\alpha$-homogeneous one. Note that this is a remark \textit{a posteriori}, since we could not deduce it directly from the relation
			\[
				\frac{\Uhom(q) - 1}{\alpha} \sim \Ulog(q)	\qquad \text{as } \alpha \to 0^+,
			\]
		which is only asymptotic.
	\end{rmk}

%	----------------------------------------------------------------------------------------
	\section{Maslov-type index theories} \label{sec:symplecticpreliminaries}
%	----------------------------------------------------------------------------------------

	The aim of this section is to briefly describe some Maslov-type index theories for paths of symplectic matrices as well as for paths of Lagrangian subspaces.
	In Subsection~\ref{subs:maslovsymplectic} we recall a geometric definition of the Maslov index for symplectic paths exploiting the intersection number of a curve and a singular cycle (an
	algebraic variety of codimension~$1$ in the symplectic group). Then, in Subsection~\ref{subsec:omega_Maslov}, we recollect the basic definitions of the $\omega$-index theory, essentially 
	developed by Long and his school, and exhibit the relation with the geometric Maslov-type index. Our main sources for these two subsections are 
	\cite{MR733717, MR1124230, MR1762278} and references therein.
	Subsection~\ref{subs:otherMaslov} is devoted to a brief presentation of other Maslov-type index theories defined through a suitable intersection theory in the Lagrangian Grassmannian
	manifold by means of the crossing forms. We also show the relationship with the Maslov-type index theories previously introduced in the symplectic context. Our basic references for all this
	are \cite{MR1241874, MR1263126, MR2383373, MR2057171, MR1898560, MR2563212, MR2817146, MR0211415, MR2046769, MR2133393, MR2574386}.

	\subsection{Maslov-type index theory for symplectic paths}\label{subs:maslovsymplectic}
	
	Following Long and Zhu in \cite{MR1762278}, we define for all $n \in \N \setminus \{0\}$ the \emph{complex and real symplectic groups}
		\begin{align*}
			\Sp(2n, \C) &\= \Set{M \in \GL(2n, \C) | M^\dagger JM = J}\\
			\Sp(2n) \= \Sp(2n, \R) & \= \Set{M \in \GL(2n, \R) | \trasp{M}JM = J}
		\end{align*}
	and for $ 0 \leq k \leq 2n$ we set
		\begin{align*}
			\Sp_k(2n, \C) &\= \Set{M \in \Sp(2n, \C) | \dim_\C \ker_\C (M - I) = k}\\
			\Sp_k(2n) &\= \Set{M \in \Sp(2n, \R) | \dim\ker(M - I) = k}.
		\end{align*}
	It is clear that one has the following stratifications:
		\[
			\Sp(2n, \C) = \bigcup_{k = 0}^{2n} \Sp_k(2n, \C), \qquad \Sp(2n) = \bigcup_{k = 0}^{2n} \Sp_k(2n).
		\]
	For the sake of the reader we recall the following well-known result, which gives the properties of the stratification.
	
	\begin{prop} \label{thm:decomSp}
		The subsets $\Sp_k(2n, \C)$ and $\Sp_k(2n)$ are, respectively speaking, smooth submanifolds of $\Sp(2n, \C)$ and $\Sp(2n)$, with codimension $k^2$ and $\frac{1}{2}k(k+1)$.
		Moreover, $\Sp_1(2n, \C)$ and $\Sp_1(2n)$ are co-oriented, the transverse orientation being given by the vector field $\frac{d}{dt} (M e^{Jt}) \big\rvert_{t=0}$.
		We have in addition that
			\[
				\overline{\Sp_k(2n, \C)}= \bigcup_{l \geq k} \Sp_l(2n,\C) \qquad \text{and} \qquad \overline{\Sp_k(2n)} = \bigcup_{l \geq k} \Sp_l(2n).
			\]
	\end{prop}
	By Proposition~\ref{thm:decomSp} the intersection points of the curve
		\[
			\gamma(t) \= Me^{Jt}, \qquad M \in \Sp_1(2n,\C)
		\]
	with the cycle $\overline{\Sp_1(2n, \C)}$ form a discrete subset of $\gamma(\R)$. We recall that a matrix in $\Sp(2n,\C)$ is called \emph{non-degenerate} if it does not admit $1$ as an
	eigenvalue. A straightforward computation allows to see that for a continuous path $\gamma:[a,b] \to \Sp(2n, \C)$ there exists $\delta > 0$ such that for any
	$\eps \in (-\delta, \delta)\setminus \{0\}$ the (perturbed) path $t \mapsto \gamma(t) e^{-\eps J}$ is non-degenerate, meaning that it has non-degenerate endpoints.
	
	%\red
	{
	\begin{defn}\label{def:maslov1}
		Let  $\gamma:[a,b] \to \Sp(2n, \C)$. We define its {\em geometric Maslov-type index\/} to be the intersection number of $t\mapsto \gamma(t) e^{-\varepsilon J}$ with
		$\overline{\Sp_1(2n, \C)}$ for all $ \varepsilon \in (0,\delta)$ (where $\delta$ is such that the perturbed path is non-degenerate).
			\begin{equation} \label{eq:indice1}
				\igeo(\gamma) \= \bigl[ \gamma e^{-\varepsilon J}:\overline{\Sp_1(2n, \C)}\bigr],
			\end{equation}
		where the right-hand side of \eqref{eq:indice1} is the usual homotopy intersection number.
	\end{defn}
	}

	For any $\omega \in \U \= \Set{ z \in \C | \abs{z} = 1 }$ and $T > 0$ it is convenient to define the set
		\[
			\mathscr P_T(2n) \= \Set{ \gamma \in \mathscr{C}^0 \bigl( [0, T]; \Sp(2n,\C) \bigr) | \gamma(0) = I_{2n} }
		\]
	and its subset
		\[
			\mathscr P ^*_{T, \omega}(2n) \= \Set{\gamma \in \mathscr{P}_T(2n) | \gamma(T) \in \omega \Sp_0(2n,\C)}.
		\]
		
	Consider now two square matrices $M_1$ and $M_2$ of sizes $2m_1 \times 2m_1$ and $2m_2 \times 2m_2$ respectively (with $m_1, m_2 \in \N \setminus \{0\}$) such that they can both be
	written in the form
			\[
				M_k \= \begin{pmatrix}
						A_k & B_k \\
						C_k & D_k
					\end{pmatrix}, \qquad  k= 1, 2,
			\]
		each block being of size $m_k \times m_k$. The \emph{diamond product} of $M_1$ and $M_2$ is defined (see~\cite[page~17]{MR1898560}) as the following
		$2(m_1 + m_2) \times 2(m_1 + m_2)$ matrix:
			\begin{equation} \label{def:diamondproduct}
				M_1 \diamond M_2 \= \begin{pmatrix}
									A_1 & 0 & B_1 & 0 \\
									0 & A_2 & 0 & B_2 \\
									C_1 & 0 & D_1 & 0 \\
									0 & C_2 & 0 & D_2
								\end{pmatrix}.
			\end{equation}
		The $k$-fold diamond product of $M$ with itself is denoted by $M^{\diamond k}$. The \emph{symplectic sum} of two paths $\gamma_j \in \mathscr P_T(2n_j)$, with $j=1,2$ and
		$n_1, n_2 \in \N \setminus \{0\}$, is defined in a natural way:
			\[
				(\gamma_1 \diamond \gamma_2)(t) \= \gamma_1(t) \diamond \gamma_2(t), \quad \forall\, t \in [0, T].
			\]
	
	%\red
	{	
	We list the basic properties of the geometric Maslov-type index that we need in the paper.
	\begin{enumerate}[(i)]
		\item (Path additivity) Let $\gamma:[a,b] \to \Sp(2n, \C)$ and $c \in [a,b]$. Then
				\[
					\igeo(\gamma)= \igeo(\gamma|_{[a,c]})+ \igeo(\gamma|_{[c,b]}).
				\]
		\item ($\diamond$-additivity) Let $\gamma_1: [a,b] \to \Sp(2k, \C)$ and $\gamma_2:[a,b]\to \Sp(2l, \C)$ be two symplectic paths. Then we have
				\[
					\igeo(\gamma_1 \diamond \gamma_2)= \igeo(\gamma_1) + \igeo(\gamma_2).
				\]
	 	\item (Homotopy invariance) For any two paths $\gamma_1$ and $\gamma_2$, if $\gamma_1$ is homotopic to $\gamma_2$ (written $\gamma_1 \sim \gamma_2$) in $\Sp(2n,\C)$ with
		 	either fixed or always non-degenerate endpoints, there holds
				\[
					\igeo(\gamma_1) = \igeo(\gamma_2).
				\]
		\item (Normalisation) If $n=1$ then
				\[
					\igeo(e^{it} I, t \in [0,a])= \begin{cases}
										1 & \text{if $a \in (0,2\pi)$}\\
										0 & \text{if $a=2\pi$}.
									\end{cases}
				\]
		\item (Affine scale invariance) For all $k>0$ and $\gamma \in \mathscr{P}_{kT}(2n)$ we have
				\[
					\igeo\big(\gamma(kt), t \in [0,\tau] \big) = \igeo\big(\gamma(t), t \in [0,k\tau]\big).
				\]
	\end{enumerate}
	}

	\subsection{The $\omega$-index theory and the iteration formula}\label{subsec:omega_Maslov}
	
	For any two continuous paths $\gamma, \delta : [0, T] \to \Sp(2n,\C)$ such that $\gamma(T) = \delta(0)$,
	we define their \emph{concatenation} $\gamma * \delta : [0, T] \to \Sp(2n, \C)$ as
		\[
			(\gamma * \delta)(t) \= \begin{cases}
								\gamma(2t) & \text{if } 0 \leq t \leq \frac{T}{2} \\
								\delta(2t - T) & \text{if } \frac{T}{2} \leq t \leq T.
							\end{cases}
		\]
	
	For any $n \in \N \setminus \{0 \}$ we also define a special continuous symplectic path $\xi_n : [0, T] \to \Sp(2n)$ as follows:
		\begin{equation} \label{eq:xin}
			\xi_n(t) \= \begin{pmatrix}
						2 - \dfrac{t}{T} & 0 \\
						0 & \biggl( 2 - \dfrac{t}{T} \biggr)^{-1}
					\end{pmatrix}^{\diamond n}, \qquad \forall\, t \in [0, T].
		\end{equation}

	\begin{defn}[{\cite{MR1674313,MR2525637}}] \label{def:indicediMaslov}
		Let $\omega \in \U$. If $\gamma \in \mathscr{P}_T(2n)$, we define 
			\[
				\nu_\omega(\gamma) \= \dim_\C\ker_\C \big(\gamma(T)-\omega I_{2n} \big).
			\]
		If $\gamma \in \mathscr P_{T,\omega}^*(2n)$  the  \emph{$\omega$-index} is defined as
			\begin{equation} \label{eq:indiceomega}
				i_\omega(\gamma) \= \bigl[  \overline \omega\gamma * \xi_n : \overline{\Sp_1(2n,\C)}\bigr].
			\end{equation}
		If $\gamma \in \P_T(2n) \setminus \P_{T, \omega}^*(2n)$, we let $\mathscr{F}(\gamma)$ be the set of all open neighbourhoods $U$ of $\gamma$ in $\P_T(2n)$, and define
			\[
				i_\omega(\gamma) \= \sup_{U \in \mathscr{F}(\gamma)} \inf \Set{i_\omega(\delta) | \delta \in U \cap \P_{T,\omega}^*(2n)}. 
			\]
		%\red
		{
		Finally the {\em $\omega$-geometric Maslov index\/} is defined as			
			\begin{equation} \label{eq:indiceomegageo}
				\igeomega(\gamma) \= \bigl[ \overline \omega\gamma  e^{-\varepsilon J} : \overline{\Sp_1(2n,\C)}\bigr],
			\end{equation}
		}
		The right-hand side of \eqref{eq:indiceomega} %\red
		{and \eqref{eq:indiceomegageo}} is the usual homotopy intersection number, the orientation of $\overline \omega\gamma* \xi_n$ is 
		its positive time direction under homotopies with fixed end-points and $\varepsilon$ is a positive real number sufficiently small. 
	\end{defn}
	
	We list the basic properties of the $\omega$-index that we need in the sequel. 
	\begin{enumerate}[(i)]
		\item (Lower semicontinuity) For all $\gamma: [a,b] \to \mathscr P_T(2n)$ and $c \in [a,b]$ we have
				\[
					i_\omega(\gamma) = \inf\Set{i_\omega(\beta)| \beta \in \mathscr P^*_T(2n)\  \text{is sufficiently $\mathscr C^0$-close to $\gamma$}}.
				\]
		\item ($\diamond$-additivity) Let $\gamma_1: [a,b] \to \Sp(2k, \C)$ and $\gamma_2:[a,b]\to \Sp(2l, \C)$ be two symplectic paths. Then we have
				\[
					i_\omega(\gamma_1 \diamond \gamma_2)= i_\omega(\gamma_1) + i_\omega(\gamma_2).
				\]
		\item (Homotopy invariance) For any two paths $\gamma_1$ and $\gamma_2$, if $\gamma_1 \sim \gamma_2$ in $\Sp(2n,\C)$ with either fixed or always non-degenerate endpoints,
			there holds
				\[
					i_\omega(\gamma_1)=i_\omega(\gamma_2).
				\]
		\item (Affine scale invariance) For all $k>0$ and $\gamma \in \mathscr{P}_{kT}(2n)$, we have
				\[
					i_\omega\big(\gamma(kt), t \in [0,T]\big)= i_\omega\big(\gamma(t), t \in [0,kT]\big).
				\]
	\end{enumerate}	
	
	The proofs of these properties are consequences of \cite[Lemma~2.2 (3), Corollary~2.1, Theorem~2.1]{MR1762278} and of the index theory contained in \cite{MR1674313}.
	
	Let $\gamma \in \P_T(2n)$ and $m \in \N \setminus \{0\}$. The \emph{$m$-th iteration of $\gamma$} is $\gamma^m : [0, mT] \to \Sp(2n)$ defined as
		\[
			\gamma^m(t) \= \gamma(t - jT) \bigl( \gamma(T) \bigr)^j, \qquad \text{for } j T \leq t \leq (j + 1)T, \quad  j = 0, \dotsc, m - 1.
		\]
	The next Bott-type iteration formula is crucial in order to study the geometric multiplicity of periodic orbits and plays a big role in the question of linear stability.
	
	\begin{lem}[Bott-Long iteration formula, {\cite[Theorem~9.2.1]{MR1898560}}] \label{thm:bott}
		For any $z \in \mathbb U$, $\gamma \in {\mathscr P}_T(2n)$ and $m \in \N \setminus \{0\}$ the following formula holds:
			\begin{equation} \label{eq:Bott}
				i_z(\gamma^m) = \sum_{\omega^m=z} i_\omega(\gamma).
			\end{equation}
		In particular one has $i_1(\gamma^2) = i_1(\gamma) + i_{-1}(\gamma)$.
	\end{lem}

	%	--------------------------------------------------------------------------------------------------------------
		\subsection{Morse index of paths of Lagrangian subspaces and relation with other Maslov-type indices} \label{subs:otherMaslov}
	%	--------------------------------------------------------------------------------------------------------------

	Let $(\C^{2n}, \{\cdot, \cdot\})$ be the complex symplectic space whose complex symplectic structure can be represented through the Hermitian product $(\cdot, \cdot)$ as
		\[
			\{v,w\} \= (J v, w), \qquad \forall\, v, w \in \C^{2n}.
		\]
	We denote by $\Lag(\C^{2n})$ the space of all Lagrangian subspaces in $\C^{2n}$.
	 
	Let $l:[a,b] \to \Lag(\C^{2n})$ be a $\mathscr C^1$-curve of Lagrangian subspaces and let $L_0 \in \Lag(\C^{2n})$. Fix $t \in [a,b]$ and let $W$ be a fixed Lagrangian complement of $l(t)$.
	If $s$ belongs to a suitable small neighborhood of $t$ for every $v \in l(t)$ we can find a unique vector $w(s)\in W$ in such a way that $v + w(s) \in l(s)$.
	
	\begin{defn}
		The \emph{crossing form} $\Gamma(l, L_0, t^*)$ at $t^*$ is the quadratic form $\Gamma(l, L_0, t^*) : l(t^{*})\cap L_0 \to \R$ defined by
			\begin{equation}\label{eq:crossingform}
				\Gamma(l, L_0, t^*)[v] \= \dfrac{d}{ds}\{v,w(s)\}\big\vert_{s=t^*}.
			\end{equation}
		The number $t^*$ is said to be a \emph{crossing instant for $l$ with respect to $L_0$} if  $l(t^*) \cap L_0 \neq \{0\}$ and it is called \emph{regular} if the crossing form is non-degenerate.
	\end{defn}  
	
	Let us remark that regular crossings are isolated and hence on a compact interval they are finitely many. Following \cite[Definition~3.1, Theorem~3.1]{MR1762278} we give the next definition.
	
	\begin{defn}
	     If $l$ has only regular crossings 
		 with respect to $L_0$, the {\em Maslov index of $l$ with respect to $L_0$} is defined as
		 	\begin{equation}\label{eq:CLMmaslovestremideg}
		 		\iclm(L_0, l, [a,b]) \= m^+\big(\Gamma(l, L_0, a) \big) + \sum_{t^* \in (a,b)} \sgn \Gamma(l,L_0,t^*) - m^- \big(\Gamma(l,L_0,b) \big)
		 	\end{equation}
		 where the summation runs over all crossings $t^* \in (a,b)$, the symbols $m^+, m^-$ denote the dimension of the positive and negative spectral subspaces respectively and
		 $\sgn \= m^+ - m^-$ is the signature.
	\end{defn}
	Let $V \= \C^{2n} \oplus \C^{2n}$, and $(\cdot, \cdot)$ be the standard Hermitian product of $V$. We define
		\[
			\{v, w\}_{\mathcal{J}} \= (\mathcal J v, w), \qquad \forall\, v, w \in V
		\]
	where
		\[
			\mathcal{J} \= \begin{pmatrix}
						-J &0\\
						0 & J
					\end{pmatrix}.
		\]
	By a direct calculation it follows that if $M \in \Sp(2n, \C)$ then the complex subspace
		\[
			\Gr(M) \= \Set{\begin{pmatrix} x\\Mx\end{pmatrix}|x \in \C^{2n}}
		\]
	is a Lagrangian subspace of the (complex) symplectic space $(V, \{\cdot, \cdot\}_{\mathcal{J}})$.
	
	Given a path of symplectic matrices $\gamma:[a,b] \to \Sp(2n,\C)$ then the \emph{graph of the path $\gamma$}, $\Gr(\gamma)$, is defined as the path of graphs:
	$\Gr(\gamma)(t) \= \Gr(\gamma(t))$, $t \in [a,b]$, and it is indeed a path of Lagrangian subspaces of $(V, \{\cdot, \cdot\}_{\mathcal{J}})$.
	The next result gives the relationship between the geometric index of a path of symplectic matrices and the Maslov index of the corresponding path of Lagrangian subspaces with respect to
	the diagonal $\Delta \= \Gr(I_{2n})$.
	
	%\red
	{
	\begin{prop}
		For all path $\gamma:[a,b] \to \Sp(2n,\C)$ we have
			\[
				\igeo(\gamma) = \iclm(\Delta, \Gr( \gamma), [a, b]),
			\]
		where the crossing forms involved in the right-hand side are calculated using the symplectic structure $\{\cdot,\cdot\}_{\mathcal{J}}$ in $\C^{2n} \oplus \C^{2n}$.
	\end{prop}
	\begin{proof}
	 	The proof of this proposition follows by \cite[Formula~(3.4) in Proposition~3.1, Theorem~3.1 (ii), Definition~3.1]{MR1762278} and by \cite[Proposition~4.1]{MR1263126}.
	\end{proof}
	}
	 	
	By \cite[Lemma~4.6, Formulas~(2.15)--(2.16)]{MR2525637} we immediately obtain
	
	\begin{lem} \label{thm:chiave}
		For any path $\gamma \in \mathscr{P}_T(2n)$ we have the following equalities:
		\begin{enumerate}
			\item $i_1(\gamma) + n = \iclm\big(\Delta, \Gr(\gamma), [0,T]\big)$;
	 	 	\item $i_\omega(\gamma) = \iclm\big(\Gr(\omega I_{2n}), \Gr(\gamma), [0,T]\big)$ for all $\omega \in \mathbb{U}\setminus\{1\}$.
	 	\end{enumerate}
	\end{lem} 
	
	\begin{rmk}
		We observe that the integer $i_1$ is sometimes denoted by $i_{\text{CZ}}$ and it is called the \emph{Conley-Zehnder index}. For further details we refer to \cite{MR1762278} and references
		therein.
 	\end{rmk}
	 	
	We now show some examples of computation of $i_1(\gamma)$ passing through $\iclm$ of some paths of matrices in $\Sp(2) \subset \Sp(2,\C)$.
	Let $\gamma:[a,b] \to \Sp(2)$ be the path 
	 	\[
	 		\gamma(t) \= \begin{pmatrix}
	 					a(t) & b(t)\\ 
	 					c(t) & d(t)
	 				   \end{pmatrix},
	 	\]
	with $a,b,c,d \in \mathscr{C}^1([a,b],\R)$, let $l$ be the induced path of Lagrangian subspaces in $\R^4$ defined by $l(t) \= \Gr\big(\gamma(t)\big)$.
	Let us assume that $t^*$ is a crossing instant for $l$ such that $l(t^*)=\Delta$. 
	In order to compute the crossing form \eqref{eq:crossingform} we consider the Lagrangian subspace complementary to $\Delta$:
		\[
			W \= \{0\}\times \R \times \R \times \{0\}.
		\] 
	Thus the Lagrangian splitting $\R^4 = \Delta \oplus W$ holds and for any $v \= (x_0,y_0,x_0,y_0) \in \Delta$ let us choose $w(t) \= (0,\eta(t),\xi(t),0)\in W$ in order that $v + w(t) \in l(t)$. This
	means that $\eta(t)$ and $\xi(t)$ solve the equations
		\begin{equation}
 			x_0 + \xi(t) = a(t) x_0 + b(t)\big(y_0 + \eta(t)\big),
			\qquad
 			y_0 = c(t) x_0 + d(t)\big(y_0 + \eta(t)\big).
		\end{equation}
	Since in a crossing instant $t^*$ we have $\xi(t^*)=\eta(t^*)=0$, differentiating the above identities gives
	 	\begin{subequations}\label{eq:lederivate}
	 		\begin{align}
	 			\xi'(t^*) &= a'(t^*) x_0 + b'(t^*)y_0 - \frac{b(t^*)}{d(t^*)}\big[c'(t^*)x_0+d'(t^*)y_0\big],\\
	 			\eta'(t^*) &= - \frac{1}{d(t^*)}\big[c'(t^*)x_0+d'(t^*)y_0\big].
	 		\end{align}		
	 	\end{subequations}
	By a direct computation we obtain
		\[
	 	\begin{split}
	 		\{v,w(t)\}_{\mathcal J} & = \big(\mathcal J v, w(t) \big) \\
							& = -\left\langle J \begin{pmatrix}
									x_0\\ y_0
								\end{pmatrix},
								\begin{pmatrix}
									0\\ \eta(t)
								\end{pmatrix}\right\rangle
								+\left\langle J \begin{pmatrix}
									x_0\\y_0
								\end{pmatrix},  \begin{pmatrix}
	 								\xi(t)\\0
								\end{pmatrix}\right\rangle\\
							& = -x_0 \eta(t) - y_0 \xi(t).
		\end{split}
	 	\]
	Hence the crossing form at the crossing instant $t=t^*$ is given by
		\begin{equation}\label{eq:crossingrot}
			\Gamma(l,\Delta,t^*)(v) = \dfrac{d}{dt} \{v,w(t)\}_{\mathcal J}\Big\vert_{t=t^*} = -x_0 \eta'(t^*) - y_0 \xi'(t^*).
		\end{equation}
	 				
	\begin{ex} \label{ex:Ralpha}
		Let us consider the path $R_\alpha : [0,2\pi] \to \Sp(2)$ with
	 		\[
	 		R_\alpha(t)=
	 		\begin{pmatrix}
	 			\cos(\sqrt{2 - \alpha}\,t) & - \sqrt{2 - \alpha}\, \sin(\sqrt{2 - \alpha}\,t) \\
	 			\frac{1}{\sqrt{2 - \alpha}}\sin(\sqrt{2 - \alpha}\, t) & \cos(\sqrt{2 - \alpha}\,t)
	 		\end{pmatrix}
	 		\qquad \alpha \in (0,2),
	 		\]
	 	that means $a=0$, $b=2\pi$, and
	 		 \begin{gather*}
	 		 a(t) = d(t) = \cos(\sqrt{2 - \alpha}\,t)  \\
	 		 b(t) = - \sqrt{2 - \alpha} \sin(\sqrt{2 - \alpha}\,t), \qquad 
	 		 c(t) = \frac{1}{\sqrt{2 - \alpha}}\sin(\sqrt{2 - \alpha}\, t)
	 		 \end{gather*}
	 	For any value of the parameter $\alpha$, $t^*=0$ is a crossing instant and $a(0)=1$, $a'(0)=0$, $b(0)=0$, $b'(0)=-(2-\alpha)$, $c(0)=0$, $c'(0)=1$.
		Using Equations~\eqref{eq:lederivate} we get
	 		\begin{equation}\label{eq:crossingrotfin}
	 			\Gamma(l_\alpha,\Delta,0)[v] = -x_0 \eta'_\alpha(0)-y_0 \xi_\alpha'(0) = x_0^2+(2-\alpha)y_0^2
		 	\end{equation}
	 	where $l_\alpha$ is the path of Lagrangian subspaces associated with $R_\alpha$. Since $\Gamma(l_\alpha,\Delta,0)$ is a positive definite quadratic form, its signature is $2$. 
	 	Thus, according to Formula~\eqref{eq:CLMmaslovestremideg}, the contribution to $\iclm$ at the starting point of the path is $2$.
	 					
	 	In order to find out all the crossing instants, we observe that they are in one-to-one correspondence with the zeros of the function $\det\big(R_\alpha(t)-I_{2}\big)$, and hence with the
		solutions in $[0,2\pi]$ of the equation
	 		\begin{equation} \label{eq:crossingsrot}
	 			\cos(\sqrt{2-\alpha}\,t)=1,
	 		\end{equation}
	 	that we write as $t^\alpha_k \= {2k\pi}/\sqrt{2-\alpha}$, with $k \in \Z$. It is readily seen that 
	 	\begin{itemize}
	 	\item if $\alpha \in (1,2)$ then the only solution of \eqref{eq:crossingsrot} is $t^\alpha_0=0$, hence there are no other contributions to the computation of $\iclm$.
	 	\item if $\alpha =1$ then we have two solutions: $t^1_0=0$ and $t^1_1=2\pi$. We need to add $m^-\big(\Gamma(l_\alpha, \Delta, 2\pi)\big)$ to the contribution of the initial instant, but this
			quantity is actually $0$.
	 	\item if $\alpha \in (0,1)$ then \eqref{eq:crossingsrot} admits also the non-zero solution%
			\footnote{We observe that this value coincides with the apsidal angle for the $\alpha$-homogeneous potential.}
			$t^\alpha_1 = \dfrac{2\pi} {\sqrt{2-\alpha}}$. The contribution of this crossing is $\sgn \Gamma(l_\alpha, \Delta, t^\alpha_1) = 2$. 
	 	\end{itemize}
	 	Summing up all these computations we obtain
	 	\begin{equation}\label{eq:CLM_R_alpha}
	 		 \iclm(R_\alpha) =
	 		 \begin{cases}
	 		 2 & \text{if }\alpha \in [1,2)\\
	 		 4 & \text{if }\alpha \in [0,1).
	 		 \end{cases}	
	 	\end{equation}
	 	\end{ex} 
	 			
	 	\begin{ex}\label{ex:Nalpha} 
	 	We now consider the path $N_\alpha : [0,2\pi] \to \Sp(2)$ with
	 	\[
	 	N_\alpha (t) = \begin{pmatrix}
	 		 			1 & 0 \\
	 		 			f_\alpha(t) 
	 		 			& 1
	 		 			\end{pmatrix}
	 	\]
		where the function
	 	\[
	 	 f_\alpha(t):= \frac{1}{36 \pi^2} \biggl( \frac{4 \sin(\sqrt{2 - \alpha}\, t)}{(2 - \alpha)^{3/2}} - \frac{2 + \alpha}{2 - \alpha}t \biggr)
	 	\]
		is drawn in Figure~\ref{fig:alpha} for $\alpha = 1$ (the other cases for different $\alpha$'s are all similar).

	\begin{figure}[tb]
	\centering
	\subfloat[][The function $f_1(t)$ in the interval ${[0, 2\pi]}$.\label{fig:alpha}]{\includegraphics[width=0.45\textwidth]{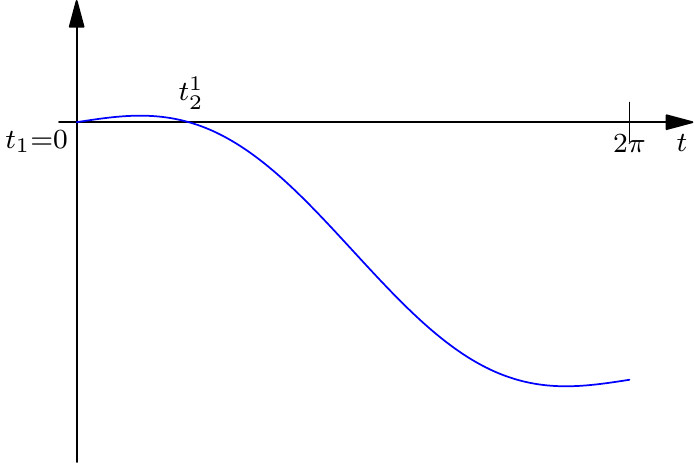}} \qquad
	\subfloat[][The function $2 - 2\cos\eps - f_1(t)\sin\eps$ in the interval ${[0, 2\pi]}$.\label{fig:falphadeformata}]{\includegraphics[width=0.45\textwidth]{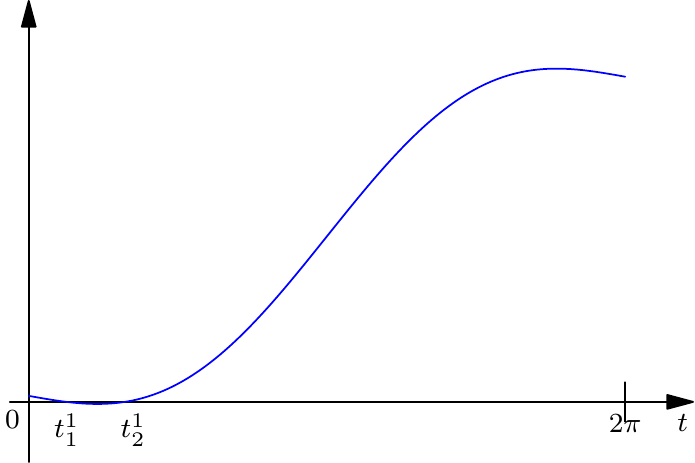}}
	\caption{The function $f_\alpha$ (a) and its deformation (b) for $\alpha = 1$.}
	\end{figure}

%	\begin{figure}[tb]
%	\centering
%		\begin{asy}
%		import graph;
%
%		size(200, 200*2/3, IgnoreAspect);
%
%		real falpha(real t) {return 1/(36*pi^2)*(4*sin(t) - 3*t);}
%		real x(real t) {return t;}
%
%		xaxis("$\scriptstyle t$", -0.2, 7, EndArrow);
%		yaxis(-0.07, 0.025, EndArrow);
%		draw(graph(x, falpha, 0, 2*pi, operator ..), blue);
%
%		label("$\scriptstyle t_1 = 0$", (0, 0), SW);
%		xtick(2*pi);
%		xtick(2*pi, S);
%		label("$\scriptstyle t_2^1$", (1.3, 0), N);
%		label("$\scriptstyle 2\pi$", (2*pi, 0), S);
%		\end{asy}
%		\caption{The function $f_1(t)$ in the interval $[0, 2\pi]$.} \label{fig:alpha}
%	\end{figure}
		
	 	We first observe that we are in a very degenerate situation, in the sense that $N_\alpha(t) \subset \overline{\Sp_1(2)}$.
	 	Furthermore, the function $f_\alpha$ admits two zeros in the interval $[0,2\pi]$, $t_1=0$ and $t_2^\alpha\in(0,2\pi)$. Thus the path is not contained in a fixed stratum of the Maslov cycle.
	 				
	 	However, by taking into account the very definition of the Maslov index in the degenerate case given in Definition~\ref{def:indicediMaslov}, we need to compute the contributions of the
		crossing of the graph of the perturbed matrix
	 	\begin{equation*}
	 		N_{\varepsilon,\alpha}(t):= N_\alpha(t)\, e^{-\eps J} \qquad \forall\, t\in [0,2\pi] 
	 		\end{equation*}
	 	and for $\varepsilon >0$ sufficiently small. By a direct computation we get:
	 		\begin{equation}\label{eq:nepsilonalpha}
	 		 N_{\varepsilon,\alpha}:= \begin{pmatrix}
	 			\cos\eps & \sin\eps \\
	 			-\sin\eps + f_\alpha(t) \cos\eps
	 			& \cos\eps + f_\alpha(t) \sin\eps
	 			\end{pmatrix}
	 		\end{equation}
	 	The crossing instants are the zeros of the equation: 
	 		\begin{equation}\label{eq:crossingnalpha}
	 		 2 - 2\cos\eps - f_\alpha(t)\sin\eps = 0.
	 		\end{equation}
		The function whose zeros we are searching is depicted in Figure~\ref{fig:falphadeformata}.
		
%	\begin{figure}[tb]
%		\centering
%		\begin{asy}
%		import graph;
%
%		size(200, 200*2/3, IgnoreAspect);
%
%		real eps=0.001;
%	
%		real falpha(real t) {return 1/(36*pi^2)*(4*sin(t) - 3*t);}
%		real g(real t) {return 2 - 2*cos(eps) - falpha(t)*sin(eps);}
%		real x(real t) {return t;}
%
%		xaxis("$\scriptstyle t$", -0.2, 7, EndArrow);
%		yaxis(-eps/100, eps/15, EndArrow);
%		draw(graph(x, g, 0, 2*pi, operator ..), blue);
%
%		label("$\scriptstyle 0$", (0, 0), SW);
%		label("$\scriptstyle t^1_1$", (0.4, 0), S);
%		label("$\scriptstyle t^1_2$", (1.1, 0), S);
%		xtick(2*pi);
%		xtick(2*pi, S);
%		label("$\scriptstyle 2\pi$", (2*pi, 0), S);
%		\end{asy}
%		\caption{The function $2 - 2\cos\eps - f_1(t)\sin\eps$ in the interval $[0, 2\pi]$.} \label{fig:falphadeformata}
%	\end{figure}	
	
	 	It is easy to see that for $\varepsilon$ sufficiently small and for any $\alpha \in (0,2)$ this equation 
	 	admits two distinct solutions $t^\alpha_1$ and $t^\alpha_2$ in $(0, 2\pi)$.
	 				
	 	Denoting by $t^\alpha$ a generic solution (crossing) we easily compute
	 		\begin{align*}
	 			\eta'_\alpha(t^\alpha) &= - f'_\alpha(t^\alpha) x_0 \\
	 			\xi'_\alpha(t^\alpha) &= 0,
	 		\end{align*}
	 	whence
	 		\[
	 		\Gamma(N_{\eps, \alpha}, \Delta, t^\alpha) = f'_\alpha(t^\alpha) x_0^2.
	 		\]
		Summing up the two contributions, from the monotonicity of $f_\alpha$ we immediately obtain that $\iclm(N_\alpha, \Delta, [0, 2\pi]) = 0$. The path $N_\alpha$ and its deformation
		$N_{\eps, \alpha}$ are represented in Figure~\ref{fig:Nalpha}.
	 	\end{ex}
		
	\begin{figure}
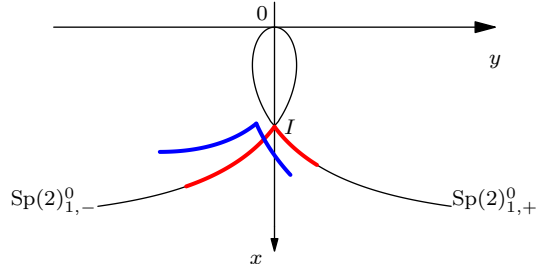

		\begin{center}
		\begin{asy}
			import graph;
		
			size(200, 100, IgnoreAspect);
			
			real angle=-0.25;
			
		//	Curva di destra
			real x1(real r) {return r*sqrt(1 - 4*r^2/(1 + r^2)^2);}
			real y1(real r) {return -2*r^2/(1 + r^2);}
		
		//	Curva di sinistra
			real x2(real r) {return -r*sqrt(1 - 4*r^2/(1 + r^2)^2);}
			real y2(real r) {return -2*r^2/(1 + r^2);}					
		
		//	Curva di destra ruotata
			real x3(real r) {return cos(angle)*x1(r) - sin(angle)*y1(r);}
			real y3(real r) {return sin(angle)*x1(r) + cos(angle)*y1(r);}
			
		//	Curva di destra ruotata
			real x4(real r) {return cos(angle)*x2(r) - sin(angle)*y2(r);}
			real y4(real r) {return sin(angle)*x2(r) + cos(angle)*y2(r);}
		
		//	Assi coordinati
			xaxis(xmin = -3, xmax = 3, arrow=EndArrow);
			yaxis(ymin = -2, ymax=0.25);

			draw(graph(x1, y1, 0, 3, operator ..));
			draw(graph(x2, y2, 0, 3, operator ..));
			draw(graph(x1, y1, 1, 1.5, operator ..), red+linewidth(1.5));
			draw(graph(x2, y2, 1, 2, operator ..), red+linewidth(1.5));
			draw(graph(x3, y3, 1, 1.5, operator ..), blue+linewidth(1.5));
			draw(graph(x4, y4, 1, 2, operator ..), blue+linewidth(1.5));
			draw((0,-2)--(0,-2.25), arrow = EndArrow);
		
		//	Etichette
			labelx("$\scriptstyle y$", 3);
			label("$\scriptstyle 0$", (0, 0), NW);
			label("$\scriptstyle I$", (0, -1), E);
			labely("$\scriptstyle x$", -2.25);
			label("$\scriptstyle\mathrm{Sp}(2)_{1,+}^0$", (3,-1.75));
			label("$\scriptstyle\mathrm{Sp}(2)_{1,-}^0$", (-3,-1.75));
		\end{asy}
		\caption{The path $N_\alpha$ (in red) and its deformation $N_{\eps, \alpha}$ (in blue). The first path starts at the identity, then goes downwards right, then comes back to the
			identity and finally bends downwards left. The second follows the same trajectory, just rotated clockwise by an angle $\eps$. See Appendix~\ref{app:SP2} for more details about the
			coordinates and the underlying curves.} \label{fig:Nalpha}
		\end{center}	 		
	\end{figure}

	%	-------------------------------------------------------------------------------------
		\subsection{Computation of the Maslov index via Krein signature and splitting numbers}
	%	-------------------------------------------------------------------------------------
	In the case of autonomous Hamiltonian systems and under the assumption of non-degeneracy it is possible, at least theoretically, to compute the Maslov index (see for instance
	\cite{MR1824111} and references therein). Let $M \in \Sp(2n, \R)$ act on $\C^{2n}$ in the usual way:
		\[
			M(\xi + i\eta) \= M\xi + i M\eta, \qquad \forall\, \xi,\eta \in \R^{2n}
		\]
	and consider the Hermitian form $g$ on $\C^{2n}$ defined as
		\[
			g(v,w) \= ( iJ v, w ).
		\]
	
	\begin{defn} \label{def:kreinsign}
		Let $\lambda \in \U$ be an eigenvalue of a complex symplectic matrix. The \emph{Krein signature} of $\lambda$ is the signature of the restriction of the Hermitian
		form $g$ to the generalised eigenspace associated with $\lambda$. If $g$ is positive definite on this subspace then $\lambda$ is said to be \emph{Krein-positive}.
	\end{defn}
	
	The next result will be useful in the following.
	
	\begin{prop}[{\cite[Theorem~1.5.1]{MR1824111}}] \label{prop:MaslovAbbo}
	Let $B$ be a real symmetric matrix. Let $i \theta_1, \dotsc, i\theta_k$ be the Krein-positive purely imaginary eigenvalues of $JB$, counted with their algebraic multiplicity.
	Then the linear autonomous Hamiltonian system
		\[
		\zeta'(t) = JB \zeta(t)
		\]
	is non-degenerate at time $T$ if and only if $\theta_j T \notin 2\pi\Z$, for any $j = 1, \dotsc, k$. If $\psi$ denotes the fundamental solution, we get:
		\[
		i_1(\psi) = - \sum_{j=1}^k \bigg[\bigg[ \frac{T \theta_j}{\pi} \bigg]\bigg]
		\]
	provided that it is non-degenerate at time $T$. The function $[[\, \cdot\, ]]$ is defined as follows:
		\[
		  [[ \theta ]] \= 
		  \begin{cases}
			\theta & \text{if } \theta \in \Z \\
			\text{the closest odd integer} & \text{if } \theta \in \R \setminus \Z.
		\end{cases}
		\]
	\end{prop}	
     	
	Now, following \cite{MR1898560}, we  recall the definition of the so-called \emph{splitting numbers} as well as their basic properties, which will be crucial later. For this we refer to
	\cite[Chapter~6, pages~190--199]{MR1898560}.
	
	\begin{defn} \label{def:splittingnumbers}
		For any $M \in \Sp(2n)$ and every $\omega \in \U$, the \emph{splitting numbers} $S^\pm_M(\omega)$ of $M$ at $\omega$ are defined by
			\begin{equation} \label{eq:splittingnumbers}
				S_M^\pm(\omega) \= \lim_{\eps \to 0^+} i_{\omega \exp (\pm i \eps)}(\gamma) - i_\omega(\gamma),
			\end{equation}
		where $\gamma \in \mathscr{P}_T(2n)$ is such that $\gamma(T) = M$.
	\end{defn}
	
	In the next proposition we recall the basic properties of the splitting numbers. For their computation we introduce 
	the \emph{normal forms}
	\begin{equation*} % \label{eq:normalforms}
			R(\theta) \= \begin{pmatrix}
						\cos \theta & -\sin \theta \\
						\sin \theta & \cos \theta
						\end{pmatrix}, \; 
					\theta \in (0,2\pi)\setminus \{0\},
							\qquad
			N_1(\lambda, a) \= \begin{pmatrix}
							\lambda & a \\
							0 & \lambda
							\end{pmatrix}, \;
					\lambda \in \R^*, \, a \in \R.
	\end{equation*}

	\begin{prop}[{\cite[Chapter~6]{MR1898560}}] \label{thm:splittingprop}
		For $M, M_0, M_1 \in \Sp(2n)$ and all $\omega \in \U$, $\theta \in (0, \pi)$, the following properties hold:
			\begin{enumerate}
				\item\label{1} The splitting numbers $S^\pm_M(\omega)$ are well defined, \ie they are independent of the choice of the path $\gamma \in \P_\tau(2n)$ satisfying
							$\gamma(\tau) = M$ in Definition~\eqref{eq:splittingnumbers}.
				\item\label{2} The splitting numbers $S^\pm_M(\omega)$ are constant in the set $\Omega^0(M)$, that is the path-connected component containing $M$ of the set
					\[
						\Omega(M) \= \Set{ N \in \Sp(2n) | \sigma(N) \cap \U = \sigma(M) \cap \U \textup{ and } \nu_\lambda(N) = \nu_\lambda(M)\ \forall\, \lambda \in \sigma(M) \cap \U }.
					\]
				\item\label{3} $S^\pm_M(\omega) = 0$ if $\omega \notin \sigma(M)$.
				\item\label{4} $S^\pm_M(\overline{\omega}) = S^\mp_M(\omega)$.
				\item\label{5} $0 \leq S^\pm(\omega) \leq \dim\ker(M - \omega I)$.
				\item\label{6} $S^+_M(\omega) + S^-_M(\omega) \leq \dim\ker(M - \omega I)^{2n}$ if $\omega \in \sigma(M)$.
				\item\label{7} $S^\pm_{M_0 \diamond M_1}(\omega) = S^\pm_{M_0}(\omega) + S^\pm_{M_1}(\omega)$.
				\item\label{formula} $i_\omega(\gamma) - i_1(\gamma) = S_{M}^+(1) + \sum_{\omega_0} \bigl( S_{M}^+(\omega_0) - S_{M}^-(\omega_0) \bigr) - S_{M}^-(\omega)$, where
							$\Im(\omega) \geq 0$ and $\omega_0 \in \sigma(M)$ lies in the interior of the arc of the upper unit semicircle connecting $1$ and $\omega$.
				\item\label{8} $\Bigl( S^+_{N_1(1,a)}(1),\ S^-_{N_1(1,a)}(1) \Bigr) = \begin{cases}
																	(1,1) & \text{if } a \in \{ 0, 1 \} \\
																	(0,0) & \text{if } a = -1.
																\end{cases}$
				\item\label{9} $\Bigl( S^+_{N_1(-1,a)}(-1),\ S^-_{N_1(-1,a)}(-1) \Bigr) = \begin{cases}
																		(1,1) & \text{if } a \in \{ -1, 0 \} \\
																		(0,0) & \text{if } a = 1.
																	\end{cases}$
				\item\label{10} $\Bigl( S^+_{R(\theta)}(e^{i\theta}),\ S^-_{R(\theta)}(e^{i\theta}) \Bigr) = (0,1)$.
				\item\label{11} $\Bigl( S^+_{R(2\pi - \theta)}(e^{i\theta}),\ S^-_{R(2\pi - \theta)}(e^{i\theta})\Bigr) = (1,0)$.
		\end{enumerate}
	\end{prop}
	
%
%	-----------------------------------------
	\section{Variational setting: an index theorem} \label{sec:variational}
%	-----------------------------------------
	
	%\red
	{
	We recall here some basic facts about the Lagrangian and Hamiltonian dynamics (for further details see 
	for instance \cite{Fat08, MR2300670, MR2465553}). The elements of the tangent bundle $T\R^n \cong \R^n\times \R^n$ are denoted by $(q,v)$ where $q \in U$ and $v \in T_q U$.
	Let $\L \in \mathscr C^\infty(T\R^n; \R)$ be a \emph{regular Lagrangian}, meaning that $\L$ is assumed to satisfy
	\begin{enumerate}[(L1)]
		\item $\partial_{vv} \L(q,v)>0$ for all $(q,v) \in T\R^n$;
		\item There is a constant $l_1>0$ such that 
				\[
					\norm{\partial_{vv}\L(q,v)} \leq l_1, \quad 
					\norm{\partial_{qv}\L(q,v)} \leq l_1(1+ \abs{v}), \quad \norm{\partial_{qq}\L(q,v)} \leq l_1(1+ \abs{v}^2).
				\]
	\end{enumerate}
	As a direct consequence of the Inverse Function Theorem, under Condition~(L1) the Legendre transformation
		\[
			\mathcal{L}_\L: T\R^n \to T^*\R^n, \qquad (q,v)\mapsto \big(D_v\L(q,v), q\big),
		\]
	is a smooth local diffeomorphism. The Fenchel transformation of $\L$ is the autonomous Hamiltonian on $T^*\R^n$ 
		\[
			\H(p,q) \= \max_{v \in T_q \R^n} \big(p[v]-\L(q,v)\big) = p\big[v(p,q)\big]- \L\big(q,v(p,q)\big),
		\]
	where $\big(q,v(p,q)\big)=\mathcal{L}^{-1}_\L(p,q)$. Under the above assumptions on $\L$, the function $\H$ is smooth on $T^*\R^n$. 
	The associated autonomous Hamiltonian vector field $X_\H$ on $T^*\R^n$, defined by 
		\[
			\big\langle J X_\H(p,q), \xi \big\rangle = -D\H(p,q)[\xi], \quad \forall\, (p,q) \in T^*\R^n,\ \forall\, \xi \in 
			T_{(p,q)} T^*\R^n,
		\]
	is then smooth, so it defines an autonomous smooth local flow on $T^*\R^n$. The corresponding flow on $T\R^n$ 
	obtained by conjugating the Hamiltonian flow
	$\varphi^\H$ by the Legendre transform $\mathcal{L}_\L$ is denoted by 
		\[
			 \varphi^\L: T\R^n \to T\R^n
		\]
	and its orbits have the form $t \mapsto \big(\gamma(t), \gamma'(t) \big)$, where 
	$\gamma $ solves the Euler-Lagrange equation
		\begin{equation}\label{def:EL}
			\dfrac{d}{dt}\partial_v \L\big(\gamma(t), \gamma'(t)\big)= \partial_q \L\big( \gamma(t),\gamma'(t) \big).
		\end{equation}
	}
	
	Let us consider the \emph{Lagrangian action functional} $\mathbb A : W^{1,2}(\R/2\pi\Z, \widehat{X}) \to \R$  defined by 
		\[
			\mathbb A(\gamma) \= \int_0^{2\pi}  \L\big(\gamma(t),\gamma'(t)\big)\,dt.
		\]
	We recall that if $\L$ satisfies (L2) then $\mathbb A$ is of class $\mathscr{C}^2$ (cf.~\cite[Proposition 4.1]{MR2300670}). Moreover if the first variation of $\mathbb A$ vanishes at
	$\gamma \in  W^{1,2}(\R/2\pi\Z, \widehat{X})$ for every $\xi \in W^{1,2}(\R/2\pi\Z, \widehat{X})$, then $\gamma$ is a (classical) solution of class $\mathscr C^2$ of the Euler-Lagrange
	equation~\eqref{def:EL} such that $\gamma(2\pi)=\gamma(0)$. Given a classical solution $\gamma$ of \eqref{def:EL} the second variation of $\mathbb A$ is given by
		\begin{equation}\label{eq:secondvariation}
			\d^2 \mathbb A(\gamma)[\xi,\eta] = \int_0^{2\pi} \Big[ \big( P(t)\xi'+Q(t)\xi \big)\eta' + \trasp{Q}(t)\xi'\eta + R(t) \xi \eta \Big] \,dt,
		\end{equation}
	where
		\[
			P(t) \= D_{vv} \L\big(\gamma(t),\gamma'(t)\big), \quad 
			Q(t) \= D_{qv} \L\big(\gamma(t), \gamma'(t)\big), \quad
			R(t) \= D_{qq} \L\big(\gamma(t), \gamma'(t)\big).
		\]
	Linearising the Euler-Lagrange equations~\eqref{def:EL} around a critical point $\gamma$ we obtain the Sturm system
		\begin{equation}\label{eq:Sturm}
			-\big( P(t)\gamma'(t) + Q(t)\gamma(t) \big)' + \trasp{Q}(t)\gamma'(t)+ R(t)\gamma(t) = 0
		\end{equation}
	Let now $\zeta(t) \= \big(D_v\L(\gamma(t), \gamma'(t),\gamma(t)\big)$ be the solution of the Hamiltonian system associated with \eqref{eq:Sturm}, whose fundamental solution $\phi$ satisfies
		\begin{equation} \label{eq:lin_sys}
			\begin{cases}
				\phi'(t) = JB(t)\phi(t)\\
				\phi(0)=I_{2n},
			\end{cases}
		\end{equation}
	with 
		\[
			B(t) \= \begin{pmatrix}
					P^{-1}(t) & -P^{-1}(t) Q(t)\\
					-\trasp{Q}(t) P^{-1}(t) & \trasp{Q}(t)P^{-1}(t) Q(t)-R(t)
				 \end{pmatrix}.
		\]
	For any $\omega \in \U$ let $h(\gamma)$ be the quadratic form on $D(\omega) \= \Set{\xi \in W^{1,2}\big([0,2\pi], \C^n\big) | \xi(2\pi)=\omega\xi(0)}$ induced by $\d^2 \mathbb A(\gamma)$.
	Then it is possible to show (arguing as in \cite[Proposition 3.1]{MR2133393} for further details) that $h(\gamma)$ is an essentially positive Fredholm quadratic form in the sense specified in
	Appendix~\ref{sec:Fredholmforms}.
		
	\begin{defn}
		Let  $\gamma \in D(\omega) $ be a critical point of $\mathbb A$. We define the \emph{$\omega$-Morse index}  
		of $\gamma$, denoted by $\iMor^\omega(\gamma)$, as the dimension of the largest subspace of $D(\omega)$
		such that the quadratic form $h(\gamma) $
		is negative definite.
	\end{defn}
	We observe that the $\omega$-Morse index is the number of negative eigendirections counted
	according to their multiplicities on which $h(\gamma)$ is negative definite. We also define 
		\[
			n_\omega (\gamma) \=\dim \ker h(\gamma).
		\]
	
	The following Morse-type index theorem relates the Morse index of a solution with the 
	$\omega$\nobreakdash-index introduced in Subsection~\ref{subsec:omega_Maslov}. 
	
	\begin{lem}[Morse Index Theorem, {\cite[page~172]{MR1898560}}] \label{thm:indextheorem}
		Let $\gamma$ be a critical point of the Lagrangian action functional (hence a classical solution 
		of the Euler-Lagrange equation~\eqref{def:EL}) and let $\phi$ be the fundamental solution
		of the linearised system around $\gamma$ (that is, $\phi$ satisfies \eqref{eq:lin_sys}). Then
			\[
				\iMor^\omega(\gamma) = i_\omega(\phi),	\qquad	n_\omega(\gamma) = \nu_\omega(\gamma),  
				\qquad \forall\, \omega \in \U.
			\]
	\end{lem}
	
	%\red
	{
	We close this section by recalling two important results about the minimising properties of the circular periodic solutions of the $\alpha$-homogeneous Kepler problem and the circular
	Lagrangian solution of the three-body problem under $\alpha$-homogeneous potential.
	The first one is due to Gordon (cf.~\cite{MR0502484}) for the case $\alpha=1$ and was generalised to different homogeneity degrees by Venturelli in \cite[Proposition~2.2.3]{Venturelli:phd}.
	\begin{lem} \label{lem:venturelli_kepler}
		In the $\alpha$-homogeneous Kepler problem with $\alpha \in [1,2)$, circular solutions are local minimisers of the Lagrangian action functional in the space of loops with winding number
		$\pm 1$ around the origin.
	\end{lem}
	As regards the circular Lagrangian solution for the $\alpha$-homogeneous $3$-body problem (without any restriction on the choice of the masses), from \cite[Theorem~3.1.17]{Venturelli:phd}
	we infer the following result.
	\begin{lem} \label{lem:venturelli_3body}
		In the $\alpha$-homogeneous $3$-body problem the circular Lagrange relative equilibrium is a local minimum of the Lagrangian action functional when $\alpha \in [1,2)$ (it is actually a
		strict minimiser if $\alpha \neq 1$). It is a non-degenerate saddle point when $\alpha \in (0,1)$.  
	\end{lem}
	}

%	---------------------------------------------------------------------------------------
	\section{Linear and spectral stability of the Lagrangian solution} \label{sec:linearstability}
%	---------------------------------------------------------------------------------------
	
	Recall that in Section~\ref{subsec:decomp} we established that there exists a system of symplectic coordinates such that the linearised system restricted to $E_2 \oplus E_3$ is represented
	in the standard basis of $\R^8$ by the matrix $\Lambda$ defined in~\eqref{eq:Lambda}. Note now that $\Lambda$ can be expressed as the diamond product
	$\Lambda_2 \diamond \Lambda_3$ of two matrices $\Lambda_2$ and $\Lambda_3$ defined by
		\begin{equation}\label{eq:Lambda_23}
			\Lambda_2 \= \begin{pmatrix}
						0 & 1 & \alpha + 1 & 0 \\
						-1 & 0 & 0 & -1 \\
						1 & 0 & 0 & 1 \\
						0 & 1 & -1 & 0
					\end{pmatrix},
					\qquad
			\Lambda_3 \= \begin{pmatrix}
						0 & 1 & \frac{1}{2} \left( \alpha + \frac{\alpha+2}{3} \sqrt{\smash[b]{9 - \beta}} \right) & 0 \\
						-1 & 0 & 0 & \frac{1}{2} \left( \alpha - \frac{\alpha+2}{3} \sqrt{\smash[b]{9 - \beta}} \right) \\
						1 & 0 & 0 & 1 \\
						0 & 1 & -1 & 0
					\end{pmatrix},
		\end{equation}
	the range of $\alpha$ now being $[0, 2)$ (cf.~Remark~\ref{rmk:alpha=0}).
	The former encodes the dynamics on the symplectic invariant subspace $E_2$, whereas the latter 
	governs the motion on $E_3$. 
	
	System~\eqref{eq:sysR8} thus decouples into two linear autonomous Hamiltonian subsystems on $E_2$ and $E_3$ respectively, and it follows that its fundamental solution
	$\Phi \in \P_{2\pi}(8)$ can be written as the diamond product of the fundamental solutions 
	$\phi_2 \in \P_{2\pi}(4)$ and $\phi_3 \in \P_{2\pi}(4)$ of these subsystems.
	
	\begin{rmk}
		By the discussion given in Section \ref{sec:description} (see also \cite[page~271]{MR2145251} for the gravitational case) the Hamiltonian system on the invariant subspace $E_2$
		is equivalent to the generalised $\alpha$-homogeneous and logarithmic Kepler problem.
		It is worthwhile noting that the matrix $\Lambda_3$ coincides with $\Lambda_2$ when%
			\footnote{Technically speaking we ruled out the possibility that the parameter $\beta$ could be equal to $0$ for two reasons. The first is that at some point of the derivation of 
				the matrix of the linearised system we divided by $\beta$ (cf.~Section~\ref{sec:description}); the second is due to the fact that if $\beta = 0$ then two masses would vanish
				and therefore there would be no dynamics at all. However we consider the limit $\beta \to 0$ and the extension by continuity.}
		$\beta = 0$: in this case then the essential part of the fundamental solution of the Lagrangian circular orbit coincides with the fundamental solution of the Kepler orbit.
	\end{rmk}

	\begin{defn}
		The linear autonomous Hamiltonian system~\eqref{eq:sysR8} is \emph{spectrally stable} if the spectrum $\sigma(\Lambda)$ of $\Lambda$ is contained in the imaginary axis $i\R$;
		we call it \emph{linearly stable} if in addition the matrix $\Lambda$ is diagonalisable. We say that System~\eqref{eq:sysR8} is \emph{degenerate} if $\ker \Lambda \neq \{ 0 \}$.
	\end{defn}
	
	Note that the spectrum $\sigma(\Lambda)$ of $\Lambda$ is the union $\sigma(\Lambda_2) \cup \sigma(\Lambda_3)$ of the spectra of $\Lambda_2$ and $\Lambda_3$ respectively. The
	eigenvalues of $\Lambda_2$ are
		\[
			0,\ 0,\ \pm i\sqrt{2 - \alpha};
		\]
	hence the system is always degenerate for every $n \geq 3$. It is then natural, following Moeckel in \cite{Moeckel:notes}, to adopt the following terminology.
	
	\begin{rmk}\label{rmk:alpha=2}
		We observe that when $\alpha = 2$ the spectrum of $\Lambda_2$ reduces to $\{0\}$, while for $\alpha > 2$ such matrix admits also two non-zero real eigenvalues.
		More precisely, when $\alpha = 2$ the two non-zero purely imaginary eigenvalues of $\Lambda_2$ collapse into the origin (this corresponds to a Krein collision in $1$ for the eigenvalues
		of the monodromy matrix) and split into a pair of non-zero real eigenvalues when $\alpha > 2$. The value $\alpha =2$ is then the threshold of linear stability on $E_2$.
	\end{rmk}

	\begin{defn}
		A relative equilibrium is \emph{non-degenerate} if the remaining $4$ eigenvalues (relative to $\Lambda_3$) are different from $0$; we say that it is \emph{spectrally stable} if these
		eigenvalues are purely imaginary and \emph{linearly stable} if, in addition to this condition of spectral stability, $\Lambda_3$ is diagonalisable.
	\end{defn}
	
	The eigenvalues of the Hamiltonian matrix $\Lambda_3$ are
		\begin{align*}
				\lambda_1^\pm & \= \pm \frac{1}{6} i \sqrt{36 - 18\alpha + 6 \sqrt{9(\alpha - 2)^2 - \beta(\alpha + 2)^2}}  \\
				\lambda_2^\pm & \= \pm \frac{1}{6} i \sqrt{36 - 18\alpha - 6 \sqrt{9(\alpha - 2)^2 - \beta(\alpha + 2)^2}}
		\end{align*}
	and their direct study leads to a picture of the zones of stability and instability in the parameter space (see Figure~\vref{fig:stab}).
		
	\begin{prop} \label{thm:stabilita}
		The rectangle $(0,9] \times [0,2)$ is divided into three regions, depending on the stability of the relative equilibrium determined by the parameters $\alpha$ and $\beta$:
			\begin{enumerate}
				\item Region of \emph{linear stability}
						\[
							\mathit{LS} \= \Set{ (\beta, \alpha) \in (0,9] \times [0,2) | \beta < 9 \left( \frac{\alpha - 2}{\alpha + 2} \right)^2 };
						\]
				\item Curve of \emph{spectral (but not linear) stability}
						\[
							\mathit{SS} \= \Set{ (\beta, \alpha) \in (0,9] \times [0,2) | \beta = 9 \left( \frac{\alpha - 2}{\alpha + 2} \right)^2 };
						\]
				\item Region of \emph{spectral instability}
						\[
							\mathit{SI} \= \Set{ (\beta, \alpha) \in (0,9] \times [0,2) | \beta > 9 \left( \frac{\alpha - 2}{\alpha + 2} \right)^2 }.
						\]
			\end{enumerate}
	\end{prop}
	
	\begin{proof}
		A direct computation shows that the eigenvalues of $\Lambda_3$ are purely imaginary in $\mathit{LS} \cup \mathit{SS}$; however on the stability curve $\mathit{SS}$ they collide and
		form two pairs of purely imaginary eigenvalues which give rise to two Jordan blocks, so that diagonalisability is lost. In the region $\mathit{SI}$ their real part is different from $0$.
	\end{proof}

	\begin{rmk}
		Let us observe that as $\beta$ is arbitrarily small (which corresponds to the presence of a dominant mass) and $\alpha$ is bounded away from 2 we lie in the region of linear stability.
		Such a result agrees with Moeckel's conjecture on the dominant mass, according to which relative equilibria with a dominant mass are linearly stable.
	\end{rmk}

%	---------------------------------------------------------------------------
	\section{Maslov index of the generalised Kepler problem}
%	---------------------------------------------------------------------------	
	
	The aim of this section is to compute the $\omega$-index of the restriction of the Hamiltonian system~\eqref{eq:sysR8} to the invariant subspace $E_2$ of the phase space. As already
	observed, the Hamiltonian function on this subspace coincides with the Hamiltonian of the generalised (\ie $\alpha$-homogeneous and logarithmic) Kepler problem.

	\subsection{Computation of the Maslov index}\label{subs:MaslovE2}
	Consider the linear autonomous Hamiltonian initial value problem
	\begin{equation}\label{eq:IVP1}
			\begin{cases}
				\dot{\phi}_2(\tau) = \Lambda_2 \phi_2(\tau) \\
				\phi_2(0) = I_4.
			\end{cases}
		\end{equation}
	Here $\phi_2$ is the restriction to $E_2$ of the fundamental solution $\Phi$ of the Lagrangian circular orbit.

	\begin{prop}\label{thm:morsekepler}
		The Maslov index of the fundamental solution $\phi_2$ of System~\eqref{eq:IVP1} is
			\[
				i_1 (\phi_2) = \begin{cases}
							0 & \text{if } \alpha \in [1, 2) \\
							2 & \text{if } \alpha \in [0, 1).
						\end{cases}
			\]
	\end{prop}
	
% % 	\begin{figure}[tb]
% % 		\begin{center}
% % 			\begin{asy}
% % 				import graph;
% % 			
% % 				size(300, 200, IgnoreAspect);
% % 				
% % 				// Assi coordinati
% % 				xaxis(xmin = -0.5, xmax = 10, arrow = EndArrow);
% % 				yaxis(ymin = 0, ymax=2, dashed);
% % 				
% % 				fill((0,0)--(9,0)--(9,1)--(0,1)--cycle, paleblue);
% % 				
% % 				// Disegno curve
% % 				// yequals(1, xmin=0, xmax=9);
% % 				yequals(2, xmin=0, xmax=9, dashed);
% % 				xequals(9, ymin=0, ymax=2);
% % 				draw((0,-0.2)--(0,0));
% % 				draw((0,2)--(0,2.35), arrow = EndArrow);
% % 						
% % 				// Etichette
% % 				label("0", (0, 0), SW);
% % 				labelx(9);
% % 				labelx("$\beta$", 10);
% % 				labely(1);
% % 				labely(2);
% % 				labely("$\alpha$", 2.35);
% % 				
% % 				label("0", (4.5,1.55), S);
% % 				label("2", (4.5,0.55), S);
% % 			\end{asy}
% % 			\caption{Values of $i_1(\phi_2)$. On the line $\alpha = 1$ it is equal to $0$.} \label{fig:MaslovE2}
% % 		\end{center}
% % 	\end{figure}
	
	\begin{proof}
	Here is the fundamental solution $\phi_2(\tau) \= \exp \bigl( \tau \Lambda_2 \bigr)$, with $\tau \in [0, 2\pi]$:
		\[
			\phi_2(\tau) =
				\begin{pmatrix}
					\frac{2 - \alpha \cos(\sqrt{2 - \alpha}\, \tau)}{2 - \alpha} & \frac{2 + \alpha}{2 - \alpha}\tau - \frac{2\alpha \sin(\sqrt{2 - \alpha}\, \tau)}{(2 - \alpha)^{3/2}} &
							\frac{2 + \alpha}{2 - \alpha}\tau - \frac{\alpha^2 \sin(\sqrt{2 - \alpha}\, \tau)}{(2 - \alpha)^{3/2}} & \frac{\alpha [1 - \cos(\sqrt{2 - \alpha}\, \tau)]}{2 - \alpha} \\[10pt]
					- \frac{\sin(\sqrt{2 - \alpha}\, \tau)}{\sqrt{2 - \alpha}} & \frac{2 \cos(\sqrt{2 - \alpha}\, \tau) - \alpha}{2 - \alpha} &
															\frac{\alpha [\cos(\sqrt{2 - \alpha}\, \tau) - 1]}{2 - \alpha} & - \frac{\sin(\sqrt{2 - \alpha}\, \tau)}{\sqrt{2 - \alpha}} \\[10pt]
					\frac{\sin(\sqrt{2 - \alpha}\, \tau)}{\sqrt{2 - \alpha}} & \frac{2 - 2 \cos(\sqrt{2 - \alpha}\, \tau)}{2 - \alpha} &
										\frac{2 - \alpha \cos(\sqrt{2 - \alpha}\, \tau)}{2 - \alpha} & \frac{\sin(\sqrt{2 - \alpha}\, \tau)}{\sqrt{2 - \alpha}} \\[10pt]
					\frac{2 \cos(\sqrt{2 - \alpha}\, \tau) - 2}{2 - \alpha} & \frac{4 \sin(\sqrt{2 - \alpha}\, \tau)}{(2 - \alpha)^{3/2}} - \frac{2 + \alpha}{2 - \alpha}\tau &														\frac{2 \alpha \sin(\sqrt{2 - \alpha}\, \tau)}{(2 - \alpha)^{3/2}} - \frac{2 + \alpha}{2 - \alpha}\tau & \frac{2 \cos(\sqrt{2 - \alpha}\, \tau) - \alpha}{2 - \alpha}
				\end{pmatrix}.
		\]
	Following \cite{MR3218836}, if we consider the symplectic matrix
		\[
			P \= \begin{pmatrix}
					1 & 0 & 0 & 6\pi \\
					0 & - \frac{1}{6\pi} & -1 & 0 \\
					0 & 0 & 1 & 0 \\
					0 & 0 & 0 & - 6\pi
				\end{pmatrix}
		\]
	we see that $\phi_2(\tau)$ is symplectically equivalent to $\tilde{\phi}_2(\tau) \= P^{-1} \phi_2(\tau) P$, which is given by
		\[
			\tilde{\phi}_2(\tau) \= \begin{pmatrix}
								\cos(\sqrt{2 - \alpha}\, \tau) & - \frac{2 \sin(\sqrt{2 - \alpha}\, \tau)}{6 \pi \sqrt{2 - \alpha}} & - \sqrt{2 - \alpha} \sin(\sqrt{2 - \alpha}\, \tau) & 0 \\[10pt]
								0 & 1 & 0 & 0 \\[10pt]
								\frac{\sin(\sqrt{2 - \alpha}\, \tau)}{\sqrt{2 - \alpha}} & \frac{2 \cos(\sqrt{2 - \alpha}\, \tau) - 2}{6 \pi (2 - \alpha)} & \cos(\sqrt{2 - \alpha}\, \tau) & 0 \\[10pt]
								\frac{2 - 2 \cos(\sqrt{2 - \alpha}\, \tau)}{6 \pi (2 - \alpha)} &
															\frac{1}{36 \pi^2} \Bigl( \frac{4 \sin(\sqrt{2 - \alpha}\, \tau)}{(2 - \alpha)^{3/2}} - \frac{2 + \alpha}{2 - \alpha}\tau \Bigr) &
																									\frac{2 \sin(\sqrt{2 - \alpha}\, \tau)}{6 \pi \sqrt{2 - \alpha}} & 1
							\end{pmatrix};
		\]
	it follows, by the naturality property, that $i_1(\phi_2) = i_1(\tilde{\phi}_2)$. Take now the homotopy $F : [0,1] \times [0, 2\pi] \to \Sp(4)$ defined by
		\[
			F(s, \tau) \= \begin{pmatrix}
									\cos(\sqrt{2 - \alpha}\, \tau) & - s \frac{2 \sin(\sqrt{2 - \alpha}\, \tau)}{6 \pi \sqrt{2 - \alpha}} & - \sqrt{2 - \alpha} \sin(\sqrt{2 - \alpha}\, \tau) & 0 \\[10pt]
									0 & 1 & 0 & 0 \\[10pt]
									\frac{\sin(\sqrt{2 - \alpha}\, \tau)}{\sqrt{2 - \alpha}} & s \frac{2 \cos(\sqrt{2 - \alpha}\, \tau) - 2}{6 \pi (2 - \alpha)} & \cos(\sqrt{2 - \alpha}\, \tau) & 0 \\[10pt]
									s \frac{2 - 2 \cos(\sqrt{2 - \alpha}\, \tau)}{6 \pi (2 - \alpha)} &
															\frac{1}{36 \pi^2} \Bigl( \frac{4 \sin(\sqrt{2 - \alpha}\, \tau)}{(2 - \alpha)^{3/2}} - \frac{2 + \alpha}{2 - \alpha}\tau \Bigr) &
																								s \frac{2 \sin(\sqrt{2 - \alpha}\, \tau)}{6 \pi \sqrt{2 - \alpha}} & 1
								\end{pmatrix}.
		\]
	It is admissible because we have that $F(1, \tau) = \tilde{\phi}_2(s, \tau) \in \Sp(4)$ and $F(s, 0) = I_4$ for all $s \in [0,1]$ and all $\tau \in [0, 2\pi]$.
	Moreover, $F(1, \tau) = \tilde{\phi}_2(\tau)$ and
		\[
			\begin{split}
				\tilde{\phi}_2(0, \tau) & = \begin{pmatrix}
										\cos(\sqrt{2 - \alpha}\, \tau) & 0 & - \sqrt{2 - \alpha} \sin(\sqrt{2 - \alpha}\, \tau) & 0 \\
										0 & 1 & 0 & 0 \\
										\frac{\sin(\sqrt{2 - \alpha}\, \tau)}{\sqrt{2 - \alpha}} & 0 & \cos(\sqrt{2 - \alpha}\, \tau) & 0 \\
										0 & \frac{1}{36 \pi^2} \Bigl( \frac{4 \sin(\sqrt{2 - \alpha}\, \tau)}{(2 - \alpha)^{3/2}} - \frac{2 + \alpha}{2 - \alpha}\tau \Bigr) & 0 & 1
									\end{pmatrix} \\
								& \\
								& = \begin{pmatrix}
										\cos(\sqrt{2 - \alpha}\, \tau) & - \sqrt{2 - \alpha} \sin(\sqrt{2 - \alpha}\, \tau) \\
										\frac{\sin(\sqrt{2 - \alpha}\, \tau)}{\sqrt{2 - \alpha}} & \cos(\sqrt{2 - \alpha}\, \tau)
									\end{pmatrix} \diamond
									\begin{pmatrix}
										1 & 0 \\
										\frac{1}{36 \pi^2} \Bigl( \frac{4 \sin(\sqrt{2 - \alpha}\, \tau)}{(2 - \alpha)^{3/2}} - \frac{2 + \alpha}{2 - \alpha}\tau \Bigr) & 1
									\end{pmatrix} \\
								& \\
								& \eq R_\alpha(\tau) \diamond N_\alpha(\tau).
			\end{split}
		\]
	Therefore, being the Maslov index a homotopic invariant, we have
		\begin{equation} \label{eq:MaslovE2}
			i_1 (\phi_2) = i_1(R_\alpha) + i_1(N_\alpha).
		\end{equation}
	From Example~\ref{ex:Ralpha}, Example~\ref{ex:Nalpha} and Lemma~\ref{thm:chiave} we find
		\[
			i_1 (R_\alpha) =
				\begin{cases}
					1 & \text{if } \alpha \in (1, 2) \\
					3 & \text{if } \alpha \in [0, 1),
				\end{cases}
				\qquad \qquad
			i_1 (N_\alpha) = -1 \qquad \forall\, \alpha \in [0, 2),
		\]
	and the thesis follows.
	\end{proof}

	\subsection{Computation of the $\omega$-index on $E_2$}

	Next we compute the $\omega$-index $i_\omega(\phi_2)$ for all $\omega \in \U \setminus \{ 1 \}$. To this end we have to compute first the splitting numbers of the monodromy matrix
		\[
			M_2 \= \phi_2(2\pi) \sim R_\alpha(2\pi) \diamond N_1(1, 1) \qquad \text{for every } \alpha \in [0, 2).
		\]
	We note that $R_\alpha(\tau)$ is not a normal form for every $\tau \in [0, 2\pi]$; however, it is homotopic to the rotation $R(\sqrt{2 - \alpha}\, \tau)$ via the map $G : [0, 1] \times [0, 2\pi] \to \Omega^0(R_\alpha)$ defined by
		\[
			G(s, \tau) \= \begin{pmatrix}
						\cos(\sqrt{2 - \alpha}\, \tau) & - \frac{\sqrt{2 - \alpha} \sin(\sqrt{2 - \alpha}\, \tau)}{1 - s + s\sqrt{2 - \alpha}} \\
						\frac{(1 - s + s\sqrt{2 - \alpha}) \sin(\sqrt{2 - \alpha}\, \tau)}{\sqrt{2 - \alpha}} & \cos(\sqrt{2 - \alpha}\, \tau)
					\end{pmatrix}.
		\]
	Accordingly, for all $\alpha \in [0, 2)$
		\begin{equation} \label{eq:M2decomp}
			M_2 \sim R(\theta_\alpha) \diamond N_1(1, 1),
		\end{equation}
	where, modulo $2\pi$,
		\begin{equation} \label{eq:thetaalpha}
			\theta_\alpha \= 2\pi\sqrt{2 - \alpha} \in \begin{cases}
											\{ 0 \} & \text{if } \alpha = 1 \\
											(0, \pi) & \text{if } \alpha \in [0, 1) \cup \bigl( \frac{7}{4}, 2 \bigr) \\
											\{ \pi \} & \text{if } \alpha = \frac{7}{4} \\
											(\pi, 2\pi) & \text{if } \alpha \in \bigl( 1, \frac{7}{4} \bigr)
										\end{cases}
		\end{equation}
	
%	\begin{figure}[tb]
%		\centering
%			\begin{asy}
%				import graph;
%			
%				size(200, 200*2/3, IgnoreAspect);
%				
%				// Assi coordinati
%				xaxis(xmin = -0.5, xmax = 10, arrow = EndArrow);
%				yaxis(ymin = 0, ymax=2, dashed);
%				
%				fill((0,0)--(9,0)--(9,7/4)--(0,7/4)--cycle, paleblue);
%				
%				// Disegno curve
%				yequals(2, xmin=0, xmax=9, dashed);
%				xequals(9, ymin=0, ymax=2);
%				draw((0,-0.2)--(0,0));
%				draw((0,2)--(0,2.35), arrow = EndArrow);
%						
%				// Etichette
%				label("$\scriptstyle 0$", (0, 0), SW);
%				label("$\scriptstyle 9$", (9, 0), S);
%				label("$\scriptstyle \beta$", (10, 0), S);
%				label("$\scriptstyle 7/4$", (0, 7/4), W);
%				label("$\scriptstyle 2$", (0, 2), W);
%				labely("$\scriptstyle \alpha$", (0, 2.35), W);
%				
%				label("$\scriptstyle 0$", (4.5,1.975), S);
%				label("$\scriptstyle 2$", (4.5,1), S);
%			\end{asy}
%			\caption{Values of $i_{-1}(\phi_2)$. On the line $\alpha = \frac{7}{4}$ it is equal to $0$.} \label{fig:i-1phi2}
%	\end{figure}
	
	\begin{prop} \label{prop:omegaE2}
		The $\omega$-index $i_\omega(\phi_2)$ of the fundamental solution $\phi_2$ is given by:
			\begin{enumerate}[(i)]
				\item $\alpha \in \bigl[ \frac{7}{4}, 2 \bigr)$:
						\[
							i_\omega(\phi_2) = \begin{cases}
											1 & \text{if $0 < \theta < \theta_\alpha$} \\
											0 & \text{if $\theta_\alpha \leq \theta \leq \pi$}
										\end{cases}
						\]
				\item $\alpha \in \bigl( 1, \frac{7}{4} \bigr)$:
						\[
							i_\omega(\phi_2) = \begin{cases}
											1 & \text{if $0 < \theta \leq -\theta_\alpha$} \\
											2 & \text{if $-\theta_\alpha < \theta \leq \pi$}
										\end{cases}
						\]
				\item $\alpha = 1$:
						\[
							i_\omega(\phi_2) = 2 \quad \text{for all $\theta \in (0, \pi]$}
						\]
				\item $\alpha \in [0,1)$:
						\[
							i_\omega(\phi_2) = \begin{cases}
											3 & \text{if $0 < \theta < \theta_\alpha$} \\
											2 & \text{if $-\theta_\alpha \leq \theta \leq \pi$}
										\end{cases}
						\]
			\end{enumerate}
		where $\omega = e^{i \theta} \neq 1$.
	\end{prop}
	
	\begin{proof}
	Item~\ref{formula} of Proposition~\ref{thm:splittingprop} gives
		\begin{equation} \label{eq:33HS}
			i_\omega(\phi_2) = i_1(\phi_2) + S_{M_2}^+(1) + \sum_{\omega_0} \bigl( S_{M_2}^+(\omega_0) - S_{M_2}^-(\omega_0) \bigr) - S_{M_2}^-(\omega),
		\end{equation}
	where $\omega \in \U \setminus \{ 1 \}$ is such that $\Im(\omega) \geq 0$ and $\omega_0 \in \sigma(M_2)$ lies in the interior of the arc of the upper unit semicircle connecting $1$ and
	$\omega$ (see Figure~\ref{fig:circle}).
	Note that the assumption $\Im(\omega) \geq 0$ does not imply any loss of generality: by virtue of Item~\ref{4} of Proposition~\ref{thm:splittingprop} we have indeed that
		\[
			i_{\overline{\omega}}(\phi_2) = i_\omega(\phi_2).
		\]
	
	\begin{figure}[tb]
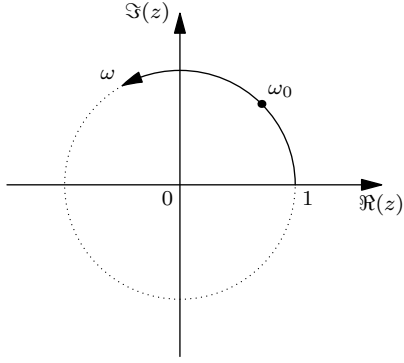

		\begin{center}
			\begin{asy}
				import graph;
			
				size(150,0);
			
				xaxis(-1.5, 1.75, arrow = EndArrow);
				yaxis(-1.5, 1.5, arrow = EndArrow);
				label("$\scriptstyle \Re(z)$", (1.75, 0), S);
				label("$\scriptstyle 1$", (1,0), SE);
				label("$\scriptstyle \Im(z)$", (0, 1.5), W);
				label("$\scriptstyle 0$", (0, 0), SW);
			
				draw(arc((0, 0), 1, 0, 120, CCW), EndArrow);
				draw(arc((0, 0), 1, 120, 360, CCW), dotted);
				
				label("$\scriptstyle \omega$", (-0.5, 0.866), NW);
				dot((0.707, 0.707));
				label("$\scriptstyle \omega_0$", (0.707, 0.707), NE);
			\end{asy}
			\caption{Position of $\omega$ and $\omega_0$.} \label{fig:circle}
		\end{center}
	\end{figure}
	
	From \eqref{eq:M2decomp} we find that for every $\omega \in \U$ with $\Im(\omega) \geq 0$
		\[
			S^{\pm}_{M_2}(\omega) = \begin{cases}
									S^{\pm}_{R(\theta_\alpha)}(\omega) + S^{\pm}_{N_1(1, 1)}(\omega) & \text{if } \alpha \in [0, 2) \setminus \bigl\{ 1, \frac{7}{4} \bigr\}, \\
									S^{\pm}_{-I_2}(\omega) + S^{\pm}_{N_1(1, 1)}(\omega) & \text{if } \alpha = \frac{7}{4}, \\
									S^{\pm}_{I_2}(\omega) + S^{\pm}_{N_1(1, 1)}(\omega) & \text{if } \alpha = 1.
								\end{cases}
		\]
	Thanks to the results collected in Proposition~\ref{thm:splittingprop} we know that if $\omega \notin \sigma(M_2) = \{ 1, 1, e^{i\theta_\alpha}, e^{-i\theta_\alpha} \}$ then
	$S^\pm_{M_2}(\omega) = 0$; moreover the splitting numbers involved are the following:
		\begin{subequations} \label{eq:splitnum}
		\begin{align}
			\big(S_{N_1(1,1)}^+(1), S_{N_1(1,1)}^-(1)\big) & = (1,1), \\
			\big(S_{R(\theta_\alpha)}^+(e^{i \theta_\alpha}), S_{R(\theta_\alpha)}^-(e^{i \theta_\alpha})\big) & = (0,1), \qquad \forall\, \alpha \in [0, 2) \setminus \bigl\{ 1, \tfrac{7}{4} \bigr\} \\
			\bigl( S_{I_2}^+(1), S_{I_2}^-(1) \bigr) & = (1, 1), \\
			\bigl( S_{-I_2}^+(-1), S_{-I_2}^-(-1) \bigr) & = (1, 1).
		\end{align}
		\end{subequations}

	Writing $\omega \= e^{i\theta}$, we are now able to compute the $\omega$-index depending on $\alpha$ and on the position of $\omega$ with respect to the eigenvalues
	$e^{\pm i \theta_\alpha}$ (modulo $2\pi$). Using Formula~\eqref{eq:33HS}, we distinguish the following cases:
		\begin{enumerate}[(i)]
			\item $\alpha \in \bigl( \frac{7}{4}, 2 \bigr)$:
					\[
						i_\omega(\phi_2) = \begin{cases}
										i_1(\phi_2) + S^+_{M_2}(1) & \text{if $\theta \in (0, \theta_\alpha)$} \\
										i_1(\phi_2) + S^+_{M_2}(1) - S^-_{M_2}(e^{i \theta_\alpha}) & \text{if $\theta = \theta_\alpha$} \\
										i_1(\phi_2) + S^+_{M_2}(1) + S^+_{M_2}(e^{i \theta_\alpha}) - S^-_{M_2}(e^{i \theta_\alpha}) & \text{if $\theta \in (\theta_\alpha, \pi]$},
									\end{cases}
					\]
				leading to
					\[
						i_\omega(\phi_2) = \begin{cases}
											1 & \text{if $\theta \in (0, \theta_\alpha)$} \\
											0 & \text{if $\theta \in [\theta_\alpha, \pi]$}.
										\end{cases}
					\]
			\item $\alpha = \frac{7}{4}$:
					\[
						i_\omega(\phi_2) = \begin{cases}
										i_1(\phi_2) + S^+_{M_2}(1) & \text{if $\theta \in (0, \pi)$} \\
										i_1(\phi_2) + S^+_{M_2}(1) - S^-_{M_2}(-1) & \text{if $\theta = \pi$},
									\end{cases}
					\]
				giving
					\[
						i_\omega(\phi_2) = \begin{cases}
											1 & \text{if $\theta \in (0, \pi)$} \\
											0 & \text{if $\theta = \pi$}.
										\end{cases}
					\]
			\item $\alpha \in \bigl( 1, \frac{7}{4} \bigr)$:
					\[
						i_\omega(\phi_2) = \begin{cases}
										i_1(\phi_2) + S^+_{M_2}(1) & \text{if $\theta \in (0, -\theta_\alpha)$} \\
										i_1(\phi_2) + S^+_{M_2}(1) - S^-_{M_2}(e^{-i \theta_\alpha}) & \text{if $\theta = -\theta_\alpha$} \\
										i_1(\phi_2) + S^+_{M_2}(1) + S^+_{M_2}(e^{-i \theta_\alpha}) - S^-_{M_2}(e^{-i \theta_\alpha}) & \text{if $\theta \in (-\theta_\alpha, \pi]$},
									\end{cases}
					\]
				yielding
					\[
						i_\omega(\phi_2) = \begin{cases}
											1 & \text{if $\theta \in (0, -\theta_\alpha]$} \\
											2 & \text{if $\theta \in (-\theta_\alpha, \pi]$}.
										\end{cases}
					\]
			\item $\alpha = 1$:
					\[
						i_\omega(\phi_2) = i_1(\phi_2) + S^+_{M_2}(1) = 2 \qquad \text{for all $\theta \in (0, \pi]$}.
					\]
			\item $\alpha \in [0, 1)$:
					\[
						i_\omega(\phi_2) = \begin{cases}
										i_1(\phi_2) + S^+_{M_2}(1) & \text{if $\theta \in (0, \theta_\alpha)$} \\
										i_1(\phi_2) + S^+_{M_2}(1) - S^-_{M_2}(e^{i \theta_\alpha}) & \text{if $\theta = \theta_\alpha$} \\
										i_1(\phi_2) + S^+_{M_2}(1) + S^+_{M_2}(e^{i \theta_\alpha}) - S^-_{M_2}(e^{i \theta_\alpha}) & \text{if $\theta \in (\theta_\alpha, \pi]$},
									\end{cases}
					\]
				obtaining
					\[
						i_\omega(\phi_2) = \begin{cases}
											3 & \text{if $\theta \in (0, \theta_\alpha)$} \\
											2 & \text{if $\theta \in [\theta_\alpha, \pi]$}.
										\end{cases} \qedhere
					\]
		\end{enumerate}
	\end{proof}

	The following result is a direct consequence of Lemma~\ref{thm:bott} and generalises \cite[Proposition~3.6]{MR2563212} to the $\alpha$-homogeneous case.
	
	\begin{prop}\label{thm:iteratesuE_2}
		Let $\phi_2$ be the fundamental solution of System~\eqref{eq:IVP1} and $k \in \N \setminus \{ 0 \}$. Then the Maslov index of the $k$-th iteration $\phi_2^k$ of
		$\phi_2$ is given by $i_1(\phi_2^k) = \sum_{\omega^k = 1} i_\omega(\phi_2)$ and is equal to:
			\begin{enumerate}[(i)]
				\item $\alpha \in \big( \frac{7}{4}, 2 \big)$:
						\[
							i_1(\phi_2^k) = 2(n_{k, \alpha}^- - 1),
						\]
					where $n_{k, \alpha}^-$ is the number of $k$-th roots of unity in the arc $[1, e^{i \theta_\alpha})$;
				\item $\alpha \in \big( 1, \frac{7}{4} \big)$:
						\[
							i_1(\phi_2^k) = \begin{cases}
											2(n_{k, \alpha}^- - 1) + 4(n_{k, \alpha}^+ - 1) + 2 & \text{if $k$ is even} \\
											2(n_{k, \alpha}^- - 1) + 4n_{k, \alpha}^+ & \text{if $k$ is odd} \\
										\end{cases}
						\]
					where $n_{k, \alpha}^-$ is the number of $k$-th roots of unity in the arc $[1, e^{-i \theta_\alpha}]$ and $n_{k,\alpha}^+$ is the number of $k$-th roots of unity in the
					arc $(e^{- i \theta_\alpha}, -1]$;
				\item $\alpha = 1$:
						\[
							i_1(\phi_2^k) = 2(k - 1)
						\]
				\item $\alpha \in [0, 1)$:
						\[
							i_1(\phi_2^k) = \begin{cases}
											6(n_{k, \alpha}^- - 1) + 4(n_{k, \alpha}^+ - 1) + 4 & \text{if $k$ is even} \\
											6(n_{k, \alpha}^- - 1) + 4n_{k, \alpha}^+ + 2 & \text{if $k$ is odd},
										\end{cases}
						\]
					where $n_{k, \alpha}^-$ is the number of $k$-th roots of unity in the arc $[1, e^{\pm i \theta_\alpha})$ and $n_{k,\alpha}^+$ is the number of $k$-th roots of unity in the
					arc $[e^{\pm i \theta_\alpha}, -1]$.
			\end{enumerate}
	\end{prop}
	
%	\begin{rmk} \label{rmk:alpha=2bis}
		We observe that, for fixed $k$, the index $i_1(\phi_2^k)$ is constant on horizontal bands of the rectangle $(0,9] \times [0,2)$, since it is independent of $\beta$ (see
		Figure~\vref{fig:iterates}). From the previous proposition it is evident that the index is monotonically non-increasing as $\alpha$ increases for every $k \in \N \setminus \{0\}$.
		
		Since the computation of the Maslov index of the iterate is based on the Bott-Long formula, it is clear that the only contributions to this value are given by those $\omega$-indices for
		which $\omega$ is a root of unity. This means that one has a jump in the index of the $k$-th iterate only when the angle $\theta_\alpha$ (defined in \eqref{eq:thetaalpha}) is a rational
		multiple of $2\pi$, \ie $\theta_\alpha = \frac{2l\pi}{k}$ for some $l \in \N \setminus \{0\}$. Now, since $\theta_\alpha \in [0, 2\sqrt{2}\pi]$ it follows that $l$ actually ranges in the set
		$\{1, \dotsc, [\sqrt{2}k] \}$.
		
		In particular the Maslov index vanishes when $0 < \theta_\alpha < \frac{2\pi}{k}$, that is when $\alpha > 2 - \frac{1}{k^2}$.
		As $k$ increases, the horizontal lines corresponding to the jumps of $i_1(\phi_2^k)$, which are characterised by the double sequence $(\alpha_{k, l})$ with
		$\alpha_{k, l} \= 2 - \frac{l^2}{k^2}$, accumulate at the stability threshold $\alpha = 2$ as $k \to +\infty$ (see Remark~\ref{rmk:alpha=2}).
		
		Let us now fix $\alpha \in [0,2)$. The number of $k$-th roots of unity in the arc $[1, e^{\pm i \theta_\alpha})$ increases with $k$ and diverges to $+\infty$ as $k \to +\infty$, hence
		$i_1(\phi_2^k) \to +\infty$ as $k \to +\infty$.
%	\end{rmk}

%	---------------------------------------------------------------------
	\section{$\omega$-index associated with the restriction to $E_3$}
%	---------------------------------------------------------------------

	In this section we perform the computation of the $\omega$-index of the restriction $\phi_3$ to $E_3$ of the fundamental solution $\Phi$ of the Lagrangian circular orbit. 
	This will be achieved, as before, by means of the splitting numbers.

	\subsection{Computation of the Maslov index}\label{subs:MaslovE3}
	
	The restriction $\phi_3$ to $E_3$ of the fundamental solution $\Phi$ of the Lagrangian circular orbit satisfies the linear autonomous Hamiltonian initial value problem
		\begin{equation}
			\begin{cases}
				\dot{\phi}_3(\tau) = \Lambda_3 \phi_3(\tau) \\
				\phi_3(0) = I_4.
			\end{cases}
		\end{equation}
	
	By taking into account Proposition~\ref{thm:stabilita}, we immediately get the following result.
	
	\begin{prop} \label{thm:Maslovinst}
		The Maslov index $i_1(\phi_3)$ is zero for all $(\beta, \alpha) \in \mathit{SI}$.
	\end{prop}
	
	\begin{proof}
		The eigenvalues that contribute to the Maslov index are only the ones contained in $\U$. %\newline
		If $9(\alpha - 2)^2 - \beta(\alpha + 2)^2 < 0$ (\ie in the region $\mathit{SI}$) the spectrum is contained in $\C\setminus(\U \cup \R)$ and the result follows.
	\end{proof}
	
	The monodromy matrix $M_3 \= \phi_3(2\pi) \= \exp(2\pi \Lambda_3)$ is non-degenerate in the whole region $\mathit{LS}$ of linear stability, except on the curve of equation 
		\begin{equation} \label{eq:degeneratecurve}
			\beta = \dfrac{36(1 - \alpha)}{(\alpha + 2)^2},
		\end{equation}
	where two of the four eigenvalues are equal to $1$. On the stability curve $\mathit{SS}$ of equation
		\[
			\beta = 9 \left( \frac{\alpha - 2}{\alpha + 2} \right)^2,
		\]
	instead, $M_3$ is non-degenerate but not diagonalisable. We can compute its Maslov index in the non-degenerate subzone of $\mathit{LS}$ by using again the formula of
	Proposition~\ref{prop:MaslovAbbo}: the Krein-positive eigenvalues of $\Lambda_3$ are
		\begin{align*}
			\lambda_1^- & = -\frac{1}{6} i \sqrt{36 - 18\alpha + 6 \sqrt{9(\alpha - 2)^2 - \beta(\alpha + 2)^2}} \\
		\intertext{and}
			\lambda_2^+ & = \frac{1}{6} i \sqrt{36 - 18\alpha - 6 \sqrt{9(\alpha - 2)^2 - \beta(\alpha + 2)^2}}
		\end{align*}
	for all $(\beta, \alpha) \in \mathit{LS}$, so that
		\[
			i_1(\phi_3) = 
					\begin{cases}
						0 & \textup{if }\ \dfrac{36(1 - \alpha)}{(\alpha + 2)^2} < \beta < 9 \dfrac{(\alpha - 2)^2}{(\alpha + 2)^2} \\[10pt]
						2 & \textup{if }\ 0 < \beta < \dfrac{36(1 - \alpha)}{(\alpha + 2)^2}.
					\end{cases}
		\]
	However, since the Maslov index is a lower semicontinuous function, we conclude that $i_1(\phi_3) = 0$ also on the curve~\eqref{eq:degeneratecurve} and on the stability curve:
		\[
			i_1(\phi_3) = 
					\begin{cases}
						0 & \textup{if }\ \beta \geq \dfrac{36(1 - \alpha)}{(\alpha + 2)^2} \\[10pt]
						2 & \textup{if }\ 0 < \beta < \dfrac{36(1 - \alpha)}{(\alpha + 2)^2}.
					\end{cases}
		\]
	The result is depicted in Figure~\ref{fig:i1phi3}.
	
	\begin{figure}[tb]
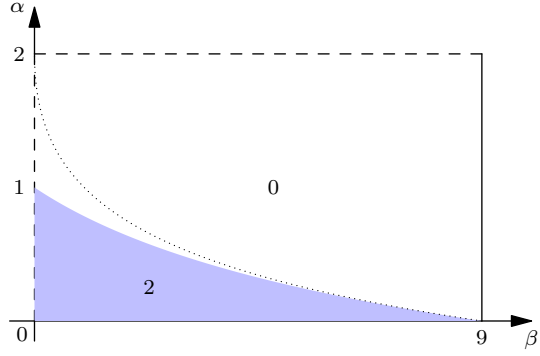

		\centering
			\begin{asy}
				import graph;

				size(200, 200*2/3, IgnoreAspect);
		
				real x1(real t) {return 9*(t - 2)^2/(t + 2)^2;}		// Curva di stabilità
				real x2(real t) {return 36*(1 - t)/(t + 2)^2;}
				real y(real t) {return t;}					// Serve solo per la parametrizzazione
		
			///	Assi coordinati
				xaxis(xmin = -0.5, xmax = 10, arrow=EndArrow);
				yaxis(ymin = 0, ymax=2, dashed);
				
			///	Riempimenti
				path q = buildcycle(graph(x2, y, 0, 1, operator ..), (0,1)--(0,0)--(9,0));
				fill(q, paleblue);
				
			///	Disegno curve
				// yequals(1, xmin=0, xmax=9, dotted);
				yequals(2, xmin=0, xmax=9, dashed);
				xequals(9, ymin=0, ymax=2);
				draw(graph(x1, y, 0, 2, operator ..), dotted, "$\beta = 9 \left( \frac{\alpha - 2}{\alpha + 2} \right)^2$");
				// draw(graph(x2, y, 0, 1, operator ..), "$\beta = \frac{36(1 - \alpha)}{(\alpha + 2)^2}$");
				draw((0,-0.15)--(0,0));
				draw((0,2)--(0,2.35), arrow = EndArrow);
		
			///	Etichette
				label("$\scriptstyle 0$", (0, 0), SW);
				label("$\scriptstyle 9$", (9, 0), S);
				label("$\scriptstyle \beta$", (10, 0), S);
				label("$\scriptstyle 1$", (0, 1), W);
				label("$\scriptstyle 2$", (0, 2), W);
				label("$\scriptstyle \alpha$", (0, 2.35), W);
				
				label("$\scriptstyle 0$", (4.5, 1), E);
				label("$\scriptstyle 2$", (2, 1/4), E);
		
			///	Legenda
			//	add(legend(invisible), point(S), 40S, UnFill);
			\end{asy}
			\caption{Values of $i_1(\phi_3).$ The dotted curve is the stability curve.} \label{fig:i1phi3}
	\end{figure}

	\subsection{Computation of the $\omega$-index on $E_3$} \label{subsec:omegaE3}
	
	The monodromy matrix $M_3 \= \exp(2\pi \Lambda_3)$ is similar to the diagonal matrix
		\[
			\diag(e^{2\pi \lambda_1^-}, e^{2\pi \lambda_2^+}, e^{2\pi \lambda_1^+}, e^{2\pi \lambda_2^-})
		\]
	and can consequently be expressed as
		\[
			M_3 = R(\theta_{\alpha, \beta}^{(1)}) \diamond R(\theta_{\alpha, \beta}^{(2)}),
		\]
	with $\theta_{\alpha, \beta}^{(1)} \= \Im(2\pi \lambda_1^+)$ and $\theta_{\alpha, \beta}^{(2)} \= \Im(2\pi \lambda_2^-)$.
	
	\begin{rmk}
		Note that these two angles correspond to the Krein-negative eigenvalues; the reason is the following. When $\beta \to 0$ the dynamics of the problem reduces to that of a generalised
		Kepler problem, \ie to the restriction to $E_2$ previously analysed. The values of the $\omega$-index must then agree with the ones found in the previous study when
		approaching the segment $\{0\} \times [0, 2)$ as $\beta$ tends to $0$, and this forces the choice of the two eigenvalues.
	\end{rmk}
	
	Observe that in the region $\mathit{LS}$ these angles take the following values (modulo $2\pi$):
		\begin{gather}
			\theta_{\alpha, \beta}^{(1)} \in \begin{cases}
										\{ 0 \} & \text{if $\beta = \dfrac{36(1 - \alpha)}{(\alpha + 2)^2}$} \\[10pt]
										(0, \pi) & \text{if $\beta < \dfrac{36(1 - \alpha)}{(\alpha + 2)^2}$ or
																	\bigg( $\beta > \dfrac{9(7 - 4\alpha)}{4(\alpha + 2)^2}$ and $\alpha > \dfrac{3}{2}$ \bigg)} \\[10pt]
										\{ \pi \} & \text{if $\beta = \dfrac{9(7 - 4\alpha)}{4(\alpha + 2)^2}$ and $\alpha > \dfrac{3}{2}$} \\[10pt]
																							%$\sqrt{4 - 2\alpha + \abs{3 - 2\alpha}} = 1$} \\[10pt]
										(\pi, 2\pi) & \text{if $\beta > \dfrac{36(1 - \alpha)}{(\alpha + 2)^2}$ and
																	\bigg( $\beta < \dfrac{9(7 - 4\alpha)}{4(\alpha + 2)^2}$ or $\alpha < \dfrac{3}{2}$ \bigg)}
									\end{cases} \label{eq:theta1} \\[10pt]
			\theta_{\alpha, \beta}^{(2)} \in \begin{cases}
										(0, \pi) & \text{if $\beta > \dfrac{9(7 - 4\alpha)}{4(\alpha + 2)^2}$ and $\alpha < \dfrac{3}{2}$} \\[10pt]
										\{ \pi \} & \text{if $\beta = \dfrac{9(7 - 4\alpha)}{4(\alpha + 2)^2}$ and $\alpha < \dfrac{3}{2}$} \\[10pt]
										(\pi, 2\pi) & \text{if $\beta < \dfrac{9(7 - 4\alpha)}{4(\alpha + 2)^2}$ or $\alpha > \dfrac{3}{2}$}.
									\end{cases} \label{eq:theta2}
		\end{gather}
	Figure~\ref{fig:theta1} and Figure~\ref{fig:theta2} show the involved regions, and they are superposed in Figure~\ref{fig:theta1theta2}. 
	In order to compute the splitting numbers and eventually find the $\omega$-index we have to determine not only the absolute position of $\theta_{\alpha, \beta}^{(1)}$ and
	$\theta_{\alpha, \beta}^{(2)}$ on $\U$ (which is the one given above), but also how their \emph{relative} position changes as the parameters $\alpha$ and $\beta$ vary. This is represented in
	Figure~\ref{fig:theta1theta2relpos}.
	
	\begin{figure}[tb]
	\centering
	\subfloat[][\emph{Dashed-dotted line:} $\theta_{\alpha, \beta}^{(1)} = 0$; \emph{Solid line:} $\theta_{\alpha, \beta}^{(1)} = \pi$;
				\emph{Light shade:} $\theta_{\alpha, \beta}^{(1)} \in (0, \pi)$, \emph{Dark shade:} $\theta_{\alpha, \beta}^{(1)} \in (\pi, 2\pi)$.\label{fig:theta1}]
					{\includegraphics[width=0.45\textwidth]{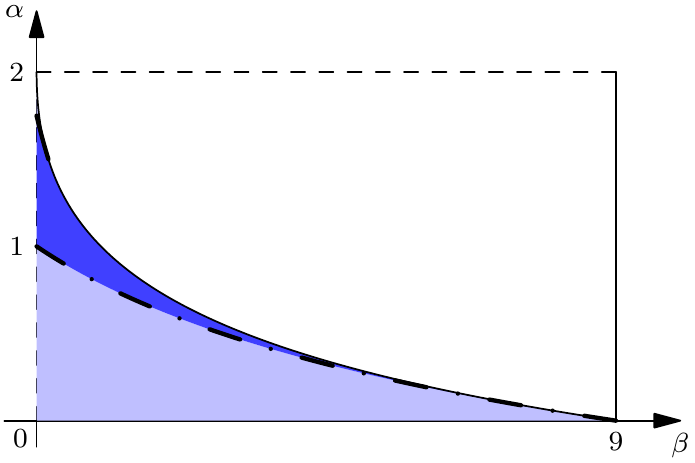}} \qquad
	\subfloat[][\emph{Solid line:} $\theta_{\alpha, \beta}^{(2)} = \pi$; \emph{Light shade:} $\theta_{\alpha, \beta}^{(2)} \in (0, \pi)$;
				\emph{Dark shade:} $\theta_{\alpha, \beta}^{(2)} \in (\pi, 2\pi)$. \label{fig:theta2}]{\includegraphics[width=0.45\textwidth]{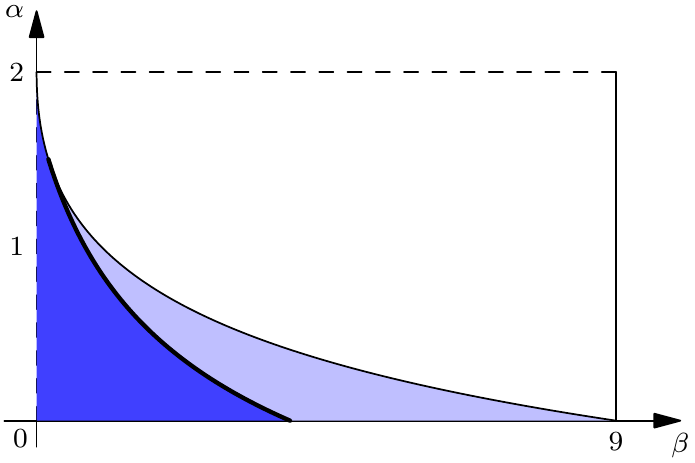}}
	\caption{Values of $\theta_{\alpha, \beta}^{(1)}$ (a) and $\theta_{\alpha, \beta}^{(2)}$ (b) modulo $2\pi$.}
	\end{figure}

	\begin{figure}[tb]
	\centering
	\subfloat[][\emph{Light shade:} $\theta_{\alpha, \beta}^{(1)} \in (0, \pi)$ and $\theta_{\alpha, \beta}^{(2)} \in (\pi, 2\pi)$;
				\emph{Medium shade:} $\theta_{\alpha, \beta}^{(1)}, \theta_{\alpha, \beta}^{(2)} \in (\pi, 2\pi)$;
				\emph{Heavy shade:} $\theta_{\alpha, \beta}^{(1)} \in (\pi, 2\pi)$ and $\theta_{\alpha, \beta}^{(2)} \in (0, \pi)$;
				\emph{Dark shade:} $\theta_{\alpha, \beta}^{(1)}, \theta_{\alpha, \beta}^{(2)} \in (0, \pi)$.\label{fig:theta1theta2}]
					{\includegraphics[width=0.45\textwidth]{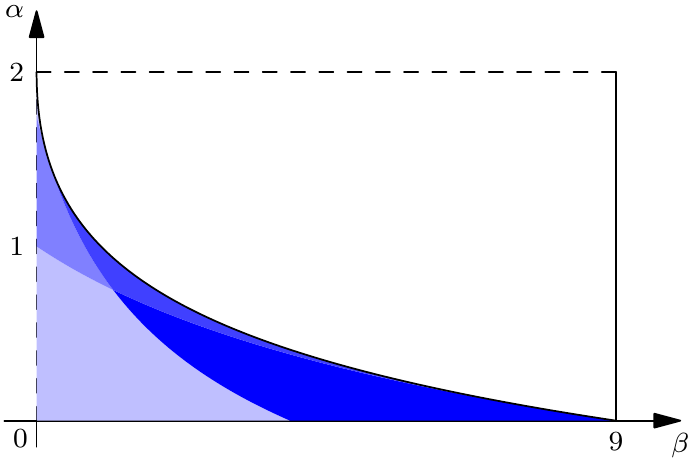}} \qquad
	\subfloat[][If $\widetilde{\theta}_{\alpha, \beta}^{(1)}$ and $\widetilde{\theta}_{\alpha, \beta}^{(2)}$ are the representatives of $\pm \theta_{\alpha, \beta}^{(1)}$ and
				$\pm \theta_{\alpha, \beta}^{(2)}$ in the upper unit semicircle, then the colours have to be interpreted in the following way.
				\emph{Light shade:} $\widetilde{\theta}_{\alpha, \beta}^{(1)} < \widetilde{\theta}_{\alpha, \beta}^{(2)}$;
				\emph{Solid line:} $\widetilde{\theta}_{\alpha, \beta}^{(1)} = \widetilde{\theta}_{\alpha, \beta}^{(2)}$;
				\emph{Dark shade:} $\widetilde{\theta}_{\alpha, \beta}^{(1)} > \widetilde{\theta}_{\alpha, \beta}^{(2)}$.\label{fig:theta1theta2relpos}]
						{\includegraphics[width=0.45\textwidth]{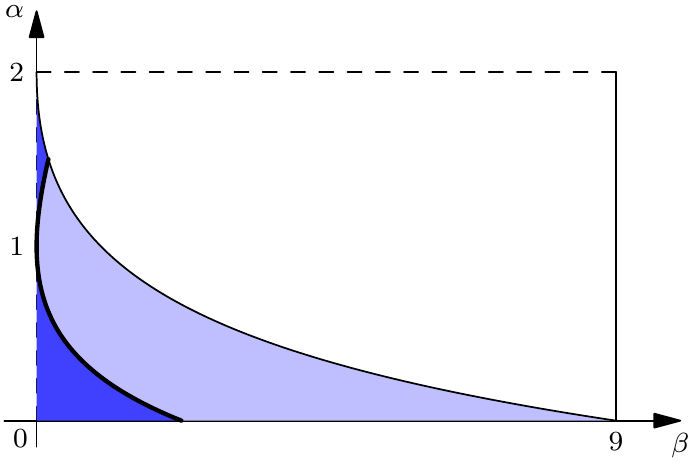}}
	\caption{Values of $\theta_{\alpha, \beta}^{(1)}$ and $\theta_{\alpha, \beta}^{(2)}$ (a) and their relative position (b) modulo $2\pi$.}
	\end{figure}

	Now, for every $\omega \in \U$ we have that
		\[
			S^\pm_{M_3}(\omega) = S^\pm_{R(\theta_{\alpha, \beta}^{(1)})}(\omega) + S^\pm_{R(\theta_{\alpha, \beta}^{(2)})}(\omega)
		\]
	and $S^\pm_{M_3}(\omega) = 0$ if $\omega \notin \sigma(M_3) = \{ e^{\pm i \theta_{\alpha, \beta}^{(1)}}, e^{\pm i \theta_{\alpha, \beta}^{(2)}} \}$. In order to compute the $\omega$-index
	we use the formula
		\[
			i_\omega(\phi_3) - i_1(\phi_3) = S_{M_3}^+(1) + \sum_{\omega_0} \bigl( S_{M_3}^+(\omega_0) - S_{M_3}^-(\omega_0) \bigr) - S_{M_3}^-(\omega),
		\]
	where $\omega \in \U \setminus \{ 1 \}$ is such that $\Im(\omega) \geq 0$ and $\omega_0 \in \sigma(M_3)$ lies in the interior of the arc of the upper unit semicircle connecting $1$ and
	$\omega$ (see Figure~\ref{fig:circle}). The splitting numbers involved are the following:
		\begin{align*}
			\big( S^+_{M_3}(1), S^-_{M_3}(1) \big) &= \begin{cases}
												(1, 1) & \text{if $\beta = \dfrac{36(1 - \alpha)}{(\alpha + 2)^2}$} \\[10pt]
												(0, 0) & \text{otherwise}
											\end{cases} \\[10pt]
			\big( S^+_{M_3}(-1), S^-_{M_3}(-1) \big) &= \begin{cases}
												(1, 1) & \text{if $\beta = \dfrac{9(7 - 4\alpha)}{4(\alpha + 2)^2}$ and $\alpha \neq \dfrac{3}{2}$} \\[10pt]
												(0, 0) & \text{otherwise}
											\end{cases} \\[10pt]
			\big( S^+_{M_3}(e^{i \theta_{\alpha, \beta}^{(1)}}), S^-_{M_3}(e^{i \theta_{\alpha, \beta}^{(1)}}) \big) &= \begin{cases}
																		(0, 1) & \text{for all $\theta_{\alpha, \beta}^{(1)} \notin \{0, \pi, \pm\theta_{\alpha, \beta}^{(2)}\}$} \\
																		(0, 2) & \text{if $\theta_{\alpha, \beta}^{(1)} = \theta_{\alpha, \beta}^{(2)}$} \\
																		(1, 1) & \text{if $\theta_{\alpha, \beta}^{(1)} = -\theta_{\alpha, \beta}^{(2)}$}
																		\end{cases} \\[10pt]
			\big( S^+_{M_3}(e^{i \theta_{\alpha, \beta}^{(2)}}), S^-_{M_3}(e^{i \theta_{\alpha, \beta}^{(2)}}) \big) &= \begin{cases}
																		(0, 1) & \text{for all $\theta_{\alpha, \beta}^{(2)} \notin \{0, \pi, \pm\theta_{\alpha, \beta}^{(1)}\}$} \\
																		(0, 2) & \text{if $\theta_{\alpha, \beta}^{(2)} = \theta_{\alpha, \beta}^{(1)}$} \\
																		(1, 1) & \text{if $\theta_{\alpha, \beta}^{(2)} = -\theta_{\alpha, \beta}^{(1)}$}
																		\end{cases}
		\end{align*}
		
	The $\omega$-index depends therefore on the values of $\alpha$ and $\beta$. Writing $\omega \= e^{i\theta}$, we have
		\begin{enumerate}[i)]
			\item $\beta > \dfrac{9(7 - 4\alpha)}{4(\alpha + 2)^2}$ and $\alpha > \dfrac{3}{2}$:
					\[
						i_\omega(\phi_3) = \begin{cases}
											0 & \text{if $0 < \theta \leq -\theta_{\alpha, \beta}^{(2)}$} \\
											1 & \text{if $-\theta_{\alpha, \beta}^{(2)} < \theta < \theta_{\alpha, \beta}^{(1)}$} \\
											0 & \text{if $\theta_{\alpha, \beta}^{(1)} \leq \theta \leq \pi$}
										\end{cases}
					\]
			\item $\beta = \dfrac{9(7 - 4\alpha)}{4(\alpha + 2)^2}$ and $\alpha > \dfrac{3}{2}$:
					\[
						i_\omega(\phi_3) = \begin{cases}
											0 & \text{if $0 < \theta \leq -\theta_{\alpha, \beta}^{(2)}$} \\
											1 & \text{if $-\theta_{\alpha, \beta}^{(2)} < \theta < \pi$} \\
											0 & \text{if $\theta = \pi$}
										\end{cases}
					\]
			\item $\beta < \dfrac{9(7 - 4\alpha)}{4(\alpha + 2)^2}$ and $\beta < 9 \dfrac{(\alpha - 1)^2}{(\alpha + 2)^2}$ and $\alpha > 1$:
					\[
						i_\omega(\phi_3) = \begin{cases}
											0 & \text{if $0 < \theta \leq -\theta_{\alpha, \beta}^{(2)}$} \\
											1 & \text{if $-\theta_{\alpha, \beta}^{(2)} < \theta \leq -\theta_{\alpha, \beta}^{(1)}$} \\
											2 & \text{if $-\theta_{\alpha, \beta}^{(1)} < \theta \leq \pi$}
										\end{cases}
					\]
			\item $\beta = 9 \dfrac{(\alpha - 1)^2}{(\alpha + 2)^2}$ and $\alpha > 1$:
					\[
						i_\omega(\phi_3) = \begin{cases}
											0 & \text{if $0 < \theta \leq \theta_{\alpha, \beta}^{(2)} = \theta_{\alpha, \beta}^{(1)}$} \\
											2 & \text{if $\theta_{\alpha, \beta}^{(2)} = \theta_{\alpha, \beta}^{(1)} < \theta \leq \pi$}
										\end{cases}
					\]
			\item $\beta < \dfrac{9(7 - 4\alpha)}{4(\alpha + 2)^2}$ and $\beta > 9 \dfrac{(\alpha - 1)^2}{(\alpha + 2)^2}$ and $\beta > \dfrac{36(1 - \alpha)}{4(\alpha + 2)^2}$:
					\[
						i_\omega(\phi_3) = \begin{cases}
											0 & \text{if $0 < \theta < -\theta_{\alpha, \beta}^{(1)}$} \\
											1 & \text{if $-\theta_{\alpha, \beta}^{(1)} < \theta \leq -\theta_{\alpha, \beta}^{(2)}$} \\
											2 & \text{if $-\theta_{\alpha, \beta}^{(2)} < \theta \leq \pi$}
										\end{cases}
					\]
			\item $\beta = \dfrac{9(7 - 4\alpha)}{4(\alpha + 2)^2}$ and $\beta > \dfrac{36(1 - \alpha)}{(\alpha + 2)^2}$ and $\alpha < \dfrac{3}{2}$:
					\[
						i_\omega(\phi_3) = \begin{cases}
											0 & \text{if $0 < \theta \leq -\theta_{\alpha, \beta}^{(1)}$} \\
											1 & \text{if $-\theta_{\alpha, \beta}^{(1)} < \theta < \pi$} \\
											0 & \text{if $\theta = \pi$}
										\end{cases}
					\]
			\item $\beta > \dfrac{9(7 - 4\alpha)}{4(\alpha + 2)^2}$ and $\beta > \dfrac{36(1 - \alpha)}{(\alpha + 2)^2}$:
					\[
						i_\omega(\phi_3) = \begin{cases}
											0 & \text{if $0 < \theta \leq -\theta_{\alpha, \beta}^{(1)}$} \\
											1 & \text{if $\theta_{\alpha, \beta}^{(1)} < \theta < \theta_{\alpha, \beta}^{(2)}$} \\
											0 & \text{if $\theta_{\alpha, \beta}^{(2)} \leq \theta \leq \pi$}
										\end{cases}
					\]
			\item $\beta = \dfrac{36(1 - \alpha)}{(\alpha + 2)^2}$ and $\beta < \dfrac{9(7 - 4\alpha)}{4(\alpha + 2)^2}$:
					\[
						i_\omega(\phi_3) = \begin{cases}
											1 & \text{if $0 < \theta \leq -\theta_{\alpha, \beta}^{(2)}$} \\
											2 & \text{if $-\theta_{\alpha, \beta}^{(2)} < \theta \leq \pi$}
										\end{cases}
					\]
			\item $\beta = \dfrac{36(1 - \alpha)}{(\alpha + 2)^2}$ and $\beta = \dfrac{9(7 - 4\alpha)}{4(\alpha + 2)^2}$:
					\[
						i_\omega(\phi_3) = \begin{cases}
											1 & \text{if $0 < \theta < \pi$} \\
											0 & \text{if $\theta = \pi$}
										\end{cases}
					\]
			\item $\beta = \dfrac{36(1 - \alpha)}{(\alpha + 2)^2}$ and $\beta > \dfrac{9(7 - 4\alpha)}{4(\alpha + 2)^2}$:
					\[
						i_\omega(\phi_3) = \begin{cases}
											1 & \text{if $0 < \theta < \theta_{\alpha, \beta}^{(2)}$} \\
											0 & \text{if $\theta_{\alpha, \beta}^{(2)} \leq \theta \leq \pi$}
										\end{cases}
					\]
			\item $\beta < 9 \dfrac{(\alpha - 1)^2}{(\alpha + 2)^2}$ and $\alpha < 1$:
					\[
						i_\omega(\phi_3) = \begin{cases}
											2 & \text{if $0 < \theta \leq -\theta_{\alpha, \beta}^{(2)}$} \\
											3 & \text{if $-\theta_{\alpha, \beta}^{(2)} < \theta < \theta_{\alpha, \beta}^{(1)}$} \\
											2 & \text{if $-\theta_{\alpha, \beta}^{(1)} \leq \theta \leq \pi$}
										\end{cases}
					\]
			\item $\beta = 9 \dfrac{(\alpha - 1)^2}{(\alpha + 2)^2}$ and $\alpha < 1$:
					\[
						i_\omega(\phi_3) = \begin{cases}
											2 & \text{if $\theta \neq \theta_{\alpha, \beta}^{(1)} = -\theta_{\alpha, \beta}^{(2)}$} \\
											1 & \text{if $\theta = \theta_{\alpha, \beta}^{(1)} = -\theta_{\alpha, \beta}^{(2)}$}
										\end{cases}
					\]
			\item $\beta > 9 \dfrac{(\alpha - 1)^2}{(\alpha + 2)^2}$ and $\beta < \dfrac{36(1 - \alpha)}{(\alpha + 2)^2}$ and $\beta < \dfrac{9(7 - 4\alpha)}{4(\alpha + 2)^2}$:
					\[
						i_\omega(\phi_3) = \begin{cases}
											2 & \text{if $0 < \theta < \theta_{\alpha, \beta}^{(1)}$} \\
											1 & \text{if $\theta_{\alpha, \beta}^{(1)} \leq \theta \leq -\theta_{\alpha, \beta}^{(2)}$} \\
											2 & \text{if $-\theta_{\alpha, \beta}^{(2)} < \theta \leq \pi$}
										\end{cases}
					\]
			\item $\beta < \dfrac{36(1 - \alpha)}{(\alpha + 2)^2}$ and $\beta = \dfrac{9(7 - 4\alpha)}{4(\alpha + 2)^2}$:
					\[
						i_\omega(\phi_3) = \begin{cases}
											2 & \text{if $0 < \theta < \theta_{\alpha, \beta}^{(1)}$} \\
											1 & \text{if $\theta_{\alpha, \beta}^{(1)} \leq \theta < \pi$} \\
											0 & \text{if $\theta = \pi$}
										\end{cases}
					\]
			\item $\beta < \dfrac{36(1 - \alpha)}{(\alpha + 2)^2}$ and $\beta > \dfrac{9(7 - 4\alpha)}{4(\alpha + 2)^2}$:
					\[
						i_\omega(\phi_3) = \begin{cases}
											2 & \text{if $0 < \theta < \theta_{\alpha, \beta}^{(1)}$} \\
											1 & \text{if $\theta_{\alpha, \beta}^{(1)} \leq \theta < \theta_{\alpha, \beta}^{(2)}$} \\
											0 & \text{if $\theta_{\alpha, \beta}^{(2)} \leq \theta \leq \pi$}
										\end{cases}
					\]
		\end{enumerate}
	
%	\begin{figure}[tb]
%		\centering
%			\begin{asy}
%				import graph;
%		
%				size(200, 200*2/3, IgnoreAspect);
%		
%				real x1(real t) {return 9*(t - 2)^2/(t + 2)^2;}		// Curva di stabilità
%				real x2(real t) {return 36*(1 - t)/(t + 2)^2;}
%				real x3(real t) {return 9/4*(7 - 4*t)/(t + 2)^2;}
%				real y(real t) {return t;}					// Serve solo per la parametrizzazione
%		
%			///	Assi coordinati
%				xaxis(xmin = -0.5, xmax = 10, arrow=EndArrow);
%				yaxis(ymin = 0, ymax=2, dashed);
%				
%			///	Riempimenti
%				path p = buildcycle(graph(x3, y, 0, 7/4, operator ..), (0,7/4)--(0,0)--(63/16,0));
%				fill(p, paleblue);
%				
%			///	Disegno curve
%				yequals(2, xmin=0, xmax=9, dashed);
%				xequals(9, ymin=0, ymax=2);
%				draw(graph(x1, y, 0, 2, operator ..), dotted, "$\beta = 9 \left( \frac{\alpha - 2}{\alpha + 2} \right)^2$");
%				draw((0,-0.15)--(0,0));
%				draw((0,2)--(0,2.35), arrow = EndArrow);
%		
%			///	Etichette
%				label("$\scriptstyle 0$", (0, 0), SW);
%				label("$\scriptstyle 9$", (9, 0), S);
%				label("$\scriptstyle \beta$", (10, 0), S);
%				label("$\scriptstyle 7/4$", (0, 7/4), W);
%				label("$\scriptstyle 2$", (0, 2), W);
%				label("$\scriptstyle \alpha$", (0, 2.35), W);
%				
%				label("$\scriptstyle 0$", (4.5, 1), E);
%				label("$\scriptstyle 2$", (1, 1/3), W);
%			\end{asy}
%			\caption{Values of $i_{-1}(\phi_3)$. The dotted curve is the stability curve.} \label{fig:i-1phi3}
%	\end{figure}

	\begin{figure}[tb]
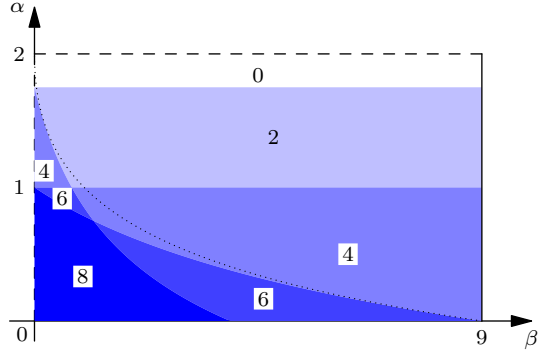

		\centering
			\begin{asy}
				import graph;

				size(200, 200*2/3, IgnoreAspect);
		
				real x1(real t) {return 9*(t - 2)^2/(t + 2)^2;}		// Curva di stabilità
				real x2(real t) {return 36*(1 - t)/(t + 2)^2;}
				real x3(real t) {return 9/4*(7 - 4*t)/(t + 2)^2;}
				real y(real t) {return t;}					// Serve solo per la parametrizzazione
		
			///	Assi coordinati
				xaxis(xmin = -0.5, xmax = 10, arrow=EndArrow);
				yaxis(ymin = 0, ymax=2, dashed);
				
			///	Riempimenti
				path p = buildcycle(graph(x3, y, 1, 7/4, operator ..), (0,7/4)--(9,7/4)--(9,1)--(3/4,1));
				fill(p, paleblue);
				path q = buildcycle((9,0)--(9,1)--(3/4,1), graph(x3, y, 1, 7/4, operator ..), (0,7/4)--(0,1)--(3/4,1), graph(x3, y, 1, 3/4, operator ..),
						graph(x2, y, 3/4, 0, operator ..));
				fill(q, lightblue);
				path r = buildcycle(graph(x2, y, 0, 3/4, operator ..), graph(x3, y, 3/4, 1, operator ..), (3/4,1)--(0,1), graph(x2, y, 1, 3/4, operator ..),
						graph(x3, y, 3/4, 0, operator ..), (63/16,0)--(9,0));
				fill(r, mediumblue);
				path s = buildcycle((0,0)--(63/16,0), graph(x3, y, 0, 3/4, operator ..), graph(x2, y, 3/4, 1, operator ..), (0,1)--(0,0));
				fill(s, blue);
				
			///	Disegno curve
				// yequals(1, xmin=0, xmax=9, dotted);
				yequals(2, xmin=0, xmax=9, dashed);
				xequals(9, ymin=0, ymax=2);
				draw(graph(x1, y, 0, 2, operator ..), dotted, "$\beta = 9 \left( \frac{\alpha - 2}{\alpha + 2} \right)^2$");
				draw((0,-0.15)--(0,0));
				draw((0,2)--(0,2.35), arrow = EndArrow);
		
			///	Etichette
				label("$\scriptstyle 0$", (0, 0), SW);
				label("$\scriptstyle 9$", (9, 0), S);
				label("$\scriptstyle \beta$", (10, 0), S);
				label("$\scriptstyle 1$", (0, 1), W);
				label("$\scriptstyle 2$", (0, 2), W);
				label("$\scriptstyle \alpha$", (0, 2.35), W);
				
				label("$\scriptstyle 0$", (4.5, 1.965), S);
				label("$\scriptstyle 2$", (4.5, 11/8), E);
				label("$\scriptstyle 4$", (1/3+1/6, 1+1/8), W, UnFill);
				label("$\scriptstyle 4$", (6,1/2), E, UnFill);
				label("$\scriptstyle 6$", (4.3, 1/6), E, UnFill);
				label("$\scriptstyle 6$", (1/4,11/12), E, UnFill);
				label("$\scriptstyle 8$", (2/3,1/3), E, UnFill);
			\end{asy}
			\caption{Values of the Maslov index $i_1(\phi_3^2)$ of the second iteration of $\phi_3$. The dotted curve is the stability curve.}
			\label{fig:i1phitot2}
	\end{figure}
		
	As we did analogously for $E_2$, we now turn our attention to the computation of the Maslov index $i_1(\phi_3^k)$ of the iterates of $\phi_3$.
	Once again we have that the Maslov index jumps in correspondence of those $\omega$ that are roots of unity, due to the structure of Bott-Long formula.	
	Hence, in the region $\mathit{LS}$, there are jumps of the index of the $k$-th iterate if and only if
		\begin{equation} \label{eq:condE3}
			\theta_{\alpha, \beta}^{(i)} = \frac{2l\pi}{k},
		\end{equation}
	for some $i = 1, 2$ and $l \in \N \setminus \{0\}$ (here $\theta_{\alpha, \beta}^{(i)}$ are the angles defined in \eqref{eq:theta1} and \eqref{eq:theta2}). In actual fact $\theta_{\alpha, \beta}^{(2)}$
	ranges in $(0, 2\pi)$, whereas $\theta_{\alpha, \beta}^{(1)}$ varies in $(0, 2\sqrt{2}\pi)$: this implies that $l$ takes values in the finite set $\{ 1, \dots, [\sqrt{2}k] \}$.
	
	Condition~\eqref{eq:condE3} defines a family of curves $\{f_{k, l}\}$ in the plane $(\beta, \alpha)$, parameterised by $k$ and $l$, that are defined by the equations
		\[
			\beta = - \frac{36}{(\alpha + 2)^2} \frac{l^2}{k^2} \bigg( \frac{l^2}{k^2} + \alpha - 2 \bigg).
		\]
	Each of these curves is convex and for $l \in \{1, \dotsc, k \}$ they are tangent at exactly one point to $\mathit{SS}$, namely
		\begin{equation} \label{eq:tangpoint}
			\bigg( \frac{9 l^4}{(2k^2 - l^2)^2},\ 2 \bigg( 1 - \frac{l^2}{k^2} \bigg) \bigg),
		\end{equation}
	and it turns out that the stability curve is actually the envelope of the one-parameter family $\{ f_t \}_{t \in (0, 1]}$ consisting of curves of equations
		\[
			\beta = - \frac{36}{(\alpha + 2)^2} t^2 (t^2 + \alpha - 2),
		\]
	into which the collection $\{f_{k, l}\}$ is contained. We observe that at every point in $\mathit{LS}$ the Maslov index $i_1(\phi_3^k)$ increases with $k$ and that, for each fixed
	$k \in \N \setminus \{0\}$, it decreases along half-lines from the origin. The index is also monotonically increasing when one crosses any of the curves $f_{k, l}$ (going towards the origin).
	Note that the intersections of these curves with the line $\beta = 0$ yield exactly the values of the sequence $(\alpha_{k, l})$ introduced in $E_2$ that tends to $\alpha = 2$ as $k \to +\infty$.
	
	In Figure~\ref{fig:i1phitot2} we present, as an example, a complete computation of $i_1(\phi_3^2)$, whereas Figure~\vref{fig:iterates3} shows some of the curves $f_{k, l}$ for some values of
	$k$.

%	----------------------------------------------------------------------
	\section{The $\omega$-Morse index of the Lagrangian circular orbit}
%	----------------------------------------------------------------------

%	In this last section we compute the $\omega$-Morse indices of the circular solution as well as the Morse index of the generalised Kepler problem. The central ingredient needed to perform
%	such a computation is the Morse index theorem (Lemma~\ref{thm:indextheorem}) and the computations of the Maslov index performed in the previous sections.
	
	Let $\mathscr L \in \mathscr C^\infty(T\widehat X, \R)$ and $\mathbb A : W^{1,2}(\R/2\pi\Z, \widehat X) \to \R$ be the Lagrangian function and the Lagrangian action functional respectively, 		as given in \eqref{eq:Lagrangian} and \eqref{eq:action}.
%		\begin{equation}\label{eq:funzionale}
%			\mathbb A(\gamma) \= \int_0^{2\pi} \mathscr L\big( \gamma(t), \dot \gamma (t) \big) \, dt
%		\end{equation}
%	which is smooth on the collisionless loops.
	Since the Euler-Lagrange equation for $\mathbb A$, which is smooth on collisionless loops,
	coincides with the Newton's equations given in \eqref{eq:Newtonintro}, for each pair   
	$(\beta,\alpha) \in (0,9]\times [0,2) $ the Lagrangian circular solution 
	$\gamma_{\alpha,\beta}$ of Newton's equation can be found (up to a standard bootstrap argument) as a critical point of $\mathbb{A}$.
	
	From Equation~\eqref{eq:secondvariation} we see that the second variation at the critical point $\gamma_{\alpha,\beta}$ is 
	\begin{equation}\label{eq:secondvariationlagrange}
	\d^2 \mathbb A(\gamma_{\alpha,\beta})[\xi,\eta] = \int_0^{2\pi}
		 \langle M\xi', \eta'\rangle  + 
		 \langle D^2U\big( \gamma_{\alpha, \beta}(t) \big) \xi,  \eta \rangle \,dt. 
	\end{equation}
	Using the Sobolev Embedding Theorem it follows that the 
	second variation is a (bounded) essentially positive Fredholm quadratic form, 
	being a weakly compact perturbation of an invertible quadratic form 
	(cf.~for instance \cite[Section~2, Proposition~3.1]{MR2133393} and references therein). 
	This in particular ensures that the $\omega$-Morse index $\iMor^\omega$ is finite. 

	By taking into account the Morse index theorem (Lemma~\ref{thm:indextheorem}), in order to 
	compute the $\iMor^\omega(\gamma_{\beta,\alpha})$ it is enough to 
	compute the $\omega$-index $i_\omega(\psi)$, where $\psi: [0,2\pi] \to \Sp(8)$ is the 
	fundamental solution of the first-order Hamiltonian system obtained from the associated Sturm 
	system through the Legendre transformation, \ie 
	$\psi$ satisfies
	\begin{equation}\label{eq:nonautonomous}
	 \begin{cases}
	  \psi'(t)= J B_{\alpha,\beta}(t) \psi(t)\\
	  \psi(0)= I_{2n}
	 \end{cases}
	\end{equation}
	where 
	\[
	 B_{\alpha,\beta}(t) \= \begin{pmatrix}
	         M & 0\\
	         0 & -D^2U\big( \gamma_{\alpha, \beta}(t) \big)
	        \end{pmatrix}.
	\]
	Taking into account \cite[Theorem 2.1]{MR2145251} there exists a linear symplectomorphism 
	between $T^*\widehat X$ and $E_2 \oplus E_3$. By the symplectic invariance of $\iclm$ 
	(cf.~\cite[Property~V, page~128]{MR1263126}) and hence of $i_\omega$ (as a direct 
	consequence of Lemma~\ref{thm:chiave}), it follows that 
	\[
	 i_\omega(\psi)=i_\omega(\Phi),
	\]
	where $\Phi$ was defined in Section~\ref{sec:linearstability}. Since 
	$\Phi= \phi_2\diamond \phi_3$, by using the symplectic additivity 
	property of $i_\omega$ and considering the previous discussion it follows that 
	\[
	 \iMor^\omega(\gamma_{\alpha, \beta})= i_\omega(\phi_2)+ i_\omega(\phi_3).
	\]
	\begin{rmk}
	We assume that  $H$ is a Hilbert space and there exist $H_1,\dots, H_n$ such that
	$H= \bigoplus_{k=1}^n H_k$. Let $A$ be a self-adjoint essentially positive bounded 
	Fredholm operator such that $A(H_k) \subseteq H_k$ for $i=1, \dots, n$. Setting 
	$A_k\=A|_{H_k}$ we have  
	\[
	 \iMor^\omega(A)= \sum_{k=1}^n \iMor^\omega(A_k).
	\]
	\end{rmk}
	It is worth noting that in correspondence of the $4$-dimensional subspaces 
	$E_2$ and $E_3$ there exist two $2$-dimensional subspaces 
	$\widehat X_2$ and $\widehat X_3$ of $\widehat{X}$ such that  $E_2=T^*\widehat X_2$ and  
	$E_3=T^*\widehat X_3$. Hence
	\[
	 W^{1,2}(\R/2\pi\Z, \widehat X) = W^{1,2}(\R/2\pi\Z, \widehat X_2)
	 \times  W^{1,2}(\R/2\pi\Z, \widehat X_3).
	\]
	In the next two subsections we shall compute the Lagrangian functions on the 
	aforementioned subspaces $\widehat X_2$ and $\widehat X_3$ as well as the 
	differential operators on such subspaces.

        \subsection{$\omega$-Morse index of the generalised Kepler problem}

%	In this Subsection we shall compute the Morse index of the circular Lagrangian  
%	solution on the invariant subspace $W^{1,2}(\R/2\pi\Z, \widehat X_2)$.
       Define the Lagrangian function on $W^{1,2}(\R/2\pi\Z, \widehat X_2)$ as
		\begin{equation}\label{eq:Lagrangiana2}
			\L_2(x, \dot{x}) \= \frac{1}{2} \norm{\dot{x}}^2 + \langle J x, 
			\dot{x} \rangle + \frac{1}{2} \langle S_2 x, x \rangle,
		\end{equation}
	where $S_2 \= \big( \begin{smallmatrix} \alpha + 2 & 0 \\ 0 & 0 \end{smallmatrix} \big)$. By a straightforward 
	calculation it follows that the origin in the configuration space is a solution of the 
	corresponding Euler-Lagrange equation 
	\begin{equation}\label{eq:ELsystem2}
	 -\ddot x - 2 J \dot x + S_2 x=0;
	\end{equation}	
	associated with $\mathscr L_2$. Let $\mathcal{B}_2 : W^{1, 2}(\R/2\pi\Z, \hat{X}_2) \times W^{1, 2}(\R/2\pi\Z, \hat{X}_2) \to \R$ be defined as follows:
	\[
	\mathcal{B}_2(x, y) \= \int_0^{2\pi}\big[ \langle \dot x ,\dot y\rangle + \langle J y, \dot x\rangle + 
	 \langle J \dot x, y\rangle + \langle S_2 x, y\rangle\big] \, dt.
	\]
	Once again it follows from the Sobolev Embedding Theorem that $\mathcal{B}_2$ is a (bounded) 
	essentially positive Fredholm quadratic form, being a weakly compact perturbation of 
	an invertible quadratic form. 
	This in particular ensures that the Morse index $\iMor^\omega$ is finite. 

	By taking into account the Legendre transformation, the corresponding autonomous Hamiltonian function is
		\[
			\H_2(v) \= \frac{1}{2} \langle B_2 v, v \rangle, \qquad \forall\, v \in \R^4,
		\]
	where 
	\begin{equation}\label{eq:B2}
			B_2 \= \begin{pmatrix}
					1 & 0 & 0 & 1 \\
					0 & 1 & -1 & 0 \\
					0 & -1 & -(\alpha + 1) & 0 \\
					1 & 0 & 0 & 1
				\end{pmatrix}.
		\end{equation}
	Clearly the origin in the phase space is the corresponding solution of the 
	linear autonomous Hamiltonian initial value problem
		\begin{equation} \label{eq:IVP11}
			\begin{cases}
				\phi'_2(\tau) = \Lambda_2 \phi_2(\tau) \\
				\phi_2(0) = I_4
			\end{cases}
		\end{equation}
		where $\Lambda_2 = JB_2$ agrees with the one given in formula \eqref{eq:Lambda_23}.
%%	\begin{rmk}	
%%	Let 
%%	\[
%%	\ell_2: 
%%	  \overline{D}(\omega, 2\pi) \subset L^2([0,2\pi], \C^2) 
%%	\to L^2([0,2\pi], \C^2)
%%	\]
%%	be the  differential operator on $L^2$ with 
%%	dense domain $ \overline{D}(\omega, 2\pi)$ 
%%	 \[
%%	\overline{D}(\omega, 2\pi):= \Set{\xi \in W^{2,2}([0,2\pi], \C^2)|
%%	  \xi(2\pi)= \omega\, \xi(0),\  \xi'(2\pi)=\omega\, \xi'(0)}
%%	\]
%%	given by
%%	\[
%%	 \ell_2:= -\dfrac{d^2}{dt^2} - 2J\dfrac{d}{dt} + S_2.
%%	\]
%%	It can be proved that that $\ell_2$ is an
%%	unbounded self-adjoint Fredholm operator in $L^2$ with domain $\overline D(\omega, 2\pi)$. 
%%	\end{rmk}
	
%	As already observed in Section~\ref{sec:linearstability} the restriction of the Hamiltonian system on 
%	$E_2$ coincides with that of the generalised Kepler problem. Denoting by $\gamma_{\alpha, 0}$ its solution 
%	of the generalised Kepler problem (the  notation agrees with the fact that the Hamiltonian system on $E_2$ 
%	turns out to be obtained by that one on $E_3$ after imposing $\beta=0$, even though is physically 
%	meaningless). 
	
	\begin{thm}\label{thm:morseE2}
		For all $\omega \in \U$, the $\omega$-Morse index of the circular solution $\gamma_{\alpha, 0}$ of the generalised Kepler problem coincides with $i_\omega(\phi_2)$, which has been
		computed in Propositions~\ref{thm:morsekepler} and \ref{prop:omegaE2}.
	\end{thm}
	
	\begin{proof}
		First of all we observe that as a direct consequence of 
		the results proved in Section~\ref{sec:description} the subspace $E_2$ is invariant under the phase 
		flow of the Hamiltonian~\eqref{eq:hamiltonian}. Moreover on this subspace the aforementioned 
		Hamiltonian reduces to the Hamiltonian of the generalised Kepler problem.
		Now, by the above construction System~\eqref{eq:IVP11} is the Legendre 
		transformation of the Euler\nobreakdash-Lagrange system~\eqref{eq:ELsystem2}.
		The thesis is then a direct consequence of Lemma~\ref{thm:indextheorem}. 
	\end{proof}

	\begin{rmk}
		 It is worthwhile noting that this result perfectly agrees with \cite[Proposition~3.6]{MR2563212} and \cite[Proposition~2.2.3]{Venturelli:phd}. Moreover we 
		 point out that in the last quoted reference the author only states that for $\alpha \in (0,1)$ the circular solutions are not local minimisers, without any further information on the 
		 Morse index. The logarithmic case has not been treated thus far from this point of view.
		\end{rmk}

	\subsection{$\omega$-Morse index of the Lagrangian circular orbit}
	
	We proceed exactly as in the previous subsection, by introducing the Lagrangian
		\[
			\L_3(x, \dot{x}) \= \frac{1}{2} \norm{\dot{x}}^2 + \langle J x, 
			\dot{x} \rangle + \frac{1}{2} \langle S_3 x, x \rangle
		\]
	on the Sobolev space $W^{1,2}(\R/2\pi\Z, \widehat X_3)$, with
		\[
			S_3 \= \begin{pmatrix}
					\frac{1}{6} \bigl[ 6 + 3\alpha + (\alpha + 2)\sqrt{\smash[b]{9 - \beta}} \bigr] & 0 \\
					0 & \frac{1}{6} \bigl[ 6 + 3\alpha - (\alpha + 2)\sqrt{\smash[b]{9 - \beta}} \bigr]
				\end{pmatrix}.
		\]
	Defining a symmetric bilinear form $\mathcal{B}_3$ in a completely analogous way as above, we obtain the Hamiltonian system
		\begin{equation} \label{eq:IVP12}
			\begin{cases}
				\phi'_3(\tau) = \Lambda_3 \phi_3(\tau) \\
				\phi_3(0) = I_4,
			\end{cases}
		\end{equation}
	where $\Lambda_3 = JB_3$, being
		\begin{equation} \label{eq:B3}
			B_3 \=\begin{pmatrix}
					1 & 0 & 0 & 1 \\
					0 & 1 & -1 & 0 \\
					0 & -1 & -\frac{1}{2} \Bigl( \alpha + \frac{\alpha+2}{3} \sqrt{\smash[b]{9 - \beta}} \Bigr) & 0 \\
					1 & 0 & 0 & - \frac{1}{2} \Big( \alpha - \frac{\alpha+2}{3} \sqrt{\smash[b]{9 - \beta}} \Big)
				\end{pmatrix}.
		\end{equation}
		
	\begin{thm}\label{thm:morseE3}
		For all $\omega \in \U$ the $\omega$-Morse index of the Lagrangian circular solution $\gamma_{\alpha, \beta}$ is given by $i_\omega(\Phi) = i_\omega(\phi_2) + i_\omega(\phi_3)$.
		In particular for $\omega = 1$ we have
		\[
			\iMor(\gamma_{\alpha, \beta}) = 
					\begin{cases}
						0 & \text{if $\alpha \in [1, 2)$} \\[10pt]
						2  & \text{if $\beta \geq \dfrac{36(1 - \alpha)}{(\alpha + 2)^2}$ and $\alpha \in [0,1)$} \\[10pt]
						4 & \text{if $0 < \beta < \dfrac{36(1 - \alpha)}{(\alpha + 2)^2}$}.
					\end{cases}
		\]
	\end{thm}
	
	\begin{proof}
		Arguing as in the proof of Theorem~\ref{thm:morseE2}, it is enough to apply Lemma~\ref{thm:indextheorem}, use the calculations performed in Subsections~\ref{subs:MaslovE3}
		and \ref{subs:MaslovE2} and the additivity of the Maslov index $i_1$.
	\end{proof}

	\subsection{Relation between linear stability and Morse index}
	
	We have shown how both in $E_2$ and in $E_3$ there is a sequence of curves (possibly straight lines) that ``converge'', in a suitable sense, to the boundary of the region of linear stability.
	By virtue of the Index Theorem also the Morse index of the iterates jumps when crossing each of those curves.
	
	Since the angles $\theta_{\alpha, \beta}^{(1)}$ and $\theta_{\alpha, \beta}^{(2)}$ introduced in Subsection~\ref{subsec:omegaE3} cover the whole of $\U$ as $\alpha$ and $\beta$ vary, it may happen that for some values of these parameters
	one of them is a rational multiple of $2\pi$ (so that its exponential is a root of unity).
	When this occurs then the corresponding curve in the plane $(\beta, \alpha)$ is tangent to the stability curve at the point whose coordinates are given by \eqref{eq:tangpoint}. Instead, in the
	case when the aforementioned angles do not give rise to roots of unity, one obtains tangency to the stability curve at some point only after taking the limit as $k \to +\infty$. The reason of this
	fact is simply due to the density of roots of unity in $\U$.

\appendix

%	--------------------------------------------------
	\section{The geometric structure of $\Sp(2)$} \label{app:SP2}
%	--------------------------------------------------
	
	{The symplectic group $\Sp(2)$ captured the attention of I.~Gelfand and V.~Lidskii first, who in 1958 described a toric representation of it \cite{MR0073767,MR0091390}.
	The $\R^3$-cylindrical coordinate representation of $\Sp(2)$ was instead introduced by Y.~Long in 1991 \cite{long1991structure}, and what follows, including Figure~\ref{fig:Sp(2)10} and Figure~\ref{fig:curva}
	(although we re-drew them ourselves), already appeared in \cite[p.898]{long1991structure} and  \cite[Section~2.1]{MR1898560} respectively.}
	
	Every real invertible matrix $A$ can be decomposed in \emph{polar form}
		\[
			A = PO,
		\]
	where $P \= (A\trasp{A})^{1/2}$ is symmetric and positive definite and $O \= P^{-1}A$ is orthogonal. If $A \in \Sp(2)$ then $\det P = 1$ and therefore $P \in \Sp(2)$ as well.
	This entails that $O \in \Sp(2)$; in fact, being orthogonal, it belongs to $\SO(2) \cong \U$, \ie it is a proper rotation:
		\[
			O = \begin{pmatrix}
					\cos \theta & -\sin \theta \\
					\sin \theta & \cos \theta
				\end{pmatrix}.
		\]
	
	Let $u : \Sp(2) \to \U$ be the map which associates every $2 \times 2$ real symplectic matrix with the angle of rotation of its orthogonal part:
		\[
			u(A) = u(PO) \= e^{i\theta}.
		\]
	Now, the eigenvalues of $P$ are all real, positive and reciprocal of each other. Therefore we have that $\tr P \geq 2$ and we may introduce a coordinate $\xi$ ranging in
	$[0, +\infty)$ by setting $\tr P = 2 \cosh \xi$. Hence we can write
		\[
			P = \begin{pmatrix}
					\cosh \xi + a & b \\
					b & \cosh \xi - a
				\end{pmatrix}
		\]
	for some $a, b \in \R$ such that $\cosh^2 \xi - a^2 - b^2 = 1$. Thus $b^2 = \sinh^2 \xi - a^2$, which is meaningful if and only if $\abs{a} \leq \abs{\sinh \xi}$. Hence we are allowed to set
	$a \= \sinh \xi \cos \eta$ for some $\eta \in \R$, so that $b = \sinh \xi \sin \eta$ and $P$ becomes
		\[
			P = \begin{pmatrix}
					\cosh \xi + \sinh \xi \cos \eta & \sinh \xi \sin \eta \\
					\sinh \xi \sin \eta & \cosh \xi - \sinh \xi \cos \eta
				\end{pmatrix}.
		\]
	Setting now $r \= \cosh \xi + \sinh \xi \cos \eta$ and $z \= \sinh \xi \sin \eta$ yields
		\[
			P = \begin{pmatrix}
					r & z \\
					z & \frac{1 + z^2}{r}
				\end{pmatrix}
		\]
	and then every symplectic matrix $M$ of size $2$ can be written as the product
		\begin{equation} \label{eq:sympldec}
			M = \begin{pmatrix}
					r & z \\
					z & \frac{1 + z^2}{r}
				\end{pmatrix}
				\begin{pmatrix}
					\cos \theta & -\sin \theta \\
					\sin \theta & \cos \theta
				\end{pmatrix},
		\end{equation}
	where $(r, \theta, z) \in (0, +\infty) \times [0, 2\pi) \times \R$. Viewing $(r, \theta, z)$ as cylindrical coordinates in $\R^3 \setminus \{ \text{$z$-axis} \}$ we obtain a representation of $\Sp(2)$
	in $\R^3$; more precisely, we obtain a smooth global diffeomorphism $\psi : \Sp(2) \to \R^3 \setminus \{ \text{$z$-axis} \}$. We shall henceforth identify elements in $\Sp(2)$ with their
	image under~$\psi$.
	
	\begin{figure}[tb]
	\begin{center}
		\includegraphics{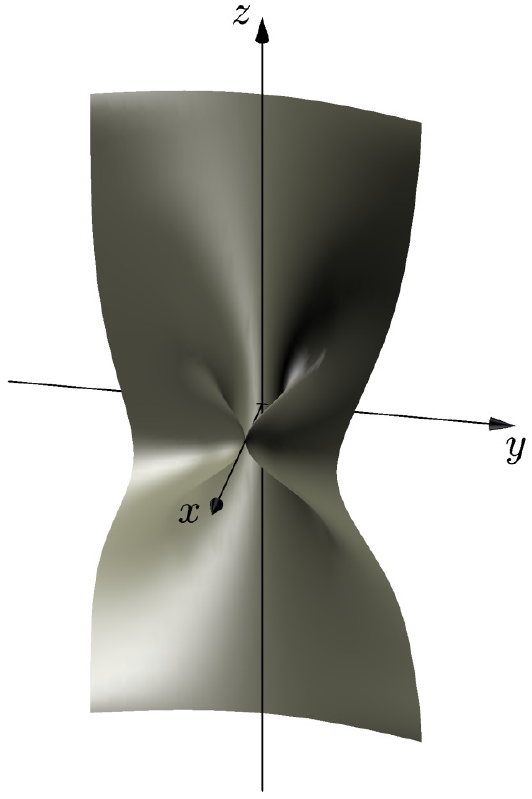}
		\caption{The singular surface $\Sp(2)_1^0$. The representation is in Cartesian coordinates $(x, y, z) = (r \cos\theta, r\sin\theta, z)$.} \label{fig:Sp(2)10}
	\end{center}
	\end{figure}
	
	\begin{figure}[tb]
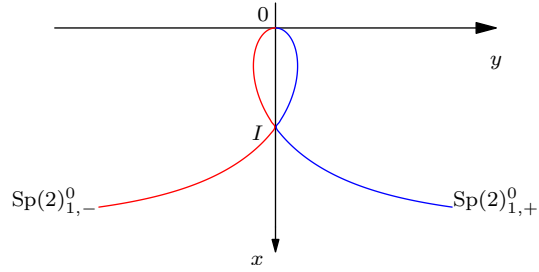

		\centering
		\begin{asy}
			import graph;
		
			size(200, 100, IgnoreAspect);
			
		//	Curva di destra
			real x1(real r) {return r*sqrt(1 - 4*r^2/(1 + r^2)^2);}
			real y1(real r) {return -2*r^2/(1 + r^2);}
		
		//	Curva di sinistra
			real x2(real r) {return -r*sqrt(1 - 4*r^2/(1 + r^2)^2);}
			real y2(real r) {return -2*r^2/(1 + r^2);}					
		
		//	Assi coordinati
			xaxis(xmin = -3, xmax = 3, arrow=EndArrow);
			yaxis(ymin = -2, ymax=0.25);

			draw(graph(x1, y1, 0, 3, operator ..), blue);
			draw(graph(x2, y2, 0, 3, operator ..), red);
			draw((0,-2)--(0,-2.25), arrow = EndArrow);
		
		//	Etichette
			labelx("$\scriptstyle y$", 3);
			label("$\scriptstyle 0$", (0, 0), NW);
			labely("$\scriptstyle I$", -1);
			labely("$\scriptstyle x$", -2.25);
			label("$\scriptstyle \mathrm{Sp}(2)_{1,+}^0$", (3,-1.75));
			label("$\scriptstyle \mathrm{Sp}(2)_{1,-}^0$", (-3,-1.75));
		\end{asy}
		\caption{Intersection of $\Sp(2)_1^0$ with the plane $z = 0$. The representation is in Cartesian coordinates $(x, y) = (r \cos\theta, r\sin\theta)$.} \label{fig:curva}
	\end{figure}
	
	The eigenvalues of a symplectic matrix $M$ written as in \eqref{eq:sympldec} are
		\[
			\lambda_\pm \= \frac{1}{2r} \Bigl[ (1 + r^2 + z^2) \cos \theta \pm \sqrt{(1 + r^2 + z^2)^2 \cos^2 \theta - 4r^2} \Bigr].
		\]
	For $\omega \= e^{i\varphi} \in \U$ we get
		\[
			\begin{split}
				D_\omega(M) & \= (-1)^{n - 1} \omega^{-n} \det(M - \omega I) \bigr|_{n = 1} \\
							& = e^{-i\varphi} \det(M - e^{i\varphi} I) \\
							& = 2 \cos \varphi - \left( r + \frac{1 + z^2}{r} \right) \cos \theta
			\end{split}
		\]
	and define
		\begin{align*}
			\Sp(2)_\omega^\pm & \= \Set{ (r, \theta, z) \in (0, +\infty) \times [0, 2\pi) \times \R | \pm(1 + r^2 + z^2) \cos \theta > 2r \cos \varphi }, \\
			\Sp(2)_\omega^0 & \= \Set{ (r, \theta, z) \in (0, +\infty) \times [0, 2\pi) \times \R | \pm(1 + r^2 + z^2) \cos \theta = 2r \cos \varphi }.
		\end{align*}
	The set $\Sp(2)_\omega^+ \cup \Sp(2)_\omega^-$ is named the \emph{$\omega$-regular part} of $\Sp(2)$, while $\Sp(2)_\omega^0$ is its \emph{$\omega$-singular part}; the former
	corresponds to the subset of $2 \times 2$ symplectic matrices which do not have $\omega$ as an eigenvalue, whereas those matrices admitting $\omega$ in their spectrum belong to the
	latter.
	
	We are particularly interested in $\Sp(2)_1^0$, the singular part of $\Sp(2)$ associated with the eigenvalue $1$, %\ie in those symplectic matrices which have $1$ as an eigenvalue.
	a representation of which is depicted in Figure~\ref{fig:Sp(2)10}. The ``pinched'' point is the identity matrix, and it is the only element satisfying $\dim\ker(M - I) = 2$.
	If we denote by
		\[
			\Sp(2)_{\omega, \pm}^0 \= \Set{ (r, \theta, z) \in \Sp(2)_\omega^0 | \pm \sin \theta > 0 },
		\]
	we see that $\Sp(2)_1^0 \setminus \{ I \} = \Sp(2)_{1, +}^0 \cup \Sp(2)_{1, -}^0$, and each subset is a path-connected component diffeomorphic to $\R^2 \setminus \{ 0 \}$.
	
	The \emph{stratum homotopy property} of the Maslov index states that the Maslov index of a path does not change if to that path is applied a homotopy that maintains each endpoint in its
	original stratum. Thanks to this property we can simplify the visualisation of paths involving $\Sp(2)_1^0$ by considering only their deformation (in the sense just described) onto the
	intersection of the surface with the plane $z = 0$ (which is the curve represented in Figure~\ref{fig:curva}).

		%\red
		{
		\section{Morse index of Fredholm quadratic forms}\label{sec:Fredholmforms}
		In this section we recall the definition of Morse index of Fredholm 
		quadratic forms acting on a (real) separable Hilbert space (for further details see \cite{PW2014}).
		Let $\big(H, \langle \cdot, \cdot \rangle\big)$ be a real separable Hilbert space.
		As usual we denote by $\mathscr L(H)$ the Banach space of all bounded linear 
		operators on $H$ and by $\mathscr F(H)\subset \mathscr L(H)$ the subspace 
		consisting of all (bounded) Fredholm operators. An operator in $\mathscr L(H)$ 
		defined on all of $H$ is self-adjoint if and only if it is symmetric. 
		We denote by $\mathscr F^s(H)$ the subspace of all (bounded) self-adjoint Fredholm 
		operators. 
		For $T \in \mathscr F^s(H)$, if $0$ belongs to the spectrum $\sigma(T)$, then (being 
		$T$ Fredholm) $0$ is an isolated point of $\sigma(T)$ and therefore it follows from 
		the Spectral Decomposition Theorem that there is an orthogonal 
		decomposition of $H$,
		\[
		 H= E_-(T)\oplus \ker T \oplus E_+(T),
		\]
		that reduces the operator $T$ and has the property that 
		\[
		 \sigma(T)\cap (-\infty,0)= \sigma(T|_{E_-(T)}) \qquad \text{ and }  \qquad
		 \sigma(T)\cap (0, +\infty)= \sigma(T|_{E_+(T)}).
		\]
		If $\dim E_-(T)< +\infty$, then $T$ is called \emph{essentially positive} and if it 
		is also an isomorphism its Morse index $\iMor(T)$ is defined as 
		\[
		 \iMor(T)\= \dim E_-(T).
		\]
		Let us consider a bounded quadratic form $q\colon H\to\R$ and we let 
		$b= b_{q}\colon H\times H\to\R$  be the bounded symmetric bilinear form 
		such that 
		\[
		q(u)=b(u,u),\qquad \forall\, u\in H.   
		 \]
		By the Riesz Representation Theorem
		there exists a bounded self-adjoint operator $A_q:H\rightarrow H$  
		such that $b_q(u,v)=\langle A_qu,v\rangle$, $u,v\in H$. 
		\begin{defn}
		We call $q\colon H\to \R$ a  \emph{Fredholm quadratic form} 
		if $A_q$ is Fredholm; \ie $\ker A_q$ is finite-dimensional and $\mathrm{Ran}\ A_q$ is closed.
		\end{defn}
		Recall that the space 
		$Q(H)$ of bounded quadratic  forms is a Banach space with respect to the norm  
		\[
			\norm{q} \= \sup_{\norm{u}=1} \abs{q(u)}.
		\]
		The subset $Q_F(H)$ of all Fredholm quadratic forms is an open subset of $Q(H)$ 
		which is stable under perturbations by weakly continuous quadratic forms. 
		A quadratic form $q\in Q_F(H)$ is called \emph{ non-degenerate} 
		if the corresponding Riesz representation $A_q$ is invertible. 
		\begin{rmk}\label{rmk:anyscalar}
		 It is worth noting that if the representation of a quadratic form on $H$ is either 
		 invertible, Fredholm or compact then so is its representation with respect to any 
		 other Hilbert product on the (real) vector space $H$.
		\end{rmk}
		\begin{prop}\label{thm:weakly-compact}
		 A quadratic form on the Hilbert space $H$ is weakly continuous if and only if one 
		 (and hence any by Remark \ref{rmk:anyscalar}) of its representations is a 
		 compact (self-adjoint) operator in $\mathscr L(H)$.
		\end{prop}
		\begin{proof} Recall that $K$ is compact if and only if it maps weakly convergent sequences  to 
		strongly convergent sequences. 
		We prove $(\Leftarrow)$. Suppose that $K$ is compact and let $(u_n)$ be a sequence in 
		$H$ such that $u_n\stackrel{w}{\rightharpoonup} u_0$. Then $(Ku_n)$ strongly converges 
		to $K u_0$. Thus we getting
		\[\lim_{n \to +\infty} q(u_n)= \lim_{n \to +\infty}\langle Ku_n, u_n\rangle = 
		 \langle Ku_0,u_0 \rangle =q(u_0),
		\]
		so the quadratic form is weakly sequentially continuous (and hence weakly continuous 
		because $H$ is first-countable).
		Now suppose that $q$ is weakly sequentially continuous. By the polarisation identity
		applied to the bilinear form $(u,v)\mapsto \langle Ku, v\rangle$ with $v=KU$ we 
		get 
		\begin{equation}\label{eq:polarization}
		 \langle Ku, Ku\rangle =\dfrac{1}{4}\Big[ \big\langle K(u+Ku), u+Ku \big\rangle - 
		 \big\langle K(u-Ku), u-Ku \big\rangle \Big]\text{ for all } u \in H.
		\end{equation}
		Let us assume that $(u_n)\subset H$ weakly converges to $u_0$. Since 
		$K \in \mathscr L(H)$ then $Ku_n\stackrel{w}{\rightharpoonup} Ku_0$. Thus
		$(u_n \pm Ku_n)$ weakly converges to $u_0\pm Ku_0$. Therefore by the 
		weak sequential continuity of $q$ and by the identity \eqref{eq:polarization} applied 
		to $u=u_n$ and $u=u_0$ we get
		\[
		\lim_{n \to +\infty}\norm{Ku_n}^2= \norm{Ku_0}^2.
		\]
		Since $(Ku_n)$ converges to $Ku_0$ weakly and in norm, it follows that it converges 
		pointwise to $Ku_0$ (strongly) in $H$. Thus $K$ is compact and this conclude the proof. 
		\end{proof}
		From this proposition we immediately get that Fredholm quadratic forms remain Fredholm 
		under perturbations by weakly continuous quadratic forms (since by definition 
		a Fredholm operator is the pre-image of the invertibles of the Calkin algebra under the 
		projection on the quotient) and that any Fredholm 
		quadratic form is weakly continuous perturbation of a non-degenerate Fredholm quadratic 
		form. 
		\begin{defn}\label{def:essentially+quadratic}
		 A Fredholm quadratic form $q:H \to \R$ is said \emph{essentially positive} if 
		 it is the perturbation of a positive definite Fredholm 
		 quadratic form by a weakly continuous quadratic form. 
		\end{defn}
		By this discussion it follows that
		\begin{prop}
		 A Fredholm quadratic form $q$ is essentially positive if and only if it is represented by an 
		 essentially positive self-adjoint Fredholm operator $A_q$. 
		\end{prop}
		\begin{proof}
		 By the Riesz representation theorem there exists a bounded self-adjoint Fredholm operator 
		 $A_q:H \to H$ such that $b_q(u,v)=\langle A_q u,v \rangle$ for all $u,v \in H$. Now since 
		 a bounded self-ajoint Fredholm operator is essentially positive if and only if it is a 
		 self-adjoint compact perturbation of a self-adjoint positive definite (and hence 
		 Fredholm, being invertible) operator, the conclusion follows by applying Proposition~\ref{thm:weakly-compact}.
		\end{proof}
		\begin{defn}\label{def:Morseindexform}
		 The \emph{Morse index of an essentially positive Fredholm quadratic form} $q:H \to \R$ 
		 is the Morse index of the (self-adjoint) bounded Fredholm operator $A_q: H \to H$ 
		 uniquely determined by the Riesz Representation Theorem, \ie
		 \[
		  b_q(u,v)=\langle A_q u, v\rangle \text{ for all } u,v \in H
		 \]
		 where $b_q$ is the bounded symmetric form induced by $q$ through the polarisation 
		 identity. 
		\end{defn}
		\begin{rmk}
		 It is worth noting that it is possible to show that the \emph{Morse index} of an 
		 essentially positive Fredholm quadratic form depends only on the quadratic form and 
		 not on the Hilbert structure on $H$. 
		\end{rmk}
	}

% % % % % % % % % % % % % % % % % % % % % % % % % % % % % % % % % % % % % % % % % % % % % % % % % % % % % % % 

%	~~~~~~~~~~~~~~~~~~~~~~~~~~~~~~~~~~~~~~~~~~~~~~~~~~~~~~~~~~~~~~~~~~~~~~~~~
%	~~~~~~~~~~~~~~~~~~~~~~~~~~~~~~~~~~~~~~~~~~~~~~~~~~~~~~~~~~~~~~~~~~~~~~~~~

%	Bibliografia con BibLaTeX:
	
	\addcontentsline{toc}{section}{\refname}	% per includere la voce "References" nell'indice
%	\nocite{*}							% per includere nella bibliografia anche gli elementi non citati esplicitamente
	\bibliographystyle{amsalpha}			% stile dell'AMS, per etichette del tipo [Mor07]
	\bibliography{MyDatabase}			% per includere il file con la bibliografia

\providecommand{\bysame}{\leavevmode\hbox to3em{\hrulefill}\thinspace}
\providecommand{\MR}{\relax\ifhmode\unskip\space\fi MR }
\providecommand{\MRhref}[2]{%
  \href{http://www.ams.org/mathscinet-getitem?mr=#1}{#2}
}
\providecommand{\href}[2]{#2}
\begin{thebibliography}{MPP05}

\bibitem[Abb01]{MR1824111}
Alberto Abbondandolo, \emph{Morse theory for {H}amiltonian systems}, Chapman \&
  Hall/CRC Research Notes in Mathematics, vol. 425, Chapman \& Hall/CRC, Boca
  Raton, FL, 2001. \MR{1824111 (2002e:37103)}

\bibitem[AF07]{MR2300670}
Alberto Abbondandolo and Alessio Figalli, \emph{High action orbits for
  {T}onelli {L}agrangians and superlinear {H}amiltonians on compact
  configuration spaces}, J. Differential Equations \textbf{234} (2007), no.~2,
  626--653. \MR{2300670 (2008f:37128)}

\bibitem[APS08]{MR2465553}
Alberto Abbondandolo, Alessandro Portaluri, and Matthias Schwarz, \emph{The
  homology of path spaces and {F}loer homology with conormal boundary
  conditions}, J. Fixed Point Theory Appl. \textbf{4} (2008), no.~2, 263--293.
  \MR{2465553 (2009i:53090)}

\bibitem[Arn67]{MR0211415}
V.~I. Arnol'd, \emph{On a characteristic class entering into conditions of
  quantization}, Funkcional. Anal. i Prilo\v zen. \textbf{1} (1967), 1--14.
  \MR{0211415 (35 \#2296)}

\bibitem[BJP14]{MR3227283}
Vivina~L. Barutello, Riccardo~D. Jadanza, and Alessandro Portaluri,
  \emph{Linear instability of relative equilibria for {$n$}-body problems in
  the plane}, J. Differential Equations \textbf{257} (2014), no.~6, 1773--1813.
  \MR{3227283}

\bibitem[Bot56]{MR0090730}
Raoul Bott, \emph{On the iteration of closed geodesics and the {S}turm
  intersection theory}, Comm. Pure Appl. Math. \textbf{9} (1956), 171--206.
  \MR{0090730 (19,859f)}

\bibitem[BT04]{MR2097664}
V.~Barutello and S.~Terracini, \emph{Action minimizing orbits in the {$n$}-body
  problem with simple choreography constraint}, Nonlinearity \textbf{17}
  (2004), no.~6, 2015--2039. \MR{2097664 (2005k:70029)}

\bibitem[CD98]{MR1642007}
Alain Chenciner and Nicole Desolneux, \emph{Minima de l'int{\'e}grale d'action
  et {\'e}quilibres relatifs de {$n$} corps}, C. R. Acad. Sci. Paris S{\'e}r. I
  Math. \textbf{326} (1998), no.~10, 1209--1212. \MR{1642007 (2000a:70014a)}

\bibitem[CLM94]{MR1263126}
Sylvain~E. Cappell, Ronnie Lee, and Edward~Y. Miller, \emph{On the {M}aslov
  index}, Comm. Pure Appl. Math. \textbf{47} (1994), no.~2, 121--186.
  \MR{1263126 (95f:57045)}

\bibitem[CZ84]{MR733717}
Charles Conley and Eduard Zehnder, \emph{Morse-type index theory for flows and
  periodic solutions for {H}amiltonian equations}, Comm. Pure Appl. Math.
  \textbf{37} (1984), no.~2, 207--253. \MR{733717 (86b:58021)}

\bibitem[Fat08]{Fat08}
Albert Fathi, \emph{Weak {KAM} theorem in {L}agrangian dynamics}, Version 15,
  June 2008.

\bibitem[GL55]{MR0073767}
I.~M. Gel'fand and V.~B. Lidski{\u\i}, \emph{On the structure of the regions of
  stability of linear canonical systems of differential equations with periodic
  coefficients}, Uspehi Mat. Nauk (N.S.) \textbf{10} (1955), no.~1(63), 3--40.
  \MR{0073767 (17,482g)}

\bibitem[GL58]{MR0091390}
\bysame, \emph{On the structure of the regions of stability of linear canonical
  systems of differential equations with periodic coefficients}, Amer. Math.
  Soc. Transl. (2) \textbf{8} (1958), 143--181. \MR{0091390 (19,960b)}

\bibitem[Gor77]{MR0502484}
W.~B. Gordon, \emph{A minimizing property of {K}eplerian orbits}, Amer. J.
  Math. \textbf{99} (1977), no.~5, 961--971. \MR{0502484 (58 \#19497)}

\bibitem[GPP04]{MR2057171}
Roberto Giamb{{\`o}}, Paolo Piccione, and Alessandro Portaluri,
  \emph{Computation of the {M}aslov index and the spectral flow via partial
  signatures}, C. R. Math. Acad. Sci. Paris \textbf{338} (2004), no.~5,
  397--402. \MR{2057171 (2004k:53128)}

\bibitem[GPS80]{Goldstein}
Herbert Goldstein, Charles Poole, and John Safko, \emph{Classical dynamics},
  Addison-Wesley, Canada, 1980.

\bibitem[HLS14]{MR3218836}
Xijun Hu, Yiming Long, and Shanzhong Sun, \emph{Linear stability of elliptic
  {L}agrangian solutions of the planar three-body problem via index theory},
  Arch. Ration. Mech. Anal. \textbf{213} (2014), no.~3, 993--1045. \MR{3218836}

\bibitem[HS09]{MR2525637}
Xijun Hu and Shanzhong Sun, \emph{Index and stability of symmetric periodic
  orbits in {H}amiltonian systems with application to figure-eight orbit},
  Comm. Math. Phys. \textbf{290} (2009), no.~2, 737--777. \MR{2525637
  (2010h:37140)}

\bibitem[HS10]{MR2563212}
\bysame, \emph{Morse index and stability of elliptic {L}agrangian solutions in
  the planar three-body problem}, Adv. Math. \textbf{223} (2010), no.~1,
  98--119. \MR{2563212 (2011e:37118)}

\bibitem[HS11]{MR2817146}
\bysame, \emph{Variational principle and linear stability of periodic orbits in
  celestial mechanics}, Progress in variational methods, Nankai Ser. Pure Appl.
  Math. Theoret. Phys., vol.~7, World Sci. Publ., Hackensack, NJ, 2011,
  pp.~40--51. \MR{2817146 (2012f:70029)}

\bibitem[Lag72]{Lagrange:3corps}
Joseph-Louis Lagrange, \emph{{Essai sur le probl{\`e}me des trois corps}}, Prix
  de l'Acad{\'e}mie Royale des Sciences de Paris \textbf{IX} (1772), 229--331.

\bibitem[Lon91]{long1991structure}
Yiming Long, \emph{The structure of the singular symplectic matrix set},
  Science in China, Series A --- Mathematics, Physics, Astronomy \&
  Technological Sciences \textbf{34} (1991), no.~8, 897--907.

\bibitem[Lon99]{MR1674313}
\bysame, \emph{Bott formula of the {M}aslov-type index theory}, Pacific J.
  Math. \textbf{187} (1999), no.~1, 113--149. \MR{1674313 (2000d:37073)}

\bibitem[Lon02]{MR1898560}
\bysame, \emph{Index theory for symplectic paths with applications}, Progress
  in Mathematics, vol. 207, Birkh{\"a}user Verlag, Basel, 2002. \MR{1898560
  (2003d:37091)}

\bibitem[LZ90]{MR1124230}
Yiming Long and Eduard Zehnder, \emph{Morse-theory for forced oscillations of
  asymptotically linear {H}amiltonian systems}, Stochastic processes, physics
  and geometry ({A}scona and {L}ocarno, 1988), World Sci. Publ., Teaneck, NJ,
  1990, pp.~528--563. \MR{1124230 (92j:58019)}

\bibitem[LZ00]{MR1762278}
Yiming Long and Chaofeng Zhu, \emph{Maslov-type index theory for symplectic
  paths and spectral flow. {II}}, Chinese Ann. Math. Ser. B \textbf{21} (2000),
  no.~1, 89--108. \MR{1762278 (2001i:58049)}

\bibitem[Moe94]{Moeckel:notes}
Richard Moeckel, \emph{{Celestial Mechanics (especially central
  configurations)}}, Unpublished lecture notes, available at
  \href{http://www.math.umn.edu/~rmoeckel/notes/Notes.html}{\textsf{http://www.math.umn.edu/$\sim$rmoeckel/notes/Notes.html}},
  October 1994.

\bibitem[MPP05]{MR2133393}
Monica Musso, Jacobo Pejsachowicz, and Alessandro Portaluri, \emph{A {M}orse
  index theorem for perturbed geodesics on semi-{R}iemannian manifolds}, Topol.
  Methods Nonlinear Anal. \textbf{25} (2005), no.~1, 69--99. \MR{2133393
  (2006d:58013)}

\bibitem[MS05]{MR2145251}
Kenneth~R. Meyer and Dieter~S. Schmidt, \emph{Elliptic relative equilibria in
  the {$N$}-body problem}, J. Differential Equations \textbf{214} (2005),
  no.~2, 256--298. \MR{2145251 (2006b:70018)}

\bibitem[Por08]{MR2383373}
Alessandro Portaluri, \emph{Maslov index for {H}amiltonian systems}, Electron.
  J. Differential Equations (2008), No. 09, 10. \MR{2383373 (2009a:53144)}

\bibitem[Por10]{MR2574386}
\bysame, \emph{On a generalized {S}turm theorem}, Adv. Nonlinear Stud.
  \textbf{10} (2010), no.~1, 219--230. \MR{2574386 (2011g:34056)}

\bibitem[PPT04]{MR2046769}
Paolo Piccione, Alessandro Portaluri, and Daniel~V. Tausk, \emph{Spectral flow,
  {M}aslov index and bifurcation of semi-{R}iemannian geodesics}, Ann. Global
  Anal. Geom. \textbf{25} (2004), no.~2, 121--149. \MR{2046769 (2005b:53128)}

\bibitem[PW14]{PW2014}
Alessandro Portaluri and Nils Waterstraat, \emph{A {M}orse-{S}male index
  theorem for indefinite elliptic systems and bifurcation}, arXiv:1408.1419
  (2014).

\bibitem[RS93]{MR1241874}
Joel Robbin and Dietmar Salamon, \emph{The {M}aslov index for paths}, Topology
  \textbf{32} (1993), no.~4, 827--844. \MR{1241874 (94i:58071)}

\bibitem[RT95]{MR1329405}
Miguel Ramos and Susanna Terracini, \emph{Noncollision periodic solutions to
  some singular dynamical systems with very weak forces}, J. Differential
  Equations \textbf{118} (1995), no.~1, 121--152. \MR{1329405 (96d:58115)}

\bibitem[Ven01]{MR1841900}
Andrea Venturelli, \emph{Une caract{\'e}risation variationnelle des solutions
  de {L}agrange du probl{\`e}me plan des trois corps}, C. R. Acad. Sci. Paris
  S{\'e}r. I Math. \textbf{332} (2001), no.~7, 641--644. \MR{1841900
  (2002h:70021)}

\bibitem[Ven02]{Venturelli:phd}
\bysame, \emph{{Application de la minimisation de l'action au Probl{\`e}me des
  $N$ corps dans le plan et dans l'espace}}, Ph.D. thesis, Universit{\'e} Paris
  7 (D.~Diderot), 2002.

\bibitem[ZZ01]{MR1852963}
Shi~Qing Zhang and Qing Zhou, \emph{A minimizing property of {L}agrangian
  solutions}, Acta Math. Sin. (Engl. Ser.) \textbf{17} (2001), no.~3, 497--500.
  \MR{1852963 (2002f:70018)}

\end{thebibliography}
% % \bibliography{BarJadPor13_ref_2}

%	\newpage
	\vfill
	
	\vspace{1cm}
	\noindent
	\textsc{Vivina L.~Barutello}\\
	Dipartimento di Matematica \lq\lq G.~Peano\rq\rq\\
	Universit\`a degli Studi di Torino\\
	Via Carlo Alberto, 10 \\
	10123 Torino \\
	Italy\\
	E-mail: \email{vivina.barutello@unito.it}

	\vspace{1cm}
	\noindent
	\textsc{Riccardo Danilo Jadanza}\\
	Dipartimento di Scienze Matematiche \lq\lq J.-L.~Lagrange\rq\rq\ (DISMA)\\
	Politecnico di Torino\\
	Corso Duca degli Abruzzi, 24\\
	10129 Torino\\
	Italy\\
	E-mail: \email{riccardo.jadanza@polito.it}

	\vspace{1cm}
	\noindent
	\textsc{Alessandro Portaluri}\\
	Dipartimento di Scienze Agrarie, Forestali e Alimentari (DISAFA)\\
	Universit\`a degli Studi di Torino\\
	Largo Paolo Braccini, 2\\
	10095 Grugliasco (TO)\\
	Italy\\
	E-mail: \email{alessandro.portaluri@unito.it}\\
	Website: \href{http://aportaluri.wordpress.com}{\textsf{http://aportaluri.wordpress.com}}

\end{document}